\documentclass[11pt]{amsart}
\usepackage{amssymb,amsmath}
\usepackage{enumerate}
\usepackage{mathrsfs}
\usepackage[alphabetic]{amsrefs}
\usepackage{graphicx,color}
\usepackage[utf8]{inputenc}
\usepackage{amsthm}
\usepackage{latexsym}
\usepackage{slashed}
\usepackage[all]{xy}
\usepackage{hyperref}
\usepackage[mathscr]{eucal}
\usepackage{tikz}
\usepackage{mathtools}
\usepackage{accents}
\usepackage{xcolor}
\usepackage{fancybox}
\usepackage{bbm}
\usepackage[normalem]{ulem}

\usepackage{fullpage}

\def\theequation{\thesection.\arabic{equation}}
\makeatletter
\@addtoreset{equation}{section}
\makeatother

\makeatletter
\newcommand{\eqnum}{\refstepcounter{equation}\textup{\tagform@{\theequation}}}
\makeatother

\newcounter{copy}
\makeatletter
\renewcommand{\thecopy}{\ifnum0=\c@section\arabic{copy}\else\thesection.\arabic{copy}'\fi}
\makeatother

\BibSpec{innerbook}{%
    +{}{\PrintAuthors}          {author}
    +{,}{ \textit}              {title}
    +{.}{ }                     {part}
    +{:}{ }                     {subtitle} 
    +{. In:}{ }                 {booktitle}
    +{,}{ \PrintEdition}        {edition}
    +{}{ \PrintEditorsB}        {editor}
    +{,}{ \PrintTranslatorsC}   {translator}
    +{,}{ \PrintContributions}  {contribution}
    +{,}{ }                     {series} 
    +{}{ \voltext}              {volume}
    +{,}{ }                     {publisher}
    +{,}{ }                     {organization}
    +{,}{ }                     {address}
    +{,}{ \PrintDateB}          {date}
    +{.}{ }                     {note}
}

\usepackage{cleveref}
\theoremstyle{definition}
\newtheorem{df}[equation]{Definition}
\newtheorem{notn}[equation]{Notation}

\theoremstyle{plain}
\newtheorem{thm}[equation]{Theorem}
\newtheorem{prop}[equation]{Proposition}
\newtheorem{lem}[equation]{Lemma}
\newtheorem{cor}[equation]{Corollary}

\theoremstyle{remark}
\newtheorem{rem}[equation]{Remark}
\newtheorem{ex}[equation]{Example}

\crefname{df}{Definition}{Definitions}
\crefname{notn}{Notation}{Notations}
\crefname{thm}{Theorem}{Theorems}
\crefname{prop}{Proposition}{Propositions}
\crefname{lem}{Lemma}{Lemmas}
\crefname{cor}{Corollary}{Corollaries}
\crefname{conj}{Conjecture}{Conjectures}
\crefname{rem}{Remark}{Remarks}
\crefname{ex}{Example}{Examples}
\crefname{section}{Section}{Sections}
\crefname{subsection}{Subsection}{Subsections}
\crefname{para}{}{}
\crefname{appendix}{Appendix}{Appendices}
\crefname{subappendix}{Appendix}{Appendices}
\crefname{table}{Table}{Tables}

\usepackage{framed}
\usepackage{changepage}
\definecolor{shade}{gray}{0.94}

\newcommand{\bC}{\mathbb{C}}

\newcommand{\bM}{\mathbb{M}}
\newcommand{\bN}{\mathbb{N}}

\newcommand{\bR}{\mathbb{R}}

\newcommand{\bT}{\mathbb{T}}
\newcommand{\bU}{\mathbb{U}}

\newcommand{\bZ}{\mathbb{Z}}

\newcommand{\bfE}{\mathbf{E}}

\newcommand{\bfK}{\mathbf{K}}

\newcommand{\sfA}{\mathsf{A}}

\newcommand{\sfE}{\mathsf{E}}

\newcommand{\sfH}{\mathsf{H}}

\newcommand{\sfK}{\mathsf{K}}

\newcommand{\sfM}{\mathsf{M}}

\newcommand{\bbmv}{\mathbbm{v}}
\newcommand{\bbmw}{\mathbbm{w}}
\newcommand{\bbmx}{\mathbbm{x}}

\newcommand{\cA}{\mathcal{A}}
\newcommand{\cB}{\mathcal{B}}
\newcommand{\cC}{\mathcal{C}}
\newcommand{\cD}{\mathcal{D}}
\newcommand{\cE}{\mathcal{E}}
\newcommand{\cF}{\mathcal{F}}
\newcommand{\cG}{\mathcal{G}}
\newcommand{\cH}{\mathcal{H}}
\newcommand{\cI}{\mathcal{I}}
\newcommand{\cJ}{\mathcal{J}}
\newcommand{\cK}{\mathcal{K}}
\newcommand{\cL}{\mathcal{L}}
\newcommand{\cM}{\mathcal{M}}

\newcommand{\cO}{\mathcal{O}}

\newcommand{\cS}{\mathcal{S}}
\newcommand{\cT}{\mathcal{T}}
\newcommand{\cU}{\mathcal{U}}

\newcommand{\cX}{\mathcal{X}}
\newcommand{\cY}{\mathcal{Y}}
\newcommand{\cZ}{\mathcal{Z}}
\newcommand{\fA}{\mathfrak{A}}

\newcommand{\fC}{\mathfrak{C}}
\newcommand{\fD}{\mathfrak{D}}

\newcommand{\fT}{\mathfrak{T}}
\newcommand{\fU}{\mathfrak{U}}

\newcommand{\fu}{\mathfrak{u}}
\newcommand{\fv}{\mathfrak{v}}
\newcommand{\fw}{\mathfrak{w}}

\newcommand{\bv}{\mathbf{v}}
\newcommand{\bw}{\mathbf{w}}
\newcommand{\bx}{\mathbf{x}}

\newcommand{\sA}{\mathscr{A}}
\newcommand{\sB}{\mathscr{B}}
\newcommand{\sC}{\mathscr{C}}
\newcommand{\sD}{\mathscr{D}}

\newcommand{\sF}{\mathscr{F}}

\newcommand{\sH}{\mathscr{H}}
\newcommand{\sI}{\mathscr{I}}
\newcommand{\sJ}{\mathscr{J}}

\newcommand{\sM}{\mathscr{M}}

\newcommand{\sS}{\mathscr{S}}

\newcommand{\mrq}{\mathrm{q}}

\newcommand{\rF}{\mathrm{F}}

\newcommand{\rQ}{\mathrm{Q}}

\newcommand{\op}{\mathrm{op}}

\newcommand{\blank}{\text{\textvisiblespace}}
\renewcommand{\Im}{\mathrm{Im} \hspace{0.1em}}
\newcommand{\pt}{\mathrm{pt}}
\newcommand{\loc}{\mathrm{loc}}
\newcommand{\id}{\mathrm{id}}

\newcommand{\Hilb}{\mathrm{Hilb}}
\newcommand{\Hilbgo}{\operatorname{Hilb^f_{\Gamma,\omega}}}
\newcommand{\Hilbg}{\operatorname{Hilb^f_{\Gamma}}}
\newcommand{\kC}{\mathrm{k}_{\mathscr{C}}}
\newcommand{\qC}{\mathrm{q}_{\mathscr{C}}}

\DeclareMathOperator{\hotimes}{\hat{\otimes}}
 
\DeclareMathOperator{\K}{\mathrm{K}}

\DeclareMathOperator{\KK}{\mathrm{KK}}

\DeclareMathOperator{\iEnd}{\underline{\mathrm{End}}}
\DeclareMathOperator{\iiEnd}{\underline{END}}
\DeclareMathOperator{\Hom}{\mathrm{Hom}}
\DeclareMathOperator{\Der}{\mathrm{Der}}

\DeclareMathOperator{\cspan}{\overline{\mathrm{span}}}
\DeclareMathOperator{\Ad}{\mathrm{Ad}}
\DeclareMathOperator{\ad}{\mathrm{ad}}
\DeclareMathOperator{\Aut}{\mathrm{Aut}}
\DeclareMathOperator{\diag}{\mathrm{diag}}
\DeclareMathOperator{\Ind}{Ind}
\DeclareMathOperator{\Res}{Res}
\DeclareMathOperator{\Mod}{\mathrm{Mod}}
\DeclareMathOperator{\Bimod}{\mathrm{Bimod}}
\DeclareMathOperator{\sMod}{\mathscr{M}\mathrm{od}}
\DeclareMathOperator{\sBimod}{\mathscr{B}\mathrm{imod}}
\DeclareMathOperator{\Func}{\mathrm{Func}}
\DeclareMathOperator{\Rep}{\mathrm{Rep}}

\DeclareMathOperator{\qHom}{\mathrm{qHom}}
\newcommand{\Calg}{\mathfrak{C}^*\mathchar`-\mathfrak{alg}}
\newcommand{\Ctriv}{\mathfrak{C}^*\mathchar`-\mathfrak{alg}_{\mathrm{triv}}}
\newcommand{\Ccat}{\mathfrak{C}^*\mathchar`-\mathfrak{cat}}
\newcommand{\Corr}{\mathfrak{Corr}}
\newcommand{\eval}{\mathrm{ev}}

\newcommand{\evst}{R}
\newcommand{\ev}{R^*}
\newcommand{\coev}{\overline{R}}
\newcommand{\ass}{\Phi}

\newcommand{\cone}{\mathsf{C}}

\newcommand{\Kas}{\mathfrak{KK}}
\newcommand{\pr}{\mathrm{pr}}
\newcommand{\alg}{\mathrm{alg}}

\DeclareMathOperator{\Obj}{\mathrm{Obj}}
\DeclareMathOperator{\Irr}{\mathrm{Irr}}

\title{Tensor category equivariant KK-theory}
\author[Y. Arano]{Yuki Arano}
\address{Graduate School of Mathematics, Nagoya University
Furocho, Chikusaku, Nagoya, 464-8602, Japan}
\email{y.arano@math.nagoya-u.ac.jp}
\author[K. Kitamura]{Kan Kitamura}
\address{RIKEN iTHEMS, 2-1 Hirosawa, Wako, Saitama 351-0198, Japan}
\email{kan.kitamura@riken.jp}
\author[Y. Kubota]{Yosuke Kubota}
\address{Graduate School of Science, Kyoto University, Kitashirakawa Oiwake-cho, Sakyo-ku, Kyoto 606-8502, Japan}
\address{RIKEN iTHEMS, 2-1 Hirosawa, Wako, Saitama 351-0198, Japan}
\email{ykubota@math.kyoto-u.ac.jp}
\date{}
\begin{document}
\maketitle
\begin{abstract}
	In this paper, we introduce Kasparov's bivariant K-theory that is equivariant under symmetries of a C*-tensor category. 
	It is motivated by some dualities in quantum group equivariant KK-theory, and the classification theory of inclusions of C*-algebras. 
	The fundamental properties of the KK-theory, i.e., the existence of the Kasparov product, Cuntz's picture, universality, and triangulated category structure, hold in this generalization as well.
	Moreover, we further prove a new property specific to this theory; the invariance of KK-theory under weak Morita equivalence of the tensor categories. 
 As an example, we study the Baum--Connes type property for $3$-cocycle twists of discrete groups. 
	\end{abstract}
	\tableofcontents 

	\section{Introduction}
	This paper aims to develop the tensor category equivariant version of Kasparov's KK-theory \cite{kasparovOperatorFunctorExtensions1980}, as a generalization of discrete (quantum) group equivariant KK-theory \cites{kasparovEquivariantKKTheory1988,baajCalgebresHopfTheorie1989}.
    We start by introducing two motivations for the present work.

    The first is to understand compact quantum group equivariant KK-theory with the internal language of its representation category, following the line of  Tannaka--Krein duality. 
    A prior work by Voigt \cite{voigtBaumConnesConjectureFree2011} leads us to this problem. 
    In his study of the Baum--Connes conjecture for free orthogonal quantum groups, it is proved as a key ingredient that two compact quantum groups have equivalent equivariant KK-theory if their representation categories are monoidally equivalent.
    This suggests a tensor-categorical view of the equivariant KK-theory, or more specifically, a description of it with the language of the representation category. 
    The developments in the theory of Tannaka--Krein duality in this decade \cites{decommerTannakaKreinDualityCompact2013,decommerTannakaKreinDualityCompact2015,neshveyevDualityTheoryNonergodic2014,neshveyevCategoricallyMoritaEquivalent2018} support this idea, in which the actions of quantum groups on C*-algebras are translated to the language of C*-categories. We are now ready to build KK-theory on this basis.

    The second is the classification of C*-algebra inclusions. 
	The classification theory of C*-algebras is expected to follow a similar path as the corresponding theory of von Neumann algebras, but with the addition of a topological distinction captured by K(K)-theory. Two notable successes in this regard are the celebrated Kirchberg--Phillips classification theorem \cite{phillipsClassificationTheoremNuclear2000} and the recent classification of group actions on Kirchberg algebras announced by Gabe and Szab\'{o} \cite{gabeDynamicalKirchbergphillipsTheorem2022}. 
	The next problem would be the classification of inclusions, which is known as subfactor theory initiated by Jones~\cite{jones1983index} and has been a great success in the von Neumann algebra side. 
    Classification of inclusions is understood to involve three steps; classification of rigid C*-tensor categories $\sC$, of actions of $\sC$, and of Q-systems in $\sC$ \cites{ocneanuquantized1988,popa1990classification,longo1994duality,popa1995axiomatization}. 
    Our KK-theory will serve as a `topological' counterpart of the second part.

    In the context of Tannaka--Krein duality, the notion of tensor category action onto a C*-algebra $A$ is described in terms of a $\sC$-module structure on the C*-category $\Mod(A)$. 
    Familiar examples of such actions include those of the category $\Hilb_\Gamma^{\rm f}$ of finite dimensional graded Hilbert spaces for a discrete group $\Gamma$, the representation category $\Rep(G)$ for a compact quantum group $G$, and the tensor category $\Hilb_{\Gamma,\omega}^{\rm f}$ obtained by twisting the associator of $\Hilb_{\Gamma}^{\rm f}$ by a $3$-cocycle $\omega \in Z^3(\Gamma ; \bT)$, each of which corresponds to an action of $\Gamma$, an action of $\hat{G}$, and an anomalous action of $\Gamma$ \cites{decommerTannakaKreinDualityCompact2013,neshveyevDualityTheoryNonergodic2014}, respectively.
    To ensure that our KK-theory has a rich cohomological nature, such as the Bott periodicity, we have to include nonunital C*-algebras (e.g. $A \otimes C_0(\bR)$) in the target.  
    This is why we work on nonunital C*-categories \cites{mitchenerCATEGORIES2002,antounBicolimitsCategories2020}. 
    The duality of a C*-algebra and the associated C*-category $\Mod(A)$ extends to this class, and a functor $\Mod(A) \to \Mod(B)$ corresponds to an \emph{essential} (non-degenerate) Hilbert $A$-$B$ bimodule. 

    Here we are faced with the problem that we need to deal with non-essential Kasparov bimodules to reduce a Kasparov bimodule to a $\ast$-homomorphism following Cuntz's argument in ordinary KK-theory \cites{cuntzNewLookKK1987,meyerEquivariantKasparovTheory2000}.
    For this reason, we translate the notion of tensor category action to the C*-algebra side. 
    Through the category-algebra duality, the functor $\blank \otimes_\alpha \pi$ tensoring $\pi \in \Obj \sC$ corresponds to a $A$-$A$ bimodule $\alpha_\pi$, besides the coherence map corresponding to a natural unitary $\fu_{\pi,\sigma} \colon \alpha_\pi \otimes \alpha_\sigma \to \alpha_{\pi \otimes \sigma}$ called a cocycle.
    Now, a $\sC$-Hilbert bimodule is a Hilbert $A$-$B$ bimodule $E$ together with a family of `1-cocycles' $\bbmv_\pi \colon \alpha_\pi \otimes_A E \to E \otimes_B \beta_\pi$, which serves a C*-algebraic counterpart of $\sC$-module functor in the C*-category side, but encompasses a larger class since we allow the left $A$-action on $E$ to be non-essential.
    The technical problem is that these 1-cocycles are now not necessarily unitary or even adjointable when $E$ is not essential, which leads to technical difficulties in the entire theory.

    We build the $\sC$-equivariant KK-theory on this basis. We first introduce the notion of $\sC$-Kasparov bimodule between ($\bZ_2$-graded) $\sC$-C*-algebras, and then define the KK-group as the set of their homotopy classes. 
    In the familiar cases of $\sC=\Hilb_\Gamma^{\rm f}$ or $\Rep(G)$, our KK-theory is the same thing as the (quantum) group equivariant KK-theory\cites{kasparovEquivariantKKTheory1988,baajCalgebresHopfTheorie1989}.
    The notion of Kasparov bimodule and equivariant KK-theory is reformulated later in terms of a $\sC$-module category as a pair of a $\sC$-module functor and family of odd Fredholm operators.
    A standard argument parallel to the prior works \cites{higsonTechnicalTheoremKasparov1987,kasparovEquivariantKKTheory1988,baajCalgebresHopfTheorie1989} shows that our KK-theory is equipped with the associative product operation called the Kasparov product.
    It makes a family of groups $\KK^\sC(A,B)$ form the morphism set of an additive category, which we call the $\sC$-equivariant Kasparov category $\Kas^\sC$. 

    We further study the categorical nature of $\KK^\sC$-theory, or more specifically, the universality (\cref{thm:Kasparov.universality}) and the triangulated category structure (\cref{thm:triangulated}) of the category $\Kas^\sC$. 
    This is done by generalizing Cuntz's quasihomomorphism picture \cites{cuntzNewLookKK1987} to our $\sC$-equivariant setting (\cref{thm:quasihom.replacement}), which enables us to replace a $\sC$-Kasparov bimodule with an equivariant $\ast$-homomorphism. 
    In the first step reducing a $\sC$-Kasparov bimodule to the one having a trivial Fredholm operator, we follow the idea of `equivariant stabilization' by Meyer \cite{meyerEquivariantKasparovTheory2000}. 
    The remaining part goes parallel to the non-equivariant case, except that the free product C*-algebra $\mathrm{Q} A = A \ast A$ and its ideal $\mathrm{q} A$ are replaced with the `free product over $\sC$' and its ideal, denoted by $\mathrm{Q}_\sC A$ and $\qC A$, in which a cocycle $\fu_{\pi,\sigma}$ of a copy of $A$ in $A \ast A$ is identified with the one of another.
    The triangulated structure of the Kasparov category allows Meyer's relative homological algebra \cite{meyerHomologicalAlgebraBivariant2008} to work. We expect this will be the basis for formulating the Baum--Connes conjecture in the future. 
    Neshveyev, Voigt, and Yamashita kindly told us that the tensor category equivariant KK-theory and its Baum-Connes conjecture have been studied by them at the same time \cite{NVY}.

    A remarkable new aspect of it is the \emph{weak Morita invariance}, that is, the categorical equivalence $\Kas^\sC \cong \Kas^\sD$ for weakly Morita equivalent C*-tensor categories $\sC$ and $\sD$. 
    This includes, as an example, the well-known Baaj--Skandalis--Takesaki--Takai duality of quantum group equivariant Kasparov categories \cite{baajCalgebresHopfTheorie1989}. 
    Furthermore, our weak Morita invariance includes a generalization of it, a partial version of the Takesaki--Takai duality. 
    For a discrete group $\Gamma$, its finite abelian normal subgroup $\Lambda$, and a $\Gamma$-C*-algebra $A$, the crossed product $\Lambda \ltimes A$ is equipped with cocycle actions of both $\Gamma/\Lambda$ and the Pontrjagin dual $\hat {\Lambda}$ independently. 
    In fact, these symmetries on $\Lambda \ltimes A$ are assembled to an action of $\hat{\Lambda} \rtimes \Gamma/\Lambda$ twisted by a $3$-cocycle arising from the extension. 
    This not only enables us to study a new class of equivariant Kasparov categories through an existing one, but also shows new categorical equivalences between existing Kasparov categories.
    
    As a new example, we study the equivariant $\KK$-theory of $(\Hilb_{\Gamma, \omega}^{\rm f})^{\rm rev}$, the $3$-cocycle twist of a discrete group, from the viewpoint of the Baum--Connes conjecture. 
    An action of this tensor category on a C*-algebra $A$ is identical to an anomalous action of $\Gamma$ \cite{jonesRemarksAnomalousSymmetries2021}, or a $(\Gamma, \omega)$-action, given by a family of automorphisms $\{ \alpha_g\}$ and $2$-cocycles $\{ \fu_{g,h} \}$ with the twisted cocycle conditions. 
    A $3$-cocycle $\omega$ can be untwisted if it is a $(\Gamma , \omega)$-C*-algebra over a locally compact $\Gamma$-space $X$
    and $p_X^*\omega =0 \in H^3(\Gamma ; C(X;\bT))$.   
    A cohomology computation shows that  
    this happens for any proper $(\Gamma , \omega )$-C*-algebra if $\omega$ is trivialized by the Bockstein map. 
    Consequently, if we further assume that $\Gamma$ has the Haagerup property \cite{higsonTheoryKKTheory2001}, 
    then the twisted and untwisted $\Gamma$-equivariant Kasparov categories are equivalent. 
    This also yields a Baum--Connes type theorem for this class of symmetries, that is, a $(\Gamma, \omega)$-C*-algebra is equivariantly $\KK$-contractible if the underlying C*-algebra is $\KK$-contractible.
    This provides a prototype of the Baum--Connes type conjecture for tensor category actions.

    The paper is organized as follows. In \cref{section:action}, we prepare the setup of actions of tensor categories on C*-algebras we are working on.
    In \cref{section:equivariant.KKtheory}, we give a definition of the tensor category equivariant KK-theory in both C*-algebraic and C*-categorical ways, and show that they are equivalent. 
    In \cref{section:Cuntz.category}, we study the equivariant Cuntz's picture of our KK-theory, and show that the Kasparov category $\Kas^\sC$ possesses the universality and the triangulated structure. 
    In \cref{section:Morita}, we study the weak Morita invariance of the equivariant Kasparov category, with application to partial Takesaki-Takai duality. 
    In \cref{section:group.cocycle}, we study the Baum--Connes property of $3$-cocycle twists of torsion-free discrete groups.

\subsection{Notations and Conventions}
	\begin{itemize}
	    \item We write $\sC$ for a rigid C*-tensor category with countably many objects. The unit and zero object of $\sC$ is denoted by $\mathbf{1}_\sC$ and $\mathbf{0}_\sC$, or simply by $\mathbf{1}$ and $\mathbf{0}$ when there is no confusion. 
        We also always assume the simplicity of $\mathbf{1}_\sC$ when $\sC$ is rigid. 
		For a triple of objects $\pi , \rho, \sigma \in \Obj \sC$, the associator is denoted by ${\ass} (\pi,\rho,\sigma) \colon (\pi \otimes \rho) \otimes \sigma \to \pi \otimes (\rho \otimes \sigma)$. 
		For an object $\pi \in \Obj  \sC$, its conjugate object is denoted by $\bar{\pi}$, and the counit and unit morphisms are denoted by $\evst_\pi \colon \mathbf{1}_\sC \to \bar{\pi} \otimes \pi $ and $\coev_\pi \colon \mathbf{1}_\sC \to \pi \otimes \bar{\pi} $ respectively. 
		\item Unless otherwise noted, a locally compact group is always assumed to be second countable, as well as a discrete group is always assumed to be countable. Similarly, a compact quantum group $G$ is always assumed to have separable function algebra $C(G)$ and a faithful Haar state. 
		\item Unless otherwise noted, a C*-algebra is always assumed to be $\sigma$-unital and a Hilbert C*-module is always assumed to be countably generated.
		\item Unless otherwise noted, we always mean by $A \otimes B$ or $\sA \boxtimes \sB$ the minimal tensor product of C*-algebras or C*-categories. In most parts, however, the same statement is true even if the minimal tensor product is replaced by the maximal one. 
		In the same way, we mean by $\Gamma \ltimes A$ or $\hat{G} \ltimes A$ the reduced crossed product. 
		\item For a C*-algebra $B$ and a Hilbert $B$-module $E$, we write $\cL_B(E)$ or $\cL(E)$ (resp.~$\cK_B(E)$ or $\cK(E)$) for the C*-algebra of bounded (resp.~compact) adjointable operators. As well, we write $\cU(E)$ for the set of unitary operators on $E$ (we use the same notation for C*-algebras). 
		A $\ast$-homomorphism $A \to \cL(E)$ is called \emph{essential} in this paper if it satisfies $\phi(A)E=E$ (although it is often called to be non-degenerate in the literature) in order to distinguish it from the notion of degenerated Kasparov bimodules. 
	\end{itemize}

\subsection*{Acknowledgment}
The authors would like to thank Sergey Neshveyev, Christian Voigt, and Makoto Yamashita for their interest and discussions. 
YA was supported by JSPS KAKENHI Grant Number 18K13424. 
KK was partially supported by JSPS KAKENHI Grant Number JP21J21283, JST CREST program JPMJCR18T6, JSPS Overseas Challenge Program for Young Researchers, and the WINGS-FMSP program at the University of Tokyo, and he would like to thank his advisor Yasuyuki Kawahigashi for his encouragement and the University of Glasgow for the hospitality during his visit. YK was supported by RIKEN iTHEMS and JSPS KAKENHI Grant Numbers 19K14544, 22K13910, JPMJCR19T2. 
The authors also thank for the kind local hospitality and the support by the Operator Algebra Supporters' Fund provided for the research camp at Tsurui village in September 2022, where part of this study was done.

	\section{Tensor category action: from module category to C*-algebra}\label{section:action}
	In this section we summarize the basic materials for actions of tensor categories on C*-algebras, which is in the spirit of Tannaka--Krein duality of quantum group actions \cites{decommerTannakaKreinDualityCompact2013,decommerTannakaKreinDualityCompact2015,neshveyevDualityTheoryNonergodic2014}.
	We first define the notion of tensor category action in the categorical language, and later introduce a more concrete C*-algebraic counterpart.
	Although these two definitions are essentially the same, there are some minor, rather technical differences. 
	
	\subsection{Nonunital \texorpdfstring{$\sC$}{C}-module C*-category}
    Following the line of \cites{decommerTannakaKreinDualityCompact2013,decommerTannakaKreinDualityCompact2015}, the actions of tensor categories on C*-algebras are described in terms of module C*-categories.
    We work in a slightly broader framework than C*-category in the standard terminology (which we call unital C*-category in this paper), the \emph{nonunital C*-category} \cites{mitchenerCATEGORIES2002,antounBicolimitsCategories2020}, which is the categorical counterpart of non-unital C*-algebra.

	\begin{df}[{cf.~\cite{antounBicolimitsCategories2020}*{2.8}}]\label{def:separable.category}
		A \emph{nonunital C*-category} is a C*-category $\sA$ with a distinguished categorical ideal $\cK_\sA$ such that the space of endomorphisms $\cL_\sA(X)$ in $\sA$ is naturally isomorphic to the multiplier $\cM(\cK_\sA(X))$ for any object $X \in\Obj\sA$. We say that a nonunital C*-category $\sA$ is \emph{separable} if
		\begin{enumerate}
			\item it is idempotent complete and closed under countable \emph{$\ell^2$-direct sums}, 
			\item it has an \emph{dominant object} $X$, in the sense that any $Y \in\Obj\sA$ is a direct summand of $X^{\oplus\infty}$, and
			\item its morphism space $\cK_\sA(X,Y)$ is separable for any $X,Y \in\Obj\sA$.
		\end{enumerate}
	\end{df}
	
	Here, for a countable family $\{ X_n \}_{n \in \bN} \subset \Obj \sA$ in a nonunital C*-category, the object $\bigoplus_n X_n \in \Obj \sA$ is called their $\ell^2$-direct sum if  $\cK_{\sA}( \bigoplus_n X_n, Y) \cong \bigoplus_n \cK_{\sA} (X_n , Y) $ for any $Y \in \Obj \sA$, where the direct sum is taken as Hilbert modules with respect to the inner product $\langle T,S \rangle := TS^* \in \cK_\sA( Y, Y)$.

	For example, a C*-category is a nonunital C*-category with $\cL = \cK$. 
	A C*-category is sometimes called a unital C*-category for emphasis.
	On the other hand, for a separable nonunital C*-category $\sA$, its categorical ideal $\cK_\sA$ is uniquely recovered from $\cL_\sA$. 
	This is because, for a separable C*-algebra $A$, the C*-subalgebra $A \subset \cM(A)$ is characterized as the set of operators $a$ with separable $a \cM(A)$ \cite{woronowiczAlgebrasGeneratedUnbounded1995}*{Proposition~1.1}. 
	
	\begin{rem}\label{rem:separable.envelope}
		For a nonunital C*-category $\sA$, one can add formal subobjects and countable direct sums to obtain a new nonunital C*-category $\tilde \sA$ satisfying (1). 
		When $\sA$ satisfies (2) and (3) of \cref{def:separable.category} (as for (2) we mean that there exists $X \in \Obj \sA$ such that any $Y \in \Obj \sA$ is a direct summand of $X^{\oplus \infty}$ in $\tilde{\sA}$), then $\tilde \sA$ is separable. In this case, we call $\tilde \sA$ the \emph{separable envelope} of $\sA$. 
	\end{rem}
	
	\begin{df}\label{def:nonunital.functor}
		Let $\sA$ and $\sB$ be nonunital C*-categories. 
		A \emph{nonunital $\ast$-functor} $\cF \colon \sA \to \sB$ is a $\ast$-functor which is strictly continuous on the unit ball on each morphism set. In other words, a nonunital $\ast$-functor consists of
		\begin{itemize}
			\item the assignment $\Obj \sA \ni X \mapsto \cF(X) \in \Obj \sB$, and
			\item  the family of unital bounded linear maps $\cF=\cF_{X,Y} \colon \cL_\sA(X,Y) \to \cL_\sB(\cF(X), \cF(Y))$ that is strictly continuous on the unit ball, 		\end{itemize} 
		such that the $\ast$-functoriality $\cF(T\circ S)=\cF(T)\circ \cF(S)$, $\cF(T^*)=\cF(T)^*$ hold for any $S \in \cL_\sA(X,Y)$ and $T \in \cL_\sA(Y,Z)$. 
		A \emph{proper} nonunital $\ast$-functor is a nonunital $\ast$-functor $\cF$ that each $\cF_{X,Y}$ restricts to $\cK_\sA (X,Y) \to \cK_\sB (\cF(X),\cF(Y))$. 
	\end{df}
    Unless otherwise noted, we always assume that a nonunital $\ast$-functor $\cF \colon \sA \to \sB$ is \emph{countably additive}, i.e., the canonical morphism induces $\cF(\bigoplus_n X_n) \cong \bigoplus_n\cF(X_n)$ for any countable family $\{ X_n \}_{n \in \bN} \subset  \Obj \sA$.

    The terminology on functors (natural transform, faithfulness, fullness, and essential surjectivity, etc.) follows the usual category theory. 
    We say that a natural transform of two nonunital $\ast$-functors $\cF, \cG$ is a natural unitary if each $u_X$ is a unitary. 
    We also say that $\cF$ and $\cG$ are naturally unitarily isomorphic ($\cF \cong \cG$) if they are naturally isomorphic by natural unitaries, and two nonunital C*-categories $\sA$, $\sB$ are (C*-)categorically equivalent if there are proper nonunital $\ast$-functors $\cF \colon \sA \to \sB$ and $\cG \colon \sB \to \sA$ such that $\cF \circ \cG \cong \id_{\sB}$ and $\cG \circ \cF \cong \id_{\sA}$.
	Note that, if a fully faithful nonunital $\ast$-functor $\cF $ is essentially surjective, then $\cF$ is a categorical equivalence.

\begin{ex}\label{exmp:nonunital.category}
	The following (nonunital) C*-categories appear in this paper. Among them, (2) and (3) are sometimes treated as nonunital C*-category by $\cK:=\cL$.
	\begin{enumerate}
		\item For a $\sigma$-unital C*-algebra $A$, the category $\Mod(A)$ of (countably generated) Hilbert $A$-modules is a nonunital C*-category by $\cK_{\Mod(A)}(X,Y) = \cK_A(X,Y)$. It is a separable C*-category if so is $A$, with a dominant object $A$.
		For a Hilbert $A$-$B$ bimodule $E$, i.e., a Hilbert $B$-module together with $\phi \colon A \to \cL_B(E)$, the tensor product $ \blank \otimes_A E \colon \Mod(A) \to \Mod(B)$ gives a nonunital $\ast$-functor, which is proper if and only if $\phi (A) \subset \cK_B(E)$.
		\item For C*-algebras $A$ and $B$, 
		the category $\Bimod (A,B)$ of (countably generated) Hilbert $A$-$B$ bimodules and adjointable intertwiners is a C*-category. 
		We write  $\Bimod_{\rm es}(A,B)$ for its full subcategory consisting of essential Hilbert $A$-$B$ bimodules, which is again a C*-category.
		\item Let $\sA$ and $\sB$ be nonunital C*-categories. The functor category $\Func(\sA, \sB)$, the category whose objects are countably additive nonunital $\ast$-functors from $\sA$ to $\sB$ and morphisms are natural transformations, forms a C*-category. 
	\end{enumerate}
\end{ex}

	\begin{prop}[{\cite{antounBicolimitsCategories2020}*{Propositions 2.3, 2.4}}]\label{thm:algebra.category}
		The following hold.
  \begin{enumerate}
  \item Let $\sA$ be a separable nonunital C*-category with a distinguished dominant object $X \in\Obj\sA$. Then there exists a natural categorical equivalence $\sA \simeq \Mod(\cK_\sA(X))$.
  \item For separable C*-algebras $A$ and $B$, the natural functor
		\[\Bimod_{\rm es}(A,B) \to \Func(\Mod(A),\Mod(B)) ; \quad X \mapsto \blank \otimes_A X,\]
		is a categorical equivalence.
  \end{enumerate}
	\end{prop}
 We just recall the construction of these categorical equivalences.  The categorical equivalence in (1) is given by sending $Y$ to $\cK_\sA(X, Y) \in \Mod (\cK_\sA(X))$. The preimage of a functor $\cF \in \Func(\Mod(A),\Mod(B))$ in (2) is given by $E = \cF(A) \in\Obj\Mod(B)$, on which $A$ acts essentially from the left by 
		\[
  \cF_{A,A} \colon A = \cK(A) \to \cL(\cF(A)) = \cL(E).
  \]
  Indeed, the isomorphism $ X \otimes_A E \cong \cF(X)$ is given by the extension of
		\begin{align}
			\cK_A(A,X)  \times E \to \cF(X), \quad (T , \xi) \mapsto \cF(T) \cdot \xi .\label{eqn:functor.to.bimodule}
		\end{align}

	Let $\sA$ be a nonunital C*-category. 
	As is mentioned in \cref{exmp:nonunital.category} (3), the category of endofunctors $\Func(\sA, \sA)$ on $\sA$ is again a C*-category.
	Moreover, it is endowed with the strict monoidal structure by the composition of functors. 
	We write $\Func(\sA, \sA)^{\rm rev}$ for the C*-tensor category with the reversed monoidal structure, i.e., $\cF \otimes \cG:=\cG \circ \cF$. 
	A (right) action of a rigid C*-tensor category $\sC$ on a nonunital C*-category $\sA$ is a monoidal functor
	\[
	(\alpha , \fu) \colon \sC \to \Func(\sA, \sA)^{\rm rev}.
	\]
	We call the triple $(\sA, \alpha, \fu)$ a \emph{nonunital $\sC$-module category}, and often write it just as $\sA$ by suppressing $(\alpha,\fu)$. 
	For $\pi, \sigma \in \Obj \sC $, $X \in \Obj \sA$ and $f \in \Hom_\sC(\pi,\sigma)$, the images $\alpha_\pi (X) \in \Obj \sA$ and $\alpha_f (X) \in \Hom_{\sA}(\alpha_\pi(X), \alpha_\sigma(X))$ are also denoted by $X \otimes_\alpha \pi$ and $1_X \otimes_\alpha f$, or just $X \otimes \pi$ and $1_X \otimes f$ if there is no confusion, respectively. 
	For $\sC$-module categories $\sA$ and $\sB$, a $\sC$-module functor from $\sA$ to $\sB$ is a nonunital $\ast$-functor $\cF \colon \sA \to \sB$ together with natural unitaries $\mathsf{v}_{X,\pi} \colon \cF(X \otimes_\alpha \pi ) \to \cF(X) \otimes_\beta \pi $ satisfying the associativity and the unit conditions 
	(cf.\ \cite{etingofTensorCategories2015}*{Definition 7.2.1}).

 \begin{df}\label{defn:category.Ccat}
 We write $\Ccat^\sC$ for the category of separable nonunital $\sC$-module categories and natural isomorphism classes of proper $\sC$-module functors.
\end{df}
    \begin{df}\label{defn:action.tensor.category}
	Let $A$ be a separable C*-algebra. 
	We call an action of $\sC$ on the nonunital C*-category $\Mod(A)$ in the above sense a (right) \emph{action of $\sC$ on $A$}. 
    \end{df}
Equivalently (up to natural unitary), by \cref{thm:algebra.category} (2), a $\sC$-action on $A$ is given by a C*-tensor functor $(\alpha, \fu) \colon \sC \to \Bimod_{\rm es}(A,A)$. 
We conventionally regard $\Bimod_{\rm es}(A,A)$ as a strict C*-tensor category and omit brackets for iterated tensor products of bimodules. 

\begin{rem}\label{rem:monoidal.invariance}
    A monoidal functor $(\cS , \mathsf{s}) \colon \sC \to \sD$ induces a pull-back functor $\cS^* \colon \Ccat^\sD \to \Ccat^\sC$ by $\cS^*(\sA, \alpha,\fu):= (\sA, \alpha \circ \cS , (1 \otimes_\alpha \mathsf{s}_{\pi \otimes \sigma })\fu_{\cS(\pi),\cS(\sigma)})$ and $\cS^* (\cF, \mathsf{v}):=(\cF, \mathsf{v}_{\cS(\blank )} )$.
    If $\cS$ is a monoidal equivalence, then $\cS^*$ is a categorical equivalence. More explicitly, for the inverse $(\cT , \mathsf{t})$ of $\cS$ and the natural isomorphism $u \colon \id \to \cT \circ \cS$, the family of $\sC$-module functors $ (\id_\sA, \{ 1 \otimes_\alpha u_\pi \}_{\pi \in \Obj \sC }) \colon \sA \to \cS^*\cT^*\sA$ is associated, which forms a natural isomorphism $\id \to \cS^* \circ \cT^*$. 
\end{rem}

	\subsection{\texorpdfstring{$\sC$}{C}-action on C*-algebras and Hilbert bimodules}
	We start with making \cref{defn:action.tensor.category} explicit, to make it parallel to the cocycle actions of groups. 
	\begin{df}\label{def:action.of.tensor.category.explicit}
		A \emph{$\sC$-C*-algebra} is by a triple $(A, \alpha , \fu )$, where 
		\begin{itemize}
			\item $A$ is a C*-algebra, 
			\item $\alpha \colon \sC \to \Bimod_{\rm es}(A)$ is a functor, consisting of the families $\{ \alpha_\pi\} _{\pi \in \Obj \sC}$ and $\{ \alpha_f \in \Hom (\alpha_\pi, \alpha _\sigma) \mid f \in \Hom_\sC(\pi,\sigma) \}$ satisfying the functoriality and $\alpha(\mathbf{1}_\sC)=A$, and 
			\item each $\fu_{\pi, \sigma} \colon \alpha_\pi \otimes _A \alpha_\sigma  \to \alpha_{\pi \otimes \sigma }$ is a bimodule unitary satisfying, for any $f \in \Hom(\pi,\pi')$ and $g \in \Hom (\sigma, \sigma')$, 
			\begin{gather*}
				\xymatrix{
				\alpha_\pi \otimes_A \alpha_\sigma \ar[r]^{\fu_{\pi,\sigma}} \ar[d]^{\alpha_f \otimes_A \alpha_g} & \alpha_{\pi \otimes \sigma} \ar[d]^{\alpha_{f\otimes g}} \\ \alpha_{\pi'} \otimes_A \alpha_{\sigma'} \ar[r]^{\fu_{\pi',\sigma'}} & \alpha_{\pi' \otimes \sigma'},
				}
				\quad 
				\xymatrix@C=8ex{
					\alpha_\pi \otimes_A \alpha_\sigma \otimes_A \alpha_\rho \ar[rr]^{1_{\alpha_\pi} \otimes \fu_{\sigma, \rho}} \ar[d]^{\fu_{\pi,\sigma} \otimes 1_{\alpha_\rho}}&& \alpha_\pi \otimes_A \alpha_{\sigma \otimes \rho} \ar[d]^{\fu_{\pi, \sigma \otimes \rho}} \\
					\alpha_{\pi \otimes \sigma} \otimes _A \alpha_\rho \ar[r]^{\fu_{\pi\otimes \sigma, \rho}} & \alpha_{(\pi \otimes \sigma) \otimes \rho}\ar[r]^-{\alpha_{\ass(\pi,\sigma,\rho)}} & \alpha_{\pi \otimes (\sigma \otimes \rho)}. }
		\end{gather*}\label{eqn:diagram.Caction}
	\end{itemize}
\end{df}
Throughout the paper, a $\sC$-C*-algebra $(A,\alpha, \fu)$ is often denoted just as $A$ by suppressing $(\alpha, \fu)$.

We also introduce the notion of morphisms of $\sC$-C*-algebras corresponding to $\sC$-module functors. It is more relevant to consider Hilbert bimodules (correspondences) instead of $\ast$-homomorphisms. 

\begin{rem}
In this paper, we call a bounded $B$-liner map $\bbmv \colon E_1 \to E_2$ of Hilbert $B$-modules an \emph{isometric map} if it preserves the $B$-valued inner product, i.e., $\langle \bbmv \xi , \bbmv \eta \rangle = \langle \xi , \eta \rangle$,  but may not have a $B$-linear adjoint operator.
We remark that, for a Hilbert $B$-$D$ bimodule $\phi \colon B \to \cL(E')$, $\bbmv \otimes_\phi 1$ is well-defined and is again an isometric map. 
As well, if $E_1$ and $E_2$ are equipped with a left $A$-action preserved by $\bbmv$, then $1 \otimes_\phi \bbmv \colon E' \otimes E_1 \to E' \otimes E_2$ is well-defined and is again a isometric map for a Hilbert $A$-module $E'$. 
\end{rem}

\begin{df}\label{defn:cocycle.hom}
	A \emph{$\sC$-Hilbert $A$-$B$ bimodule} is a triple $\sfE=(E,\phi,\bbmv)$, where 
	\begin{itemize}
		\item $E$ is a Hilbert $B$-module,
		\item a $\ast$-homomorphism $\phi \colon A \to \cL(E)$, 
		\item a family $\{ \bbmv_\pi \colon \alpha_\pi \otimes_A E \to E \otimes _B \beta_\pi \}_{\pi \in \mathrm{Obj}\sC}$ of isometric maps intertwining the left $A$-actions, 
	\end{itemize}
	such that 
		the diagrams
		\begin{gather*}
				\xymatrix{
					\alpha_\pi \otimes _A E \ar[r]^-{\bbmv_\pi} \ar[d]^{\alpha_f \otimes 1_B} & E \otimes_B \beta_\pi \ar[d]^{1_B \otimes \beta_f} \\ 
					\alpha_\sigma \otimes_A E \ar[r]^-{\bbmv_\sigma} & E \otimes_B \beta_\sigma , 
				}
				\xymatrix{
					\alpha_\pi \otimes_A \alpha_\sigma \otimes_A {E} \ar[r]^{1_{\alpha_\pi} \otimes \bbmv_\sigma} \ar[d]^{\fu_{\pi,\sigma} \otimes 1_{{E}}} & \alpha_\pi \otimes_A {E} \otimes_B \beta _\sigma \ar[r]^{\bbmv_\pi \otimes 1_{\beta_\sigma}} & {E} \otimes_B \beta_\pi \otimes _B \beta_\sigma \ar[d]^{1_{E} \otimes \fv_{\pi ,\sigma }} \\ 
					\alpha_{\pi \otimes  \sigma } \otimes _A {E} \ar[rr]^{\bbmv_{\pi \otimes \sigma }} && {E} \otimes _B \beta_{\pi \otimes \sigma} 
				}
		\end{gather*}
		commute for any $\pi, \sigma \in \Obj \sC$ and $f \in \Hom (\pi, \sigma)$. 
		
		We call a $\sC$-Hilbert $A$-$B$ bimodule $\sfE$ proper or essential if so is the underlying Hilbert $A$-$B$ bimodule. 
  We further call a proper $\sC$-Hilbert $A$-$B$ bimodule of the form $(B,\phi,\bbmv)$ a \emph{cocycle $\sC$-$\ast$-homomorphism}, and denote by the pair $(\phi,\bbmv)\colon A \to B$.
\end{df}

An \emph{intertwiner} of $\sC$-Hilbert $A$-$B$ bimodules $(E_i, \phi _i, \bbmv_i)$ ($i=1,2$) is $X \in \cL(E_1, E_2)$ such that the diagrams
\[
\xymatrix{
	\alpha_\pi \otimes E_1 \ar[r]^{\bbmv_{1,\pi}} \ar[d]^{1_{\alpha_\pi} \otimes X} & E_1 \otimes \beta_\pi \ar[d]^{X \otimes 1_{\beta_\pi}}  \\
	\alpha_\pi \otimes E_2 \ar[r]^{\bbmv_{2,\pi}} & E_1 \otimes \beta_\pi
}
\]
commute for any $\pi \in \Obj \sC$. In this sense, a unitary $u \in \cL(E,E')$ intertwines 
\begin{align}
    u \colon (E,\phi,\bbmv) \to (E', \Ad(u) \circ \phi , (u \otimes 1) \bbmv (1 \otimes u^*)), \label{eqn:intertwiner.unitary}
\end{align}
where $(u \otimes 1) \bbmv (1 \otimes u^*)$ denotes the family of unitaries
\[ \alpha_\pi \otimes_{\Ad(u) \circ \phi} E' \xrightarrow{1 \otimes u^*} \alpha_\pi \otimes _\phi E \xrightarrow{\bbmv_\pi} E \otimes_B \beta_\pi \xrightarrow{u \otimes_B 1} E' \otimes_B \beta_\pi.    
\]

Two $\sC$-Hilbert $A$-$B$ bimodules are said to be isomorphic if there is a unitary intertwiner of them. 
We write the unital C*-category of $\sC$-Hilbert $A$-$B$ bimodules and intertwiners by $\Bimod^\sC(A,B)$.
Two $\sC$-C*-algebras $A$ and $B$ are said to be \emph{$\sC$-equivariantly isomorphic}, or \emph{cocycle conjugate}, if they are equivalent in the category of $\sC$-C*-algebras and isomorphism classes of cocycle $\sC$-$\ast$-homomorphisms.

\begin{rem}
In \cref{def:action.of.tensor.category.explicit}, some cocycles are automatically normalized by definition. 
First, by the assumption $\alpha_{\mathbf{1}}=A$, it follows that $\fu_{\mathbf{1},\mathbf{1}}\colon A\otimes_A A\to A$ coincides with the multiplication on $A$ (cf.~\cite{etingofTensorCategories2015}*{(2.24)}). 
Similarly, for $\pi \in \Obj \sC$,  $\fu_{\mathbf{1},\pi}$ and $\fu_{\pi,\mathbf{1}}$ are identified with the left and right module structures of $\alpha_\pi$ under $\alpha_{\mathbf{1}}=A$.
Second, for a $\sC$-Hilbert $A$-$B$ bimodule $\sfE$, we have that $\bbmv_{\mathbf{1}}$ is the inclusion from $\alpha_{\mathbf{1}} \otimes_A E = \overline{\phi(A)E}$ to $E \otimes_B \beta_{\mathbf{1}} = E$, by the commutativity of the right diagram for $\pi=\sigma=\mathbf{1}_\sC$ and the triviality of $\fu_{\mathbf{1},\mathbf{1}}$ and $\fv_{\mathbf{1},\mathbf{1}}$.
\end{rem}

\begin{ex}\label{ex:cocycle.twist}
	Let $(A, \alpha, \fu)$ be a $\sC$-C*-algebra. For $\pi \in \Obj \sC$, let $\tilde{\alpha}_\pi$ be a Hilbert $A$-module and let $v_\pi \colon \alpha_\pi \to \tilde{\alpha}_\pi$ be a unitary of Hilbert $A$-modules (we are typically interested in the case that $\tilde{\alpha}_\pi$ is $A$ or $\alpha_\pi$ itself). Then, with respect to the left $A$-action inherited by $\tilde{\alpha}_\pi$  via $v_\pi$, the triple $(A, \tilde{\alpha}, \tilde{\fu})$ is again a $\sC$-C*-algebra by letting
	\[\tilde{\fu}_{\pi, \sigma}:=v_{\pi \otimes \sigma } \fu_{\pi, \sigma }(v_\pi^* \otimes v_{\sigma}^*) \colon \tilde{\alpha}_\pi \otimes \tilde{\alpha}_\sigma \to  \alpha_\pi \otimes \alpha_\sigma \to \alpha_{\pi \otimes \sigma } \to \tilde{\alpha}_{\pi \otimes \sigma }. \]
	It is $\sC$-equivariantly isomorphic to $(A, \alpha, \fu)$ by the cocycle $\sC$-$\ast$-homomorphism $(\id_A, \{v_\pi \}_{\pi})$. We call it a \emph{cocycle twist} for $(A, \alpha, \fu)$, and write as $(A,\alpha, \fu)_v:=(A, \tilde{\alpha}, \tilde{\fu})$. 
\end{ex}

To a $\sC$-Hilbert $A$-$B$ bimodule $(E,\phi,\bbmv)$, the $\sC$-module functor $\blank \otimes_A E \colon \Mod (A) \to \Mod(B)$ is associated. There is a redundancy in this correspondence:  it is essential $\sC$-Hilbert $A$-$B$ bimodule, not all $\sC$-Hilbert $A$-$B$ bimodule, that corresponds one-to-one to $\sC$-module functor.

\begin{prop}\label{prop:functor_bimodule}
	Let $(A, \alpha, \fu)$ and $(B, \beta, \fv)$ be $\sC$-C*-algebras. Then there is an equivalence of C*-categories $\Bimod^\sC_{\rm es}(A,B)$ and the (unital) C*-category of $\sC$-module nonunital $*$-functors $\Func_\sC (\Mod(A), \Mod(B))$. 
	Moreover, 
	\begin{itemize}
		\item a proper $\sC$-module functor corresponds to a proper $\sC$-Hilbert bimodule, and
		\item a $\sC$-module functor sending $A \in\Obj\Mod(A)$ to $B \in\Obj\Mod(B)$ corresponds to an essential cocycle $\sC$-$\ast$-homomorphism.  
	\end{itemize}
\end{prop}

\begin{proof}
	A $\sC$-Hilbert $A$-$B$ bimodule $(E,\phi,\bbmv)$ induces a $\sC$-module functor $(\cF, \mathsf{v}) \colon \Mod(A) \to \Mod(B)$ given by
	\[ 
	\cF(X):=X \otimes _\phi E, \quad \mathsf{v}_{X,\pi}:= 1_X \otimes \bbmv_{\pi}.
	\]
	Conversely, if we have a $\sC$-module functor $(\cF, \mathsf{v}) \colon \Mod(A) \to \Mod(B)$, the $A$-$B$ bimodule $E:=\cF(A)$ satisfies $\cF \cong \blank  \otimes _A E$ by \cref{thm:algebra.category} (2), 
	and the family of unitaries
	\[ \bbmv_{\pi} \colon  \alpha_\pi \otimes_A E \cong \cF(\alpha_\pi )  \xrightarrow{\mathsf{v}_{A, \pi}} \cF(A) \otimes_B \beta_\pi = E \otimes_B \beta_\pi \]
	satisfies the cocycle conditions, i.e., the diagrams in \cref{defn:cocycle.hom} commute. The remaining part is obvious.
\end{proof}

\begin{df}\label{defn:category.CorrC}
For $\sC$-C*-algebras $A$, $B$, $D$ and 
$\sC$-Hilbert bimodules $\mathsf{E}_1 =(E_1,\phi_1,\bbmv_1) \in \Bimod^\sC(A,D)$ and $\mathsf{E}_2 =(E_2,\phi_2,\bbmv_2) \in \Bimod^\sC(D,B)$, their interior tensor product is defined by
\begin{align}
	\mathsf{E}_1\otimes_D\mathsf{E}_2:=(E_1 \otimes_{\phi_2}E_2, \phi_1 \otimes_{\phi_2}1, (1 \otimes_{D} \bbmv_{2,\pi})(\bbmv_{1,\pi} \otimes_{\phi_2} 1)) \label{eqn:tensor.C-bimodule}
\end{align}
as a $\sC$-Hilbert $A$-$B$ bimodule. 
    Let $\Corr_{\rm es} ^\sC$ denote the category of separable $\sC$-C*-algebras and equivalence classes of separable proper essential $\sC$-Hilbert bimodules. 
\end{df}
Under the identification given in \cref{prop:functor_bimodule}, the composition \eqref{eqn:tensor.C-bimodule} corresponds to the composition of $\sC$-module functors. 	
Moreover, the identity morphism of the category $\Corr_{\mathrm{es}}^\sC$ is given by $(A,\id_A,\mathbbm{1})$, where $\mathbbm{1}_\pi \colon \alpha_\pi \to \alpha_\pi$ is the identity. 
Note that $(E,\phi,\bbmv) \circ (A,\id_A,\mathbbm{1}) = (E,\phi,\bbmv)$ holds only if $\phi \colon A \to \cL(E)$ is essential.

In summary, we get the following equivalence of categories. 


\begin{thm}\label{thm:equivalence.category.algebra}
	The correspondence $(A , \alpha , \fu ) \mapsto (\Mod(A), \otimes_\alpha, \fu)$ coming from a comparison of \cref{defn:action.tensor.category} and \cref{def:action.of.tensor.category.explicit} gives a categorical equivalence of $\Corr_{\rm es}^\sC$ and $\Ccat^\sC$ defined in \cref{defn:category.Ccat} and \cref{defn:category.CorrC}. 
\end{thm}

\begin{proof}
    We have the canonical functor $\Mod \colon \Corr_{\rm es}^\sC \to \Ccat^\sC$ by definition, which is fully faithful by \cref{prop:functor_bimodule} and essentially surjective by \cref{thm:algebra.category}. 
\end{proof}

\begin{df}\label{defn:category.Calg}
  For cocycle $\sC$-$\ast$-homomorphisms $(\phi ,\bbmv)\colon A \to D$ and $(\psi,\bbmw) \colon D \to B$, their composition is defined by 
  \begin{align}
      (\psi,\bbmw) \circ (\phi,\bbmv) \coloneqq (\psi \circ \phi, (1_A \otimes_D \bbmw_\pi)(\bbmv_\pi \otimes_D 1_B)). \label{eqn:product.homomorphism}
  \end{align}
  We write $\Calg^\sC$ for the category of separable $\sC$-*-algebras and $\sC$-$\ast$-homomorphisms.
\end{df}
\begin{rem}
    We mention about a difference of the categories $\Corr_{\mathrm{es}}^\sC$ and $\Calg^\sC$. Indeed, the latter is not a subcategory of the former by two reasons. First, in the category $\Calg^\sC$, cocycle conjugate $\sC$-$\ast$-homomorphisms are not identified.  
    Second, we need not assume that a morphism in the category $\Calg^\sC$ is essential. 
    This works well since the composition \eqref{eqn:product.homomorphism} satisfies $(\phi,\bbmv)\circ (\id,\mathbbm{1}) =(\phi,\bbmv)$ even if $\phi $ is not essential.
\end{rem}

\begin{ex}\label{ex:group.action}
		Let $\Gamma$ be a discrete group, and let $\Hilb_\Gamma^{\rm f}$ denote the C*-tensor category of finite dimensional $\Gamma$-graded Hilbert spaces. 
		Let us consider this category with the reversed monoidal structure $\bC \delta_g \otimes \bC \delta_h :=\bC \delta_{hg}$, namely $\sC:=(\Hilb_{\Gamma}^{\rm f})^{\rm rev}$.

		A C*-algebra $A$ with an action $\alpha \colon \Gamma \to \Aut(A)$ can be viewed as a $\sC$-C*-algebra 
		by considering Hilbert bimodules ${}_{\alpha_g}A$ (the Hilbert $A$-module $A$ with the left action $\alpha_g$) and the canonical unitaries ${}_{\alpha_g}A \otimes{}_{\alpha_h} A \cong {}_{\alpha_{hg}}A$ for $g,h\in \Gamma$. 
		Conversely, take an action $(\alpha,\fu)$ of $\sC $ on $A$, and further assume $\alpha_g \cong A$ as a right Hilbert $A$-module for each $g \in \Gamma$. 
		Then, each $\alpha_g$ is identified with a $\ast$-automorphism on $A$ (denoted by the same letter $\alpha_g$), $\fu_{g,h} \in \cU \cM(A)$, and they satisfy $\Ad(\fu_{h,g}) \circ \alpha_g \circ \alpha_h = \alpha_{gh}$ and $\fu_{hk,g}\alpha_{g}(\fu_{k,h}) = \fu_{k,gh}\fu_{h,g}$.  
		That is, $((\alpha_g)_{g\in \Gamma}, (\overline{\fu}_{g,h})_{g,h\in\Gamma})$ forms a cocycle action of $\Gamma$ (cf.~\cite{szaboCategoricalFrameworkClassifying2021}*{Definition~1.1}) by letting $\overline{\fu}_{g,h}:=\fu_{h,g}^*$. 
		Note that a $\sC$-C*-algebra without the assumption $\alpha_\pi \cong A$ is viewed as a Fell bundle over $\Gamma$ (cf.~\cite{neshveyevDualityTheoryNonergodic2014}*{Example 4.2}).

		More generally, for a $3$-cocycle $\omega \in Z^3(\Gamma ; \bT)$ taking values in $\bT:=U(1)$, let $\Hilb_{\Gamma , \omega}^{\rm f} $ denote the C*-tensor category obtained from $\Hilb_\Gamma^{\rm f}$ by twisting the associator as $\ass(g,h,k):=\omega(g,h,k)$. 
		Then, a $(\Hilb_{\Gamma, \omega}^{\rm f})^{\rm rev}$-C*-algebra has an explicit description similar to left cocycle action of $\Gamma$ up to $\omega$, 
		which will be discussed later in \cref{defn:twisted.action}. 
	\end{ex}

\begin{ex}\label{ex:CQG.action}
	Let $G$ be a compact quantum group. 
	Here we observe that an action of the representation category $\Rep(G)$ is the same thing as a right cocycle action of the dual discrete quantum group $\hat{G}$. 
	More precisely, $\Corr_{\rm es}^{\Rep(G)}$ is equivalent to the category $\Corr_{\rm es}^{\hat{G}}$ of separable right cocycle $\hat{G}$-C*-algebras and unitary equivalence classes of proper essential $\hat{G}$-Hilbert bimodules. 

	We first confirm some notations. 
	The dual discrete quantum group $\hat{G}:=(C^*(G),\hat{\Delta})$ is defined by $\hat{\Delta}(x):=\Ad W^G_{21}(1\otimes x)$ for $x \in C^*(G)$ by using the leg-numbering notation. Here, $W^G$ denotes the multiplicative unitary coming from the left regular representation. 
	For each $\pi\in \Rep (G)$, use the same letter $\pi \colon C^*(G) \to \cB(\sH_\pi)$ for the associated $*$-representation. 
	Note that, in our notation, the tensor product representation $\pi \otimes \sigma$ is given by $(\pi \otimes \sigma) \circ \hat{\Delta}_{21}$. 
	A right cocycle action of $\hat{G}$ on $A$ is a pair $(\alpha, \fu)$, where $\alpha \colon A \to \cM(A\otimes C^*(G)) $ and  $\fu \in \cU\cM(A\otimes C^*(G) \otimes C^*(G))$, such that 
	$( \alpha \otimes \id ) \alpha = \Ad( \fu ) ( \id \otimes \hat{\Delta} )\alpha$ and 
	\[
	(\fu\otimes1)(\id \otimes \hat{\Delta} \otimes \id(\fu))=(\alpha\otimes\id \otimes \id (\fu))(\id \otimes \id\otimes \hat{\Delta}(\fu))
	\] (cf.~\cite{vaesExtensionsLocallyCompactQuantum2003}*{Definition 1.1}). 
		For right cocycle $\hat{G}$-C*-algebras $A$ and $B$, an essential right $\hat{G}$-Hilbert $A$-$B$ bimodule is given by a triple $(E,\phi, \delta)$, where $(E,\phi)$ is an essential Hilbert $A$-$B$ bimodule, $\delta \colon E \to \cM(E \otimes C^*(G))$ satisfies $\delta (a \xi b) = \alpha(a) \delta(\xi)\beta(b)$ and $(\delta \otimes \id)\delta = \fu(\id \otimes \hat{\Delta}) \delta(\blank) \fv^*$.

	To a right cocycle $\hat{G}$-C*-algebra $(A,\alpha,\fu)$, a $\Rep(G)$-action (denoted by $(\alpha, \fu)$ by abuse of notation) is associated by letting $\alpha_\pi := (A\otimes \sH_\pi, (\id_A \otimes \pi) \circ \alpha)$ and $\fu _{ \pi ,  \sigma }:=( \id \otimes \pi \otimes \sigma )( \fu^*_{132} ) \in \cU\cM(A \otimes \cB ( \sH_\pi ) \otimes \cB ( \sH_\sigma ))$ for $\pi, \sigma \in \Rep(G)$. A straightforward computation shows that $(\alpha_{\pi}, \fu_{\pi,\sigma})$ is characterized by the same relation as \cref{def:action.of.tensor.category.explicit}. 
	For an essential $\hat{G}$-Hilbert $A$-$B$ bimodule $(E,\phi,\delta)$, let 
	\[ 
	X \in \cL(E \otimes_\beta (B\otimes C^*(G)), E\otimes C^*(G)), \quad X(\xi \otimes b \otimes x) = \delta (\xi)(b \otimes x),
	\]
	for $\xi\in E$, $b\in B$, $x\in C^*(G)$ be the associated unitary (cf.~\cite{decommerTannakaKreinDualityCompact2013}*{Definition 3.3}). 
	Then, by letting $\bbmv_\pi:=X^* \otimes_{\id\otimes\pi} 1_{B\otimes\cB(\sH_\pi)}$, the triple $(E,\phi,\bbmv)$ forms an essential $\Rep(G)$-Hilbert $A$-$B$ bimodule. 
	The desired relation is checked by
	\[ 
	(\phi\otimes\id\otimes\id(\fu^*))(X\otimes 1_{C^*(G)})(X \otimes_{\beta\otimes\id} 1_{B\otimes C^*(G)\otimes C^*(G)})=(X \otimes_{\id \otimes \hat{\Delta}} 1_{B\otimes C^*(G)\otimes C^*(G)}) (1_E \otimes \fv^*)
	\] 
	as unitaries from $E \otimes_{(\beta\otimes\id)\beta}(B\otimes C^*(G)\otimes C^*(G))$ to $E\otimes C^*(G)\otimes C^*(G)$. 
	Conversely, for an essential $\Rep(G)$-Hilbert $A$-$B$ bimodule $(E,\phi,\bbmv)$, the unitary $X$ with the above relation (and hence $\delta$) is recovered from $\bbmv_\pi$ as $X =\bigoplus_{\pi \in \Irr \sC} \bbmv_\pi^* $.

	Now we get a well-defined functor $\Corr_{\rm es}^{\hat{G}} \to \Corr_{\rm es}^{\Rep(G)}$, which is fully faithful. 
	It is also essentially surjective since any $\Rep(G)$-action on $A$ can be replaced to the one satisfying $\alpha_\pi \cong A\otimes \sH_\pi$ as right Hilbert $A$-modules, after tensoring with the compact operator algebra $\cK$ if necessary (cf.~\cref{lem:endomorphism.realization}). 
		
	When $G=\hat{\Gamma}$ for some discrete group $\Gamma$, a right cocycle action $(\alpha,\fu)$ of $\hat{G}=\Gamma $ corresponds to a cocycle action of $\Gamma$ in the sense of \cref{ex:group.action}, by setting $\alpha_g:=(\id \otimes \eval_g) \circ \alpha$ and $\fu_{g,h}:=(\id \otimes \eval_g \otimes \eval_h)(\fu)$. 
	The above discussion is consistent with \cref{ex:group.action} via the monoidal equivalence $\Hilb_{\Gamma}^{\rm f} \simeq \Rep(\hat{\Gamma})$. 
\end{ex}
	
\begin{rem}\label{rem:appendix.CQG.action}
	We add more comments on \cref{ex:CQG.action} comparing correspondence categories arising from $\hat{G}$-C*-algebras. 
	\begin{enumerate}
	    \item A non-essential $\sC$-Hilbert $A$-$B$ bimodule may not come from any $\hat{G}$-Hilbert $A$-$B$ bimodule. 
	    Indeed, an example is given by the zero $*$-homomorphism $0 \colon A \to I$ and the zero isometries $0_\pi \colon 0 \to I\otimes\beta_\pi$, where $I$ is a non-$\hat{G}$-invariant ideal of $B$. 
	    \item Any cocycle $\hat{G}$-C*-algebra is Morita equivalent (i.e., equivalent in the category $\Corr_{\rm es}^{\hat{G}}$) to a genuine $\hat{G}$-C*-algebra, due to a Packer--Raeburn type stabilization trick \cite{vaesExtensionsLocallyCompactQuantum2003}*{Proposition 1.9}. 
	    This shows that our $\Corr_{\rm es}^{\hat{G}}$ is categorically equivalent to the category of right $\hat{G}$-C*-algebras and unitary equivalence classes of proper essential $\hat{G}$-Hilbert bimodules in the sense of \cite{baajUnitairesMultiplicatifsDualite1993}. 
	\end{enumerate}
\end{rem}

\subsection{Miscellaneous remarks on \texorpdfstring{$\sC$}{C}-Hilbert bimodules}
We remind some generalities of the Hilbert bimodules with conjugate. Standard references are \cites{izumiInclusionsSimpleAst2002,kajiwaraJonesIndexTheory2004}. 
\begin{lem}\label{lem:dual.bimodule}
    Let $E$ be an essential Hilbert $A$-$A$ bimodule with the conjugate $(\overline{E}, \evst,\coev)$. 
    \begin{enumerate}
        \item If $RR^*$ is invertible, then the underlying Hilbert $A$-module $E$ is full. 
        \item The left $A$-action on $E$ is proper. 
        \item If $A$ is $\sigma$-unital, $E$ is countably generated as a Hilbert $A$-module. Moreover, if $A$ is unital, then so is $\cK(E)$. 
        \item Let $A$ be unital, and set $\sH_A :=\ell^2\bZ \otimes _\bC A$. Then, a projection $p\in \cL(\sH_A)$ with $p \sH_A \cong E$ is a compact operator; $p \in \cK(\sH_A)$. In particular, $\cL(E) =\cK(E)$ holds. 
        \item The Hilbert $A$-module $\overline{E}$ is isomorphic to $\cK(E,A)$ with the $A$-valued inner product given by $\langle S,T\rangle := \coev^*(S^*T\otimes_A 1_{\overline{E}})\coev$. 
    \end{enumerate}    
\end{lem}
\begin{proof}
    The claim (1) is checked as $\langle E,E \rangle \supset \langle \overline{E} \otimes _A E , \overline{E} \otimes _A E  \rangle \supset \langle \ev(A),\ev(A) \rangle =A $. 
    To see (2), observe that $(\cK(E) \otimes 1_{\overline{E}})\cK(E \otimes_A \overline{E}) = \cK(E \otimes_A \overline{E})$, and hence
    \begin{align*}
    	A \cdot 1_{E}
    	={}& (1_{E} \otimes_A \ev)(\coev \otimes_A 1_{E}) \cdot (A \otimes _A 1_E) \\
	    \subset {}& (1_{E} \otimes_A \ev)(\cK(A, E \otimes_A \overline{E}) \otimes_A 1_{E}) \\
	    ={} & \cK(E) (1_{E} \otimes_A \ev)(\cK(A,E \otimes_A \overline{E})\otimes_A 1_{E}) 
	    \subset \cK(E). 
        \end{align*}
        Now (3) follows from (2), together with the $\sigma$-unitality of $\cK(E)$. The claim (4) follows from the latter half of (3). 

        To see (5), observe that the conjugation relation shows that the following $A$-module maps are inverses of each other; 
        \begin{align*}
	       &
            \cK_A(E,A) 
            \xrightarrow{\blank \otimes 1_{\overline{E}}} 
            \cK_A(E \otimes_A \overline{E},\overline{E})
            \xrightarrow{\blank \circ \coev }
        	\cK_A(A,\overline{E} )\cong \overline{E},
        	\\&
        	\overline{E} \cong \cK_A(A,\overline{E}) 
        	\xrightarrow{\blank \otimes 1_E}
        	\cK_A(E,\overline{E} \otimes_A E) 
        	\xrightarrow{\ev \circ \blank }
            \cK_A(E,A) . \qedhere 
        \end{align*}
\end{proof}

\begin{rem}\label{rmk:fullHilbmod}
	Let $B$ be a $\sC$-C*-algebra. If a Hilbert $B$-module $E$ is full (or $E \otimes_B \beta_\pi = E \otimes_B \beta_\pi \otimes_B I$ for any $\pi \in \Obj \sC$ where $I := \langle E,E\rangle$), the C*-algebra $\cK(E)$ has a canonical $\sC$-C*-algebra structure $(\beta^E, \fv^E)$ given by 
	\begin{align}
		\begin{split}
			\beta_\pi^E :=& E \otimes_B \beta_\pi \otimes_B E^*,\\
			\fv_{\pi, \sigma }^E:=& \id_E \otimes_B \fv_{\pi,\sigma} \otimes \id_{E^*}.
		\end{split} \label{eqn:module.cocycle}
	\end{align}
	Here, note that $E$ becomes canonically a $\sC$-Hilbert $\cK(E)$-$B$ bimodule by construction. Especially, we apply this to $E\oplus B$ for any Hilbert $B$-module $E$. 
	In this case, a triple $(E,\phi,\bbmv)$ is a proper $\sC$-Hilbert $A$-$B$ bimodule if and only if $(\phi\oplus 0_B,\bbmv\oplus 0) \colon A \to \cK(E\oplus B)$ is a cocycle $\sC$-$\ast$-homomorphism. 
\end{rem}

\begin{rem}\label{ex:essential.bimodule}
	For a $\sC$-Hilbert $A$-$B$ bimodule $\sfE=(E,\phi,\bbmv)$, we define its essential part by
	\[ \sfE_{\rm es}= (E_{\rm es}, \phi_{\rm es}, \bbmv_{\rm es}):= (A,\id, \mathbbm{1}) \otimes_A \sfE, \]
	i.e., $\sfE_{\rm es}= (E_{\rm es}, \phi_{\rm es}, \bbmv_{\rm es})$ where 
	\[E_{\rm es}:=A \otimes_A E = \overline{\phi(A)E}, \quad \phi_{\rm es}(a):=\phi(a)|_{E_{\rm es}}, \quad \bbmv_{{\rm es},\pi} :=\bbmv_\pi.  \]
	Then the underlying Hilbert $A$-$B$ bimodule $(E_{\rm es},\phi_{\rm es})$ is essential, i.e., $\overline{\phi_{\rm es}(A) E_{\rm es}} = E_{\rm es}$. 
\end{rem}

We provide a handy method to treat possibly non-adjointable isometric maps $\bbmv_\pi$ in \cref{defn:cocycle.hom}. 
For $\xi \in \alpha_\pi$, we write $\bar{\xi}$ for the element of $\alpha_{\bar{\pi}} \cong \cK(\alpha_\pi, A)$ (\cref{lem:dual.bimodule} (5)) determined by the inner product with $\xi$. 

\begin{notn}
	Let $A,B$ be C*-algebras, $E_1$ be a right Hilbert $A$-module and $E_2$ be a right Hilbert $A$-$B$ bimodule. 
	For $\xi\in E_1$, we write $T_\xi \in \cL_B (E_2 , E_1 \otimes_A E_2)$ for the operator $\eta \mapsto \xi \otimes_A \eta$. 
\end{notn}
\begin{lem}\label{lem:cocycle_range}
	Let $(A,\alpha,\fu)$ and $(B,\beta, \fv)$ be $\sC$-C*-algebras and let $(E,\phi, \bbmv) $ be a $\sC$-Hilbert $A$-$B$ bimodule. Then the following hold:
	\begin{enumerate}
		\item For any $\pi \in \Obj \sC$, $1_{\alpha_{\mathbf{1}} } \otimes \bbmv_\pi \colon \alpha_{\mathbf{1}}  \otimes \alpha_\pi \otimes E \to \alpha_{\mathbf{1}} \otimes E \otimes \beta_\pi$ is invertible (and hence is a unitary).
		\item For $\pi\in\Obj\sC$ and $\xi \in \alpha_\pi$, the operator $\bbmv_\pi T_\xi \colon E \to E\otimes_B\beta_\pi$ is adjointable with the adjoint 
		\[ E \otimes \beta_\pi \xrightarrow{T_{\bar{\xi}} \otimes 1_{\beta_\pi }} \alpha_{\bar{\pi}} \otimes _A E \otimes_B \beta_\pi \xrightarrow{\bbmv_{\bar{\pi}} \otimes 1_{\beta_\pi}} E \otimes _B \beta_{\bar{\pi}} \otimes_B \beta_{\pi} \xrightarrow{1_E \otimes \beta_{\ev_\pi}\fv_{\bar{\pi} , \pi }} E.  \]
	\end{enumerate}
\end{lem}

\begin{proof}
	First, we show (1). 
	The inverse of $1_{\alpha_{\mathbf{1}_{\sC}}} \otimes \bbmv_\pi$ is given by using the rigidity of $\sC$ as
	\begin{align*} 
		\alpha_{\mathbf{1}_{\sC}} \otimes_A E \otimes_B \beta_\pi
		&\xrightarrow{\fu_{\pi,\bar{\pi}}^*\alpha_{\coev_\pi} \otimes 1_E \otimes 1_{\beta_\pi}} 
		\alpha_{\pi} \otimes_A \alpha_{\bar{\pi}} \otimes_A E \otimes_B \beta_\pi 
		\xrightarrow{1_{\alpha_\pi} \otimes \bbmv_{\bar{\pi}} \otimes 1_{\beta_\pi}} 
		\alpha_{\pi} \otimes_A  E \otimes_B \beta_{\bar{\pi}} \otimes_B \beta_\pi \\ 
		&\xrightarrow{1_{\alpha_\pi} \otimes 1_{E} \otimes \beta_{\ev_\pi}\fv_{\bar{\pi},\pi}} 
		\alpha_\pi \otimes_A E.  
	\end{align*}
	Indeed, the commutativity of the diagram
	\begin{align*}
		\xymatrix@C=5em{
			\alpha_1 \otimes E \otimes \beta_\pi  \ar[r]^-{\fv_{\pi,\bar{\pi}}^*\beta_{\coev_\pi}}  \ar[d]^-{\fu_{\pi,\bar{\pi}}^*\alpha_{\coev_\pi}} &  
			\alpha _1 \otimes E \otimes  \beta_\pi \otimes \beta_{\bar{\pi}} \otimes \beta_\pi \ar[r]^-{\beta_{\ev_\pi}\fv_{\bar{\pi},\pi}}  & 
			\alpha_1 \otimes E \otimes \beta_\pi  \\
			\alpha_\pi \otimes \alpha_{\bar{\pi}} \otimes E \otimes \beta_\pi \ar[r]^-{\bbmv_{\bar{\pi}}} & \alpha_\pi \otimes E \otimes \beta_{\bar{\pi}} \otimes \beta_\pi \ar[r]^-{\beta_{\ev_\pi}\fv_{\bar{\pi},\pi}} \ar[u]_-{\bbmv_\pi} & \alpha_\pi \otimes E \ar[u]_-{\bbmv_\pi}}
	\end{align*}
	shows that the above morphism is a right inverse of $1_{\alpha_{\mathbf{1}_{\sC}}} \otimes \bbmv_\pi$. Since $1 \otimes \bbmv_\pi$ is an isometric map, it is actually the inverse. 
	
    To see (2), we first observe that $ (1_E \otimes \beta_{\ev_\pi}\fv_{\bar{\pi},\pi}) \bbmv_{\bar{\pi}} T_{\bar{\xi}} \bbmv_\pi  = T_{\xi}^*$ for any $\xi \in \alpha_\pi$. Indeed, for any $\zeta \in \alpha_\pi$ and $w \in E$, we have
    \begin{align*}
        (1_E \otimes \beta_{\ev_\pi}\fv_{\bar{\pi},\pi}) \bbmv_{\bar{\pi}} T_{\bar{\xi}}\bbmv_\pi(\zeta \otimes w) &= (1_E \otimes \beta_{\ev_\pi})\bbmv_{\bar{\pi} \otimes \pi} (\fu_{\bar{\pi},\pi} \otimes 1_E) (\bar{\xi} \otimes \zeta \otimes w) \\&= (\alpha_{\ev_\pi} \fu_{\bar{\pi},\pi} (\bar{\xi} \otimes \zeta)) \otimes w = \langle \xi,\zeta \rangle \cdot w.
    \end{align*}
    Let $\xi \in \alpha_\pi$, $x \in E$ and $y \in E \otimes_B \beta_\pi$. 
    Apply Cohen's factorization theorem to pick $a \in A$ and $\xi' \in \alpha_\pi$ such that $\xi = a \cdot \xi'$.  
    By (1), there is $z \in \alpha_\pi \otimes_A E$ with $a^*y = \bbmv_\pi (z)$. 
    Since $T_{\bar{\xi}} y = T_{\bar{\xi}'}a^*y = T_{\bar{\xi}'} \bbmv_\pi (z)$, we conclude that
		\begin{align*}
			\langle x, (1_E \otimes \beta_{\ev_\pi}\fv_{\bar{\pi},\pi}) \bbmv_{\bar{\pi}} T_{\bar{\xi}} y \rangle
			={}&\langle x, (1_E \otimes \beta_{\ev_\pi}\fv_{\bar{\pi},\pi}) \bbmv_{\bar{\pi}} T_{\bar{\xi}'} \bbmv_\pi(z) \rangle  
            = \langle T_{\xi'} x, z \rangle \\
            ={}& \langle \bbmv_\pi T_{\xi'} x, \bbmv_\pi(z) \rangle = \langle \bbmv_\pi T_{\xi'} x, a^*y \rangle =\langle \bbmv_\pi T_{\xi} x, y \rangle . \qedhere
		\end{align*}
\end{proof}

\subsection{Realization by endomorphisms}
\begin{df}\label{def:Caction.endo}
	A \emph{$\sC$-action on $A$ realized by endomorphisms} consists of  
	\begin{itemize}
		\item $\{ \alpha_\pi\}_{\pi \in \Obj \sC \setminus \{ \mathbf{0}\} }$ is a family of essential endomorphisms on $A$, 
		\item $\Hom(\pi,\sigma) \ni f \mapsto \alpha_f \in \cM(A)$ is a family of linear maps forming a $\ast$-functor that satisfies
		\[ \alpha _\sigma(a)\alpha_f = \alpha_f \alpha_\pi(a)\]
		for any $a \in A$, and
		\item $\{ \fu_{\pi, \sigma } \}_{\pi, \sigma \in \Obj \sC \setminus \{ \mathbf{0}\}}$ is a family of unitaries in $\cM(A)$ satisfying 
		\begin{gather*}
		\fu_{\pi',\sigma'} \alpha_\sigma(\alpha_f)\alpha_g = \alpha_{f \otimes g} \fu_{\pi, \sigma}  \quad \text{for any $f \in \Hom(\pi,\pi')$ and $\Hom (\sigma, \sigma')$,}\\ 
			\Ad\fu_{\pi, \sigma}\alpha_\sigma\alpha_\pi = \alpha_{\pi\otimes\sigma}, 
			\quad \fu_{\pi, \sigma \otimes \rho }\fu_{\sigma, \rho} = \alpha_{\ass(\pi,\sigma,\rho)}\fu_{\pi \otimes \sigma , \rho} \alpha_\rho(\fu_{\pi, \sigma}),
			\quad \fu_{\mathbf{1}_\sC, \pi} =1_{\alpha_\pi}=\fu_{\pi, \mathbf{1}}. 
		\end{gather*} 
	\end{itemize}
	Equivalently, we say $(A, \alpha, \fu)$ is an \emph{endomorphism $\sC$-C*-algebra}.
\end{df}
    The above definition is regarded as a special case of \cref{def:action.of.tensor.category.explicit} such that $\alpha_{\pi} \cong A$ as right Hilbert $A$-modules for all $\pi \in \Obj \sC \setminus \{ \mathbf{0}\}$, by putting $\alpha_{\mathbf{0}}:=0$, $\fu_{\sigma,\mathbf{0}}:=0$, and $\fu_{\mathbf{0},\sigma} :=0$. 
    Note that the equation in \eqref{def:Caction.endo} is equivalent to the commutativity of the diagrams in \cref{def:action.of.tensor.category.explicit}.

 In the same way, if $(A, \alpha, \fu)$ and $(B, \beta, \fv)$ are endomorphism $\sC$-C*-algebras, then a $\ast$-homomorphism $\phi \colon A \to B$ and a family of isometric maps $\bbmv_\pi \colon (A,\alpha_\pi)\otimes_\phi B=\overline{\phi(A)B} \to B =(B,\beta_\pi)$ for $\pi \in \Obj \sC\setminus\{\mathbf{0}_\sC\}$ forms a cocycle $\sC$-$\ast$-homomorphism if and only if 
	\begin{align}
		\begin{split}
			\bbmv_\pi \bigl(\phi(\alpha_\pi(a))b\bigr) &= \beta_\pi (\phi(a))\bbmv_\pi (b), \\
			\bbmv_\sigma \bigl(\phi(\alpha_f a)\bigr) &= \beta_f \cdot\bbmv_\pi(\phi(a)),  \\
			\bbmv_{\pi \otimes \sigma} \bigl(\phi(\fu_{\pi, \sigma} \alpha_\sigma(a')a )\bigr) &= \fv_{\pi, \sigma}  \cdot \beta_\sigma\bigl(\bbmv_\pi (\phi(a'))\bigr)\bbmv_\sigma(\phi(a))  , 
		\end{split}\label{eqn:cocycle.hom.endo}  
	\end{align}
	hold for any $\pi, \sigma \in \Obj \sC$, $f \in \Hom(\pi, \sigma)$, $a,a' \in A$, and $b\in\overline{\phi(A)B}$. 
	
	A feature of endomorphism $\sC$-C*-algebra is that the domain $\alpha_\pi \otimes_A E$ of $\bbmv_\pi$ is regarded as a subspace $\overline{\phi(A)E} \subset E$. It enables us to consider extending $\bbmv_\pi$ to a unitary on $E$. 
    \begin{df}\label{defn:trivial.cocycle}
        Let $A$, $B$ be $\sC$-C*-algebras such that $A$ is realized by endomorphism. 
        \begin{enumerate}
            \item For a $\sC$-Hilbert $A$-$B$ bimodule $(E,\phi,\bbmv )$, we call a family of unitaries $\{ \bv_\pi \colon E \to E \otimes_B \beta_\pi \}_{\pi \in \Obj \sC \setminus \{ \mathbf{0}_\sC \}}$ a \emph{unitary cocycle} of $(E,\phi,\bbmv)$ if $\bv_\pi|_{\overline{\phi(A) E}} = \bbmv_\pi$. 
            We call the triple $(E,\phi,\bv)$ a $\sC$-Hilbert $A$-$B$ bimodule with unitary cocycle. 
            \item Assume that $B$ is also realized by endomorphisms. A cocycle $\sC$-$\ast$-homomorphism $(\phi,\bbmv)$ is said to have \emph{trivial cocycle}, and is denoted by $(\phi,\mathbbm{1})$, if $\bv_\pi =1$ is a unitary cocycle of $(B,\phi,\bbmv)$ (in other words, each $\bbmv_\pi \colon \overline{\phi(A)B}\to B$ coincides with the inclusion). We just call the pair $(\phi,\mathbbm{1})$ a \emph{$\sC$-$\ast$-homomophism} in short.  
        \end{enumerate}
    \end{df}
    
\begin{ex}\label{ex:inner.auto}
    The inner cocycle automorphism associated to the unitary $u \in \cM(B)$ is $(\Ad (u) , u(1 \otimes u)^*)$ (cf.~\eqref{eqn:intertwiner.unitary}), 
    where $u(1 \otimes_{\beta_\pi } u)^*$ is the family
    \[
    \beta_\pi \otimes_{\Ad(u)} B \xrightarrow{(1_{\beta_\pi} \otimes u)^*} \beta_\pi \otimes_{\id } B = \beta_\pi \xrightarrow{u} \beta_\pi = B \otimes \beta_\pi.
    \]
    If $B$ is realized by endomorphisms, the above $u(1 \otimes_{\beta_\pi} u^*)$ is identified with $\beta_\pi(u)u^*$. 
    Therefore, $(\Ad (u), \mathbbm{1})$ forms a $\sC$-$\ast$-homomorphism if and only if $u$ is $\sC$-invariant, i.e., $\beta_\pi(u) = u$ for any $\pi \in \Obj \sC \setminus \{ \mathbf{0}\}$. 
    We remark that, for any cocycle $\sC$-$\ast$-homomorphism $(\phi,\bbmv)$ from another $\sC$-C*-algebra $(A,\alpha,\fu)$, we have
    \begin{align*}
        (\Ad (u), u (1 \otimes u)^*) \circ (\phi,\bbmv)&{}= (\Ad(u) \circ \phi, u (1_{\beta_\pi} \otimes u^*) (\bbmv_\pi \otimes_{\Ad(u)} 1_B)  )\\
        &{} = (\Ad(u) \circ \phi, u \bbmv_\pi (1_{\alpha_\pi \otimes_A B} \otimes u^*) ).
    \end{align*} 
\end{ex}
 
	If $\phi$ is essential and $\bbmv$ is a unitary cocycle, with the unitary extension $\bv_\pi \in \cM(B)$, then second and the third equations of \eqref{eqn:cocycle.hom.endo} are simplified
    to
	\[
	\bv_\sigma \phi(\alpha_f)\bv_\pi^* = \beta_f ,  \quad
	\bv_{\pi \otimes \sigma} \phi(\fu_{\pi, \sigma})\bv_\sigma^* \beta_\sigma (\bv_\pi)^* = \fv_{\pi, \sigma}.
	\]
	
\begin{lem}\label{lem:endomorphism.realization}
	Let $(A,\alpha,\fu)$ be a $\sC$-C*-algebra. If the underlying Hilbert $A$-modules $\alpha_\pi$ are all isomorphic to $A$, especially if 
	\begin{enumerate}
		\item[(i)] $(A,\alpha,\fu)=(B \otimes \cK,\beta \otimes \cK , \fv \otimes 1)$ for some $\sigma$-unital $\sC$-C*-algebra $(B,\beta,\fv)$, or 
		\item[(ii)] $(A, \alpha, \fu )=(B \otimes \cO_\infty , \beta \otimes \cO_\infty , \fv \otimes 1)$ for some unital $\sC$-C*-algebra $(B,\beta,\fv)$ and $[\alpha_\pi]=0 \in \K_0(A)$ for all $\pi \in \Obj \sC$, 
	\end{enumerate}
	then $(A,\alpha,\fu)$ is $\sC$-equivariantly isomorphic to an endomorphism $\sC$-C*-algebra.
\end{lem}
\begin{proof}
	Let us fix unitaries $v_\pi \colon \alpha_\pi \to A$. Then $(A, \alpha, \fu)_v$ as in \cref{ex:cocycle.twist} is an endomorphism $\sC$-C*-algebra that is $\sC$-equivariantly isomorphic to $(A,\alpha,\fu)$. 

	Let us observe that (i) and (ii) are sufficient for $\alpha_\pi \cong A$. 
	For (i), it follows from the fullness of $\beta_\pi$ proved in \cref{lem:dual.bimodule} (1). 
	For (ii), the isomorphism $\alpha_\pi \cong A$ is given by a Murray--von Neumann equivalence of the support projection $p$ of $\alpha_\pi$ in \cref{lem:dual.bimodule} (4) and $1_A$.
\end{proof}

\section{\texorpdfstring{$\sC$}{C}-equivariant \texorpdfstring{$\KK$}{KK}-theory}\label{section:equivariant.KKtheory}

\subsection{Equivariant \texorpdfstring{$\KK$}{KK}-group}

For $\bZ_2$-graded C*-algebras $A$ and $B$, a graded Hilbert $A$-$B$ bimodule $E$, a graded Hilbert $B$-$B$ bimodule $M$, and homogeneous operators $x \in \cL(E,E)$ and $y \in \cL(E,E \otimes_B M)$, we write their graded commutator $[x,y]$ as
\begin{align}
	[x,y]:= (x \otimes 1_M) y -  (-1)^{|x| \cdot |y|} yx  \in \cL(E, E \otimes_B M).    \label{eqn:graded.commutator}
\end{align}
We extend the graded commutator to inhomogeneous operators linearly.

\begin{df}\label{def:Kasparovbimodule}
	Let $A$, $B$ be $\sigma$-unital $\bZ_2$-graded $\sC$-C*-algebras (cf.\ \cref{subsubsection:graded}). 
	A \emph{$\sC$-Kasparov $A$-$B$ bimodule} is a quadruple $(E, \phi,\bbmv, F)$, where 
	\begin{itemize}
		\item the triple $\sfE=(E,\phi,\bbmv)$ is a countably generated $\bZ_2$-graded $\sC$-Hilbert $A$-$B$-bimodule, and 
		\item $F \in \cL_B(E)$ is an odd operator with $[\phi(A), F], \phi(A)(F^2-1) , \phi(A) (F-F^*) \subset \cK(E)$ and that
		\[ 
		[F,\bbmv_\pi T_\xi]=(F \otimes_B 1_{\beta_\pi})\bbmv_\pi T_\xi - (-1)^{|\xi|}\bbmv_\pi T_\xi F  \in \cK( E , E \otimes_{B} \beta_\pi)
		\]
		for any $\pi \in \Obj \sC$ and any homogeneous $\xi\in\alpha_\pi$. 
	\end{itemize}
	We write $\bfE^\sC(A,B)$ for the class of $\sC$-Kasparov $A$-$B$ bimodules. 
\end{df}

For $\sC$-Kasparov $A$-$B$ bimodules $(E_0,\phi_0,\bbmv_0,F_0), (E_1,\phi_1,\bbmv_1,F_1)$, their \emph{unitary equivalence} is an even unitary $U \in \cU_B(E_0,E_1)$, such that $\phi_1 = \Ad (U) \circ \phi_0$, $\bbmv_{1,\pi} = (U \otimes_B 1) \bbmv_{0,\pi} (1 \otimes_A U^*)$ for any $\pi \in \Obj \sC$, and $F_1 = UF_0U^*$.

\begin{ex}
	For a $\bZ_2$-graded proper $\sC$-Hilbert $A$-$B$ bimodule $\sfE=(E,\phi,\bbmv)$, the pair 
	$(\sfE,0)$ is a $\sC$-Kasparov $A$-$B$ bimodule. 
	Especially, a $\bZ_2$-graded cocycle $\sC$-$\ast$-homomorphism $(\phi,\bbmv) \colon A \to B$ determines a $\sC$-Kasparov $A$-$B$ bimodule $(B, \phi , \bbmv, 0)$. 
	We simply write $\sfE$ or $(\phi,\bbmv)$ for the corresponding Kasparov bimodule by abusing notation. 
\end{ex}

	Let $A$, $B$, $D$ be $\bZ_2$-graded $\sigma$-unital $\sC$-C*-algebras, $(\sfE,F)\in\bfE^\sC(A,D)$, and $\sfM$ be a $\bZ/2$-graded proper $\sC$-equivariant $D$-$B$ bimodule. We write 
	\begin{align*}
		(\sfE,F) \hotimes_D \sfM := (\sfE \hotimes_D \sfM,  F \hotimes 1 ) \in\bfE^\sC(A,B), 
	\end{align*}
	where $\sfE \hotimes \sfM$ is as \eqref{eqn:tensor.C-bimodule} with the product grading. We remark that this is a special case of the Kasparov product that will be defined in \cref{subsection:Kasparov.product}.

\begin{df}
	A \emph{homotopy} between $\sC$-Kasparov $A$-$B$ bimodules $(\sfE_0,F_0) , (\sfE_1,F_1)$ consists of a $\sC$-Kasparov $A$-$B[0,1]$ bimodule $(\widetilde{\sfE},\widetilde{F})$ and unitary equivalences $U_i \colon (\widetilde{\sfE},\widetilde{F}) \otimes_{B[0,1]} \eval_i \to (\sfE_i, F_i)$ for $i=0,1$, 
	where $\eval_t \colon B[0,1]\to B$ denotes the evaluation at each $t \in [0,1]$. 
	We say $(\sfE_0,F_0)$ and $(\sfE_1,F_1)$ are \emph{homotopic} if there is a homotopy between them. 
\end{df}

\begin{df}
	Let $A$ and $B$ be $\sigma$-unital $\sC$-C*-algebras. The \emph{$\sC$-equivariant KK-group} $\KK^\sC(A,B)$ is defined to be the set of homotopy classes of $\sC$-Kasparov $A$-$B$ bimodules. 
\end{df}

\begin{rem}
	For $(\sfE_1,F_1),(\sfE_2,F_2)\in\bfE^\sC(A,B)$, we call
	\[(\sfE_1,F_1)\oplus(\sfE_2,F_2) :=(\sfE_1 \oplus \sfE_2, \diag(F_1,F_2)) \in \bfE^\sC(A,B) \]
	their direct sum. 
	It is proved in the same way as in non-equivariant KK-theory that: 
	\begin{enumerate}
		\item
		The homotopy of Kasparov bimodules is an equivalence relation. 
		\item We say that a $\sC$-Kasparov $A$-$B$ bimodule is \emph{degenerate} if $[\phi(A),F]=0$, $F^2=1$, $F=F^*$, and $[F, \bbmv_\pi T_\xi]=0$ for any $\pi \in \Obj \sC$ and $\xi \in \alpha_\pi$. A degenerate $\sC$-Kasparov bimodule is homotopic to $\mathbf{0}$. Indeed, the homotopy is given by $(\sfE [0,1), F \otimes 1)$. 
		\item The set $\KK^\sC(A,B)$ is an abelian group under direct sum and the zero element $\mathbf{0}$. Indeed, for $(\sfE, F) \in \bfE^\sC(A,B)$, its inverse is given by $(\sfE^{\rm op}, -F)$.
	\end{enumerate}
\end{rem}

\begin{lem}\label{lem:self.adjoint.Cinvariance}
	The set of operators $X \in \cL_B(E)$ satisfying $[X, \bbmv_\pi T_\xi ]\in \cK(E,E \otimes_B \beta_\pi)$ for any $\pi \in \Obj \sC$ and $\xi \in \alpha_\pi$ forms a $\bZ_2$-graded C*-subalgebra.
\end{lem}
\begin{proof}
	The only non-trivial part is that the set is $\ast$-closed. 
	By \cref{lem:cocycle_range} (2), we have 
	\begin{align*}
		& \big( (X^* \otimes 1_{\beta_{\pi}})\bbmv_\pi T_\xi - (-1)^{|\xi| \cdot |X^*|} \bbmv_\pi T_\xi X^* \big)^* \\
		=& (\bbmv_\pi T_\xi)^* (X \otimes 1_{\beta_{\pi}}) - (-1)^{|\xi| \cdot |X|}X(\bbmv_\pi T_\xi)^* \\
		=& (1_E \otimes \beta_{\ev_\pi} \fv_{\bar{\pi},\pi}) \cdot \Big( \big( \bbmv_{\bar{\pi}} T_{\bar{\xi}} X - (-1)^{|\xi| \cdot |X|} (X \otimes 1_{\beta_{\bar{\pi}}}) (\bbmv_{\bar{\pi}} T_{\bar{\xi}}) \big) \otimes 1_{\beta_{\pi}} \Big) \in \cK(E\otimes_B \beta_\pi, E). \qedhere
	\end{align*}
\end{proof}

	For a $\sC$-Hilbert $A$-$B$ bimodule $\sfE=(E,\phi,\bbmv)$, we define 
	\begin{align*}
		\fD^\sC (\sfE) :=& \{ x \in \cL(E) \mid [x,\phi(A)] \subset \cK(E), [x, \bbmv_\pi T_\xi ] \in \cK(E, E \otimes_B \beta_\pi)  \}, \\
		\fC^\sC (\sfE) :=& \{ x \in \cL(E) \mid x\phi(A), \phi(A)x \subset \cK(E), [x, \bbmv_\pi T_\xi] \in \cK(E, E \otimes_B \beta_\pi)  \}. 
	\end{align*}
	We say an element of $\fC^\sC(\sfE)$ \emph{locally compact}.  
	They are C*-algebras by \cref{lem:self.adjoint.Cinvariance}, and $\fC^\sC(\sfE)$ is an ideal of $\fD^\sC(\sfE)$. 
	The condition for $(\sfE , F) $ to be a Kasparov bimodule is equivalent to that $F \in  \fD^\sC (\sfE)$ represents an odd self-adjoint unitary in  $\fD^\sC(\sfE)/\fC^\sC(\sfE)$. 
	For a $\bZ_2$-graded C*-algebra $D$, the set of homotopy classes of odd self-adjoint unitaries of $\bM_\infty(D)$ is isomorphic to its $\K_1$-group \cite{vandaeleKtheoryGradedBanach1988}. Thus we get a group homomorphism 
	\[ \K_1(\fD^\sC(\sfE) / \fC^\sC(\sfE)) \to \KK^\sC(A,B), \quad [q(F)] \mapsto [\sfE^n, F]. \]
	
	\begin{rem}\label{rem:normalization.F}
		A consequence of \cref{lem:self.adjoint.Cinvariance} and the above argument is that, for any $\sC$-Kasparov $A$-$B$ bimodule $(\sfE, F)$, the operator can be replaced to an odd self-adjoint contraction by a locally compact perturbation.
	\end{rem}
	
	Following \cite{meyerEquivariantKasparovTheory2000}*{Section 4}, for an (not necessarily adjointable) inclusion $E' \subset E$ of Hilbert $B$-modules, we define the C*-subalgebra 
	\[ 
	\cL_{E}(E') := \{ x \in \cL_B(E) \mid  xE, x^*E \subset E' \} \subset \cL_B(E).
	\]
	It is shown in \cite{meyerEquivariantKasparovTheory2000}*{Lemma~4.1} that the restriction $T \mapsto T|_{E'}$ gives an injective $\ast$-homomorphism $\cL_{E}(E') \to \cL(E')$, that is, $T \in \cL_{E}(E')$ is determined by its restriction to $E'$. 
	
	\begin{lem}\label{lem:condbimod}
		Let 
		$(A,\alpha,\fu)$, $(B,\beta,\fv)$ be $\sC$-C*-algebras, and $(E,\phi,\bbmv)$ be a $\sC$-Hilbert $A$-$B$-bimodule. 
		Then for any $\pi \in\Obj\sC$, the $\ast$-homomorphism $\phi(\blank) \otimes 1_{\beta_\pi} \colon A \to \cL(E \otimes_B \beta_\pi)$ uniquely extends to 
		\[ 
		\phi_\pi \colon \cK_A(\alpha_\pi) \to \cL_{E \otimes _B \beta_\pi} (E_{\rm es} \otimes_B \beta_\pi),
		\] 
		such that 
		\[ 
		\bbmv_\pi (x \otimes_\phi 1_E) = \phi_\pi (x) \bbmv_\pi \colon \alpha_\pi \otimes_A E \to E \otimes_B \beta_\pi 
		\]
		for any $x \in \cK(\alpha_\pi)$. 
		If $\bbmv_\pi$ is adjointable, then $\phi_\pi$ coincides with $(\Ad \bbmv_\pi)( \blank \otimes_\phi 1_{E})$. 
	\end{lem}
	
	\begin{proof}
		The assignment 
		\[ \phi _\pi \colon \cK(\alpha_\pi) \to \cL(E_{\rm es} \otimes_B \beta_\pi ), \quad  \phi_\pi(x) = \bbmv_{{\rm es}, \pi} (x\otimes_\phi 1_{E_{\rm es}}) \bbmv_{{\rm es},\pi}^* \]
		is a $\ast$-homomorphism. 
		Thus, since $\cL_{E \otimes_B \beta_\pi }(E_{\rm es} \otimes_B \beta_\pi) \subset \cL(E_{\rm es} \otimes_B \beta_\pi)$, 
		the desired $\ast$-homomorphism is obtained by showing that each $\phi_\pi(x)$ is contained in $\cL_{E \otimes_B \beta_\pi }(E_{\rm es} \otimes_B \beta_\pi)$.
        This holds since for any $\xi,\eta\in \alpha_\pi$, 
			\begin{align*}
				\phi_\pi(\xi \otimes \eta^*) = (\bbmv_\pi T_{\xi})(\bbmv_\pi T_{\eta})^*
				\in \cL_{E \otimes_B \beta_\pi}(E_{\rm es}\otimes_B\beta_\pi). 
			\end{align*}
		Finally, the uniqueness follows from that $\phi_\pi (x)$ is determined by its restriction to $E_{\rm es} \otimes_B \beta_\pi = \Im \bbmv_\pi$, which is characterized by the relation $\bbmv_\pi(x \otimes_\phi 1_E) = \phi_\pi(x)\bbmv_\pi $.
	\end{proof}

\subsection{Kasparov product}\label{subsection:Kasparov.product}
Let $A$, $B$, $D$ be $\bZ_2$-graded $\sC$-C*-algebras, with $A$ being separable, and with $B$, $D$ being $\sigma$-unital. 
We again emphasize that we are assuming Hilbert modules are countably generated. 

Recall that, for a $\bZ_2$-graded Hilbert $D$-module $M$, a $\bZ_2$-graded Hilbert $D$-$B$ bimodule $(E,\phi)$, and $ X \in \cL_B(E)$ with $[\phi(D),X]\subset \cK_B(E)$, an operator $Y \in \cL(M \otimes_D E)$ is said to be \emph{an $X$-connection} if
	\[
	YT_\xi -(-1)^{|X||\xi|} T_{\xi} X,\quad Y^*T_\xi -(-1)^{|X||\xi|} T_{\xi} X^* \in \cK(E, M \otimes _D E) 
	\]
for any $\xi \in M$. 
It is a fundamental fact in KK-theory that $X$-connection exists for any $M$, $(E,\phi)$ and $X \in \cL_B(E)$. See \cite{blackadarTheoryOperatorAlgebras1998}*{Section~18.3} for this and other facts we use.

	\begin{df}
		Let $(\sfE_1,F_1) =(E_1,\phi_1,\bbmv_1,F_1) \in \bfE^\sC(A,D)$ and $(\sfE_2,F_2) =(E_2,\phi_2,\bbmv_2,F_2) \in \bfE^\sC(D,B)$. 
		Set $\sfE=(E,\phi,\bbmv):=\sfE_1 \otimes_D \sfE_2$ as \eqref{eqn:tensor.C-bimodule}.  
		We write $\sfE_1 \# \sfE_2 \subset \cL_B(E)$ for the set of odd self-adjoint operators $F$ satisfying the following conditions; 
		\begin{itemize}
			\item $F$ is an $F_2$-connection,   
			\item $\phi(a)[F_1 \otimes_{\phi_2} 1, F]\phi(a)^* \geq 0$ modulo compact for any $a \in A$, 
			\item $(E, \phi, \bbmv, F) \in \bfE^\sC(A,B)$.
		\end{itemize}
		A \emph{Kasparov product} of $(\sfE_1,F_1)$ and $(\sfE_2,F_2)$ is defined by
		\[(\sfE_1,F_1) \hotimes_D (\sfE_2,F_2) := (\sfE,F) \in \bfE^\sC(A,B), \]
		where $F$ is an arbitrary element of $(\sfE_1,F_1) \# (\sfE_2,F_2)$.
	\end{df}
	
	As in the ordinary non-equivariant KK-theory, the well-definedness and the associativity of the Kasparov product at the level of KK-class are summarized as the following statements.
	
	\begin{thm}\label{thm:Kasparov.product}
		Let $A$, $B$, $D$ be as above. 
		\begin{enumerate}
			\item For $(\sfE_1,F_1) \in \bfE^\sC(A,D)$ and $(\sfE_2,F_2) \in \bfE^\sC(D,B)$, 
			the set $(\sfE_1,F_1)\#(\sfE_2,F_2)$ is non-empty and path-connected. 
			Especially, the Kasparov product of $(\sfE_1,F_1)$ and $(\sfE_2,F_2)$ exists and determines a unique element of $\KK^\sC(A,B)$, 
			for which we write $[\sfE_1,F_1]\otimes_D [\sfE_2,F_2]$.   
			\item The map $\bfE^\sC(A,D)\times \bfE^\sC(D,B)\to \KK^\sC(A,B)$ given by (1) induces a well-defined $\bZ$-bilinear map $\KK^\sC(A,D)\times \KK^\sC(D,B)\to \KK^\sC(A,B)$. 
			\item The Kasparov product given in (2) is associative.  
		\end{enumerate}
	\end{thm}

	We call the \emph{$\sC$-equivaraint Kasparov category} for the category $\Kas^\sC$ whose objects separable $\sC$-C*-algebras, morphisms $\sC$-equivariant KK-groups, and the composition is the Kasparov product. 
	Indeed, the unit morphism of $A \in \Obj \Kas^\sC$ is given by $[\id_A, \mathbbm{1}]$, as is proved later in \cref{lem:KK.unit}. 
	Two separable $\sC$-C*-algebras are said to be $\KK^\sC$-equivalent if they are isomorphic in this category. 
	A separable $\sC$-C*-algebra $A$ is $\KK^\sC$-contractible if it is $\KK^\sC$-equivalent to $0$, or equivalently, $\KK^\sC(A,A) \cong 0$.
	More categorical aspects of $\sC$-equivariant KK-theory will be discussed in \cref{section:Cuntz.category}. 
	
	\begin{rem}
		The Kasparov product is denoted by the graded tensor product symbol $\hotimes_D$ as $[\sfE_1,F_1] \hotimes_D [\sfE_2,F_2]$ because it is a vast generalization of the interior tensor product of bimodules. 
		On the other hand, from the categorical viewpoint, it is natural to be denoted like a composition of morphisms as $[\sfE_2,F_2] \circ [\sfE_1,F_1]$. 
		These two conflict due to the opposite funcoriality. 
		In this paper, therefore, we use both notations depending on the context, i.e., 
		\[
		[\sfE_1,F_1] \hotimes_D [\sfE_2,F_2] = [\sfE_2,F_2] \circ [\sfE_1,F_1].
		\]
	\end{rem}
	
	We give a proof of \cref{thm:Kasparov.product}, which is generally in line with the known versions by Kasparov \cites{kasparovOperatorFunctorExtensions1980,kasparovEquivariantKKTheory1988} and Baaj--Skandalis~\cite{baajCalgebresHopfTheorie1989}, but with technical modifications specific to the tensor category in some details. 
	\begin{lem}\label{lem:Kasprodbimod}
		Let $\sfE_1:=(E_1,\phi_1,\bbmv_1)$ be a $\bZ/2$-graded $\sC$-Hilbert $A$-$D$ bimodule, and 
		let $\sfE_2:=(E_2,\phi_2,\bbmv_2 ,F_2)$ be a $\sC$-Kasparov $D$-$B$ bimodule. 
		Let $\sfE=(E,\phi,\bbmv):=\sfE_1 \hotimes_D \sfE_2$. 
		Then an odd $F_2$-connection $G$ on $E$ satisfies that 
		\[
		[G,\bbmv_\pi T_\xi] \cdot (\cK(E_1) \otimes 1_{E_2}) , \; (\cK(E_1)  \otimes 1_{E_2} \otimes 1_{\beta_\pi }) [G, \bbmv_\pi T_\xi] \subset \cK(E, E \otimes _B \beta_\pi ),
		\]
		for any $\pi \in \Obj \sC$ and $\xi \in \alpha_\pi$.
	\end{lem}

	\begin{proof}
		For the former inclusion, it suffices to show that $[\bbmv_\pi T_\xi,  G] \cdot T_x T_y^* \in \cK(E, E \otimes_B \beta_\pi)$ for any $x, y \in E_1$ and $\xi \in \alpha_\pi$. 
        We first observe that 
        \[ (1 \otimes \bbmv_{2,\pi}) T_\zeta F_2 - (-1)^{|\zeta|} (G \otimes 1_{\beta_\pi}) (1 \otimes \bbmv_{2,\pi}) T_{\zeta} \in \cK(E_2, E \otimes_B \beta_\pi)\]
        for any homogeneous $\zeta \in E_1 \otimes _D \delta_\pi$. Indeed, in the case of $\zeta = y \otimes \eta$ for any homogeneous $y \in E_1$ and $\eta \in \delta_\pi$, we have
        \begin{align*}
            & (1 \otimes \bbmv_{2,\pi}) T_{y \otimes \eta} F_2 = T_y \bbmv_{2,\pi} T_\eta F_2 \equiv (-1)^{|\eta|}T_y (F_2 \otimes 1_{\beta_\pi}) \bbmv_{2,\pi } T_\eta \\
            \equiv{}& (-1)^{|\eta| + |y|} (G \otimes 1_{\beta_\pi}) T_y \bbmv_{2,\pi} T_\eta = (-1)^{|\eta| + |y|} (G\otimes 1_{\beta_\pi}) (1 \otimes \bbmv_{2,\pi})T_{y \otimes \eta}
        \end{align*}
        modulo $\cK(E_2, E \otimes_B \beta_\pi)$.
        This extends to entire $E_1 \otimes _D \delta_\pi$ by continuity of $\zeta \mapsto (1 \otimes \bbmv_{2,\pi})T_\zeta$. 
        Applying this to $\zeta=\bbmv_{1,\pi}(\xi \otimes x)$, we get
        \begin{align*}
            \bbmv_{\pi} T_\xi G T_x &{}\equiv (-1)^{|x|} \bbmv_\pi T_\xi T_x F_2 = (-1)^{|x|} (1 \otimes \bbmv_{2,\pi})T_{\bbmv_{1,\pi}(\xi \otimes x)} F_2 \\
            &{}\equiv (-1)^{|\xi|} (G \otimes 1_{\beta_\pi}) (1 \otimes \bbmv_{2,\pi}) T_{\bbmv_{1,\pi}(\xi \otimes x)} = (-1)^{|\xi|} (G \otimes 1_{\beta_\pi}) \bbmv_\pi T_\xi T_x
        \end{align*}
        modulo $\cK(E,E \otimes \beta_\pi)$.

		The latter inclusion is reduced to the former one for the $F_2$-connection $G^*$ by using \cref{lem:cocycle_range} (2)  as
		\begin{align*}
			&(\bbmv_\pi T_\xi G -(-1)^{|\xi|} (G \otimes 1_{\beta_\pi})\bbmv_\pi T_\xi)^* (\cK(E_1) \otimes 1_{E_2} \otimes 1_{\beta_\pi}) \\
			=& (1_{E} \otimes \beta_{\ev_{\pi}}\fv_{\bar{\pi},\pi}) \cdot \big( ((G^* \otimes 1_{\beta_{\bar{\pi}}}) \bbmv_{\bar{\pi}} T_{\bar{\xi}} - (-1)^{|\bar{\xi}|} \bbmv_{\bar{\pi}}T_{\bar{\xi}} G^*  ) \cdot  (\cK(E_1) \otimes 1_{E_2}) \big) \otimes 1_{\beta_\pi} \subset \cK(E). 
		\end{align*}   
		Now we get the conclusion by taking the adjoint.
	\end{proof}
	
	\begin{proof}[Proof of \cref{thm:Kasparov.product}]
		Let $G \in \cL_B(E)$ be an odd self-adjoint $F_2$-connection. For $\pi \in \Obj \sC$, let
		\begin{align}
			\sigma_\pi  \colon \cK(E) \to \cK(E \oplus E \otimes_B \beta_\pi), \quad \sigma_\pi(x):=\diag (x,x \otimes_B 1_{\beta_\pi}) \label{eqn:technical}
		\end{align} 
		and, for $\xi \in \alpha_\pi$, let
		\[V_{\xi}:= \begin{pmatrix} 0 & (\bbmv_\pi T_\xi)^* \\ \bbmv_\pi T_\xi & 0 \end{pmatrix} \subset \cL(E \oplus E \otimes_B \beta_\pi ).  \]
        Let $u_n$ be an approximate unit $u_n$ of $\cK(E_1) \hotimes 1 + \cK(E)$. Then $(1-u_n\otimes1_{\beta_\pi})(1 \hotimes \bbmv_{2,\pi}) T_\zeta \to 0$ holds for any $\zeta \in E_1 \hotimes \delta_\pi$, and hence $(1-u_n\otimes1_{\beta_\pi})\bbmv_{\pi}T_\xi T_x = (1-u_n\otimes1_{\beta_\pi})(1 \hotimes \bbmv_{2,\pi})T_{\bbmv_{1,\pi}(\xi \hotimes x)} \to 0$ for any $x \in E_1$. 
        Therefore, for $K \in \cK(E_1)$, by \cref{lem:cocycle_range} (2) we have 
		\begin{align}
			\begin{split}
				\sigma_\pi (u_n) V_\xi \sigma_\pi (K \otimes 1_{E_2}) 
				=& 
				\begin{pmatrix} 
					0 & u_n(\bbmv_\pi T_\xi)^* (K \otimes 1_{E_2} \otimes 1_{\beta_\pi}) \\ 
					(u_n \otimes 1_{\beta_\pi}) \bbmv_\pi T_\xi (K \otimes 1_{E_2}) & 0
				\end{pmatrix}  
				\\
				\to & V_\xi \sigma_\pi(K \otimes 1_{E_2}) \quad \text{as $n \to \infty$. }
			\end{split}
			\label{eqn:technical.check}
		\end{align} 
		
		Now we are ready to apply \cref{thm:technicalappendix} for $I=\Obj \sC$ and
		\begin{itemize}
			\item $\cJ=\cJ_{\mathbf{0}_\sC}:=\cK(E)$ and $\sigma_{\mathbf{0}_\sC}:=\id_\cJ$, 
			\item $ \cJ_\pi := \cK(E \oplus E \otimes \beta_\pi)$ and $\sigma_\pi \colon \cJ \to \cJ_\pi $ is as in \eqref{eqn:technical} for $\pi\neq\mathbf{0}_\sC$, 
			\item $\cA_1 := \cK(E_1) \hotimes 1 + \cK(E)$, 
			\item $\cA_{2, \mathbf{0}_\sC} := C^*(G^2-1, [G,\phi(A)], [G, F_1 \hotimes 1])$, 
			\item $\cA_{2, \pi}:=C^*( \{  [\sigma_\pi (G),  V_{\xi} ]  \mid \xi \in \alpha_\pi  \} )$ for $\pi \neq \mathbf{0}_\sC$, 
			\item $\Delta_{\mathbf{0}_\sC} := \overline{\rm span}\{ F_1 \hotimes 1, G, \phi(A) \}$ and $\Delta_\pi := \overline{\rm span}\{  V_{\xi} \mid \xi \in \alpha_\pi  \}$ for $\pi \neq \mathbf{0}_\sC$. 
		\end{itemize}
		Indeed, $\cA_1 \cdot \cA_{2,\mathbf{0}_\sC} \subset \cJ_{\mathbf{0}_\sC}=\cJ$ and $[\Delta_{\mathbf{0}_\sC}, \cA_1]  \subset \sigma_{\mathbf{0}_\sC}(\cA_1)=\cA_1$ are obvious from the assumption, and  
		\cref{lem:Kasprodbimod} shows that $\sigma_\pi (\cA_1) \cdot \cA_{2,\pi} \subset \cJ_\pi $. Moreover, by \eqref{eqn:technical.check}, we have $[\Delta_\pi, \sigma_\pi (\cA_1)] \subset \overline{\sigma_\pi(\cA_1)\cM(\cJ_\pi) \sigma_\pi(\cA_1)}$. 
		Therefore, we get an even positive operator $M \in \cM(\cJ) = \cL(E)$ with $0 \leq M \leq 1$ such that 
		\begin{enumerate}
			\item $M (\cK(E_1) \hotimes 1), (1-M)(G^2-1), (1-M)[G,\phi(A)], (1-M)[G,F_1\hotimes 1] \subset \cK(E)$, 
			\item $[F_1 \hotimes 1, M], [G, M], [\phi(A), M] \subset \cK(E)$, and
			\item $((1-M) \hotimes 1_{\beta_\pi})[G, \bbmv_\pi T_\xi], (M \hotimes 1_{\beta_\pi})[F_1 \hotimes 1, \bbmv_\pi T_\xi], [M, \bbmv_\pi T_\xi] \subset \cK(E,E \otimes_B \beta_\pi)$, 
		\end{enumerate}
		for any $\pi \in \Obj \sC$ and $\xi \in \alpha_\pi$. 
		Set 
		\[ F:= M^{1/4}(F_1 \hotimes 1)M^{1/4} + (1-M)^{1/4}G(1-M)^{1/4}.\]
		Then (1) and (2) show that $(\sfE,F)$ satisfies the conditions of (non-equivariant) Kasparov product, and  
		(3) shows that 
		\begin{align*}
			[F,\bbmv_\pi T_\xi] 
			\equiv& [M^{1/2}(F_1 \hotimes 1) + (1-M)^{1/2}G, \bbmv_\pi T_\xi]\\
			=& (-1)^{|\xi|} [M^{1/2} , \bbmv_\pi T_\xi](F_1 \hotimes 1) + (M^{1/2}\hotimes 1_{\beta_\pi}) [F_1 \hotimes 1 , \bbmv_\pi T_\xi] \\
			& \quad  +(-1)^{|\xi|} [(1-M)^{1/2}, \bbmv_\pi T_\xi]G + ((1-M)^{1/2}\hotimes 1_{\beta_\pi})[G,\bbmv_\pi T_\xi] 
			\equiv  0 
		\end{align*}
		modulo $\cK(E, E \hotimes \beta_\pi)$. 
		
		The above construction of the operator $M$ also shows that $(\sfE_1, F_1) \# (\sfE_2, F_2)$ is path-connected and the Kasparov product is associative, just in the same way as Kasparov's original work \cite{kasparovOperatorFunctorExtensions1980}. We omit this part. 
	\end{proof}
	
		A by-product of \cref{lem:Kasprodbimod} is that, for a proper $\sC$-Hilbert $A$-$B$ bimodule $\sfM$ and a $\sC$-Kasparov $D$-$B$ bimodule $(\sfE, F)$, any odd self-adjoint $F$-connection $G \in \cL(\sfM \otimes_D \sfE)$ is in $(\sfM,0) \# (\sfE,F)$. In particular, for a $\sC$-Kaspraov $A$-$B$ bimodule $(\sfE, F)$ and an $F$-connection $F_{\rm es} \in \cL(A \otimes_A E)$ (recall that $A \otimes _A \sfE$ is denoted by $\sfE_{\rm es} $, cf.~\cref{ex:essential.bimodule}), the pair $(\sfE_{\rm es},  F_{\rm es})$ is a Kasparov product $ (\id_A, \mathbbm{1}) \hotimes_A (\sfE, F)$. 
	\begin{lem}\label{lem:KK.unit}
		Any $\sC$-Kasparov bimodule $(\sfE, F) \in \bfE^\sC(A,B)$ is homotopic to $(\sfE_{\rm es}, F_{\rm es})$. That is, $(\id_A, \mathbbm{1})$ is the identity morphism in $\Kas^\sC$. 
	\end{lem}
	\begin{proof}
		The proof follows the line of \cite{meyerEquivariantKasparovTheory2000}*{Lemma 3.3}. 
		Let 
		\[
		\tilde{\phi}:=\begin{pmatrix}\phi_{11} & \phi_{12 }\\ \phi_{21} & \phi_{22}\end{pmatrix} \colon \bM_2(A) \to \begin{pmatrix} \cL(E) & \cL(E_{\rm es}, E) \\ \cL(E, E_{\rm es}) & \cL(E_{\rm es}) \end{pmatrix} = \cL(E \oplus E_{\rm es}),
		\]
		where $\phi_{ij}(a)$ is the suitable restriction of $\phi(a) \in \cL(E)$. Then, by \cite{meyerEquivariantKasparovTheory2000}*{Lemma 3.3} and that $[\diag (F,F_{\rm es}) , \diag (\bbmv_\pi, \bbmv_{{\rm es}, \pi})]$ is compact, 
		the quadruple
		\[
		(\widetilde{\sfE}, F):=\big( E \oplus E_{\rm es}, \tilde{\phi}, \diag (\bbmv_\pi, \bbmv_{{\rm es},\pi}), \diag (F,F_{\rm es}) \big) 
		\]
		is a $\sC$-Kasparov $\bM_2(A)$-$B$ bimodule. 
		Therefore, the $\sC$-Kasparov $A$-$B$ bimodules
		\begin{align*}
			(\sfE,F) \sim &
			(\sfE, F) \oplus (E_{\rm es}, 0, 0, 0) \simeq 
			i_1 \hotimes_A (\widetilde{E},F),\\
			(\sfE_{\rm es}, F_{\rm es}) \sim &
			(E, 0,0,0) \oplus (\sfE_{\rm es}, F_{\rm es}) \simeq
			i_2 \hotimes_A (\widetilde{E},F), 
		\end{align*} 
		are homotopic, where $i_1 , i_2 \colon A \to \bM_2(A)$ denote the upper-left and lower-right corner embeddings respectively.    
	\end{proof}

	A consequence of this lemma is the recovery of discrete quantum group equivariant KK-theory \cite{baajCalgebresHopfTheorie1989} as an extension of \cref{ex:CQG.action}.
    For a compact quantum group $G$, $\hat{G}^{\rm op}$ denotes the discrete quantum group with the comultiplication opposite to that of $\hat{G}$. Especially, $\Gamma^{\rm op}$ is the opposite group of $\Gamma$.
    We write $\KK^{\hat{G}^{\rm op}}(A,B)$ for $\hat{G}^{\rm op}$-equivariant $KK$-group for separable $\hat{G}^{\rm op}$-C*-algebras (in other words, right $\hat{G}$-C*-algebras).  
	\begin{thm}\label{thm:RepCQG.KK}
		Let $G$ be a compact quantum group. Then, under the identification in \cref{ex:CQG.action}, there exists a natural isomorphism
		\[ 
		\KK^{\Rep(G)}(A,B) \cong \KK^{\hat{G}^{\rm op}}(A,B) \cong \KK^{G}( A\rtimes \hat{G}, B\rtimes \hat{G} )
		\]
		for any right $\hat{G}$-C*-algebras $A$ and $B$, inducing $\Kas^{\Rep(G)} \simeq \Kas^{\hat{G}^{\rm op}} \simeq \Kas^{G}$. 
		In particular, for a countable discrete group $\Gamma$ and separable right $\Gamma$-C*-algebras $A$ and $B$, we get $\Kas^{\Hilb_\Gamma^{\rm f}} \simeq \Kas^{\Gamma^{\rm op}}$ given by the natural isomorphism
		\[
		\KK^{\Hilb_\Gamma^{\rm f} } ( A, B ) \cong \KK^{\Gamma^{\rm op}} (A,B).
		\]
	\end{thm}
    Here we have used  the Baaj--Skandalis--Takesaki--Takai duality in the second isomorphism above (see e.g.~\cite{nestEquivariantPoincareDuality2010}*{Theorem~4.2} for this version). 
    
		\begin{proof}
			By \cref{lem:KK.unit}, every $\sC$-Kasparov $A$-$B$ bimodule is equivalent to an essential one. This shows the theorem since, as is noted in \cref{rem:appendix.CQG.action}, an essential $\sC$-Kasparov $A$-$B$ bimodule is identified with a $\hat{G}^{\rm op}$-Kasparov $A$-$B$ bimodule. 
		\end{proof}
	
		\subsection{Basic properties}
		The following standard properties of $\KK$-theory immediately extend to the $\sC$-equivariant version.
		\begin{prop}\label{lem:split_exact}
			The following hold.
			\begin{enumerate}
				\item If $\sfE$ is a $\sC$-equivariant imprimitivity $A$-$B$ bimodule in the sense of \cref{subsubsection:Morita}, then $[\sfE,0]\in \KK^\sC(A,B)$ is a $\KK^\sC$-equivalence. 
				Especially, if $(\phi,\mathbbm{1}) \colon A \to B$ is a full-corner embedding, then the KK-class $[\phi , \mathbbm{1}]$ is a $\KK^\sC$-equivalence.
				\item The $\KK^\sC$-bifunctor is split exact in both the first and second variables. That is, for any split exact sequence $0 \to I \to A \to A/I \to 0$, we have split exact sequences
				\begin{align*}
					0 \to  \KK^\sC(B,I) \to \KK^\sC(B,A) \to \KK^\sC(B,A/I) \to 0, \\
					0 \to  \KK^\sC(A/I,B) \to \KK^\sC(A,B) \to \KK^\sC(I,B) \to 0. 
				\end{align*}
				\item For a cocycle $\sC$-$\ast$-homomorphism $(\phi, \bbmv) \colon A \to B$, one has the following Puppe exact sequences
				\begin{align*}
					\KK^\sC(D,SA) \to \KK^\sC(D,SB) \to \KK^\sC(D,\cone (\phi,\bbmv) ) \to \KK^\sC(D,A) \to \KK^\sC(D,B), \\
					\KK^\sC(B,D) \to \KK^\sC(A,D) \to \KK^\sC(\cone (\phi,\bbmv) ,D) \to \KK^\sC(SB,D) \to \KK^\sC(SA,D) ,
				\end{align*}
				where $\cone(\phi, \bbmv)$ denotes the mapping cone (cf.\ \cref{ex:mapping.cone}). 
				\item The Bott periodicity isomorphisms: $\KK^\sC(A,S^2B) \cong \KK^\sC(A,B) \cong \KK^\sC(S^2A,B)$. 
				\item The category $\Kas^\sC$ admits direct sums of countably many objects. 
			\end{enumerate}
		\end{prop}
		\begin{proof}
			The proof is given in the same way as in the non-equivariant case. See e.g. \cite{skandalisRemarksKasparovTheory1984}, \cite{cuntzMappingConesExact1986} or \cite{blackadarTheoryOperatorAlgebras1998}*{Sections 18, 19}. 
		\end{proof}

		Another basic operation of the $\KK^\sC$-theory is the exterior tensor product. 
		Let $\sC$, $\sD$ be rigid tensor categories with countably many objects, let $A$, $B$ be $\sC$-C*-algebras, and let $A'$, $B'$ be $\sD$-C*-algebras. 
		Recall that the (minimal) tensor product $A \otimes B$ is endowed with the action of the Deligne tensor product $\sC \boxtimes \sD$ (cf.\ \cref{subsubsection:external.tensor}). 
		For $(\sfE_1, F_1) \in \bfE^\sC(A_1, B_1)$ and $(\sfE_2, F_2) \in \bfE^\sD(A_2, B_2)$, 
		we define their exterior tensor product as 
		\begin{align*}
			(\sfE_1 ,F_1) \otimes (\sfE_2, F_2):=(E_1 \otimes E_2, \phi_1 \otimes \phi_2, \bbmv_1 \otimes \bbmv_2, F_1 \otimes F_2)\in \bfE^{{\sC} \boxtimes {\sD}}(A \otimes A', B \otimes B') . 
		\end{align*}
		
\begin{thm}
	The above exterior tensor product induces a group homomorphism  
	\[
	\blank \otimes \blank  \colon \KK^\sC(A,B) \otimes \KK^\sD(A',B') \to \KK^{\sC \boxtimes \sD}(A \otimes A', B \otimes B')
	\]
	that gives rise to the bifunctor
	\[ 
	\blank \otimes \blank  \colon \Kas^\sC \times \Kas^\sD \to \Kas^{\sC \boxtimes \sD}
	\]
	such that $\KK^{\sC \boxtimes \sD} \circ (\blank \otimes \blank ) \cong  (\blank \otimes \blank ) \circ (\KK^\sC \times \KK^\sD)$ as bifunctors $\Calg^\sC \times \Calg^\sD  \to \Kas^{\sC \otimes \sD}$. 
\end{thm}

\subsection{\texorpdfstring{$\sC$}{C}-module category picture}
We reformulate the $\KK^\sC$-theory in the language of $\sC$-module categories. Here $\sC$-Kasparov bimodule is understood as a generalization of $\sC$-module functor. 
Let $\sA$, $\sB$ be nonunital $\sC$-module categories in the sense of \cref{def:separable.category}. 
\begin{df}\label{def:equivariant.KK.category}
	A \emph{$\sC$-Kasparov $\sA$-$\sB$ bimodule} is the pair $(\cE, F)$, where $\cE=(\cE_+,\cE_-)$ is a pair of $\sC$-module functors from $\sA$ to $\sB$ in the sense of \cref{def:nonunital.functor} (we regarded it as a functor taking value in $\bZ_2$-graded object $\cE(X):=\cE_+(X) \oplus \cE_-(X)$ in $\sB$) and $F=(F_X)_{X \in \Obj \sA}$ is a family of odd adjointable operators on $\cE(X)$ such that
	\begin{enumerate}
		\item $F_Y \cE(x) - \cE(x) F_X, \cE(x)(F_X^* F_X -1), \cE(x)(F_X - F_X^*) \in \cK_\sB(\cE(X),\cE(Y))$ for any $X,Y \in\Obj\sA$ and $x \in \cK_\sA(X,Y)$,
		\item the diagram
		\[
		\xymatrix{
			\cE(X \otimes \pi) \ar[r]^-{\mathsf{v}_{X,\pi}}\ar[d]_{F_{X \otimes \pi}} & \cE(X) \otimes \pi \ar[d]^{F_X \otimes 1_\pi} \\ 
			\cE(X \otimes \pi) \ar[r]_-{\mathsf{v}_{X,\pi}} & \cE(X) \otimes \pi
				}
		\]
		commutes up to compact after multiplying $\cK_\sB (\cE(X \otimes \pi) )$ from the right.
	\end{enumerate}
\end{df}

Notice that any pair of proper functors $\cE_\pm \colon \sA \to \sB$ gives a Kasparov module by $F = 0$.
		
We say that two $\sC$-Kasparov $\sA$-$\sB$ bimodules $(\cE_{i},F_i)$ ($i=1,2$) are unitary equivalent if there are natural unitaries $u = u_+ \oplus u_- \colon \cE_{1} \to \cE_{2}$ such that each $u_{X} \colon \cE_{1,\pm}(X) \to \cE_{2,\pm}(X)$ satisfies $F_{2,X} u_X=u_X F_{1,X}$. We say $(\cE_1,F_1)$ and $(\cE_2,F_2)$ are homotopic if each of them is unitary equivalent to $(\cE , F) \otimes \eval_i$ for some $\sC$-Kasparov $\sA$-$\sB[0,1]$ bimodule $(\cE, F)$, where the tensor product $\eval_i$ with respect to a $\sC$-$\ast$-homomorphism is defined in a natural way. 
The KK-group $\KK^\sC(\sA,\sB)$ is the set of homotopy classes of Kasparov modules.
		
\begin{prop}\label{prp:category.algebra.KK}
	For separable $\sC$-C*-algebras $A$ and $B$, we have an isomorphism 
	\[ \KK^\sC(\Mod(A),\Mod(B)) \simeq \KK^\sC(A,B).\]
\end{prop}
\begin{proof}
	Let $(\cE, F)$ be a $\sC$-Kasparov $\Mod(A)$-$\Mod(B)$ bimodule. Then, as is mentioned below \cref{thm:algebra.category} (2), $\cE(A)$ has a canonical structure of an essential $\sC$-Hilbert $A$-$B$ bimodule. Now $F_A$ satisfies the assumptions for $(\cE(A), F_A)$ being a $\sC$-Kasparov $A$-$B$ bimodule. 
			
	Conversely, for an essential $\sC$-Kasparov $A$-$B$ bimodule $(E,\varphi,\bbmv, F)$, we get a $\sC$-Kasparov $\Mod(A)$-$\Mod(B)$ bimodule in the following way. 
	Let $\sH:=\ell^2\bZ$, and let us fix an even embedding $V_X \colon X \to \sH \otimes A $ for each Hilbert $A$-module $X$. 
	Set $F_\infty:=1_{\sH} \hotimes F  \in \cL(\sH \hotimes_\bC E)$.
	Then 
	\begin{align}
		F_X:= (V_X^* \hotimes_A 1_{E}) F_\infty (V_X \hotimes_A 1_{E}) \in \cL_B(X \hotimes_A E) \label{eqn:category.Fredholm}
	\end{align}  
	satisfies (1), (2) of \cref{def:equivariant.KK.category}.  
			
	By construction, these assignments preserve the homotopy of bimodules, and hence induce homomorphisms
	\begin{align*}
		&\cX \colon \KK^\sC(\Mod(A), \Mod(B)) \to \KK^\sC(A,B), \quad [\cE, F] \mapsto [\cE(A), F_A], \\
		&\cY \colon \KK^\sC(A,B) \to \KK^\sC(\Mod(A), \Mod(B)) ,\quad [\sfE , F] \mapsto [ \blank \hotimes_A \sfE, (F_X)_X ].
	\end{align*}
	The equality $\cX \circ \cY =\id$ is obvious from the definition. 
	To see $\cY \circ \cX =\id$, set $ (\cE', F'):=(\cY \circ \cX) (\cE, F)$. Then, by \cref{thm:algebra.category} (2), we have a natural isomorphism $u \colon \cE \to \cE$ such that $u_A F_A u_A^*=F'_A$. Now $t F_X + (1-t)u_X^* F_X'u_X$ gives an operator homotopy connecting $(\cE, F)$ and $(\cE', F')$. 
\end{proof}
\begin{ex}
	If $\sC = \Rep(G)$ for some compact quantum group $G$ (including the case that $\sC=\Hilb_\Gamma^{\rm f}$ for some discrete group $\Gamma$), then there exists a natural isomorphism
	\[ 
	\KK^\sC(\Mod(A),\Mod(B)) \simeq \KK^{\hat{G}^{\rm op}}(A,B)
	\]
	by \cref{thm:RepCQG.KK} and \cref{prp:category.algebra.KK}. 
\end{ex}

We describe the Kasparov product in this picture. Again, let $\sA$, $\sB$, and $\sD$ be separable $\sC$-C*-categories and $\cE \colon \sA \to \sB$ and $\cF \colon \sB \to \sD$ be $\sC$-module functors. 
We say that a family $b = (b_X \in \cL_\sD(\cF (X)))_{X \in \Obj \sB}$ is a \emph{natural transformation up to compact operators} on $\cF$ if
	\[
	b_Y\cF(x)-(-1)^{|\cF x||b|}\cF(x)b_X \in\cK_{\sD}(\cF (X),\cF (Y)) \text{  for any } x \in\cK_{\sB}(X,Y).
	\] 
		
\begin{df}\label{def:connprod}
	Let $b=(b_X)_{X \in \Obj \sB}$ be as above. 
	A natural transformation up to compact $(\tilde{b}_X )_{ X \in \Obj \sA}$ on $\cF \circ \cE$ is a \emph{$b$-connection} if 
	\[
	\tilde{b}_X \cF(a)- (-1)^{|\cF a||b|}\cF(a)b_Z \in \cK_\sD(\cF(Z),\cF \cE(X))
	\]
	for any $Z\in \Obj \sB$, $X \in \Obj \sA$, and $a\in\cK_\sA(Z,\cE(X))$. 
\end{df}
		
\begin{df}
	Let $(\cE_1,F_1) \in \bfE^\cC(\sA,\sD)$ and $(\cE_2,F_2) \in \bfE^\sC(\sD,\sB)$. 
	Let $F_1 \# F_2$ denote the set of $F_2$-connections $F=(F_X)_{ X \in \Obj \sA}$ such that 
	$(\cE,F)\in\bfE^\sC(\sA,\sB)$ and 
	\begin{align*}
				\cE(a^*)[\cE_2(F_{1,X}),F_X]\cE (a) \geq 0 \text{ modulo $\cK_\sB(\cE(X))$}, 
			\end{align*}
			where $\cE:=\cE_2 \circ \cE_1$, for any $ X \in \Obj \sA$ and $ a \in \cK_\sA(X)$.
			We say a pair $(\cE ,F)$ for $F \in F_1 \# F_2$ is a Kasparov product of $(\cE_1,F_1)$ and $(\cE_2,F_2)$. 
		\end{df}
		
		\begin{ex}
			For a proper $\sC$-module functor $\cF \colon \sB \to \sD$, $(\cF,0)\in\bfE^\sC(\sB,\sD)$. 
			It is easy to see that the Kasparov product $(\cF, 0) \circ (\cE, F)$ for $(\cE,F) \in \bfE^\sC(\sA,\sB)$ is given by $(\cF \circ \cE, \cF(F))$. Here, $\cF(F)$ denotes the family of operators $\cF(F_X) \in \cL(\cF \circ \cE(X))$. 
			For a proper $\sC$-module functor $\cF \colon \sD \to \sA$, the Kasparov product $(\cE,F) \circ (\cF,0)$ is given by $(\cE \circ \cF, F)$.
		\end{ex}

		\begin{prop}
			Let $(\cE_1, F_1) \in \bfE^\sC(\Mod(A) , \Mod(D))$ and $(\cE_2, F_2) \in \bfE^\sC(\Mod(D), \Mod(B))$. Then the map
			\[ (\cE_1, F_1) \# (\cE_2,F_2) \to (\cE_1(A), F_{1,A}) \# (\cE_2(D), F_{2,D}), \quad (F_X)_{X} \mapsto F_A, \]
			is well-defined and induces a homotopy equivalence. 
		\end{prop}
		\begin{proof}
			The map is well-defined since $\cE_1(A) \hotimes_D \cE_2(D) \cong \cE(A)$ as $\sC$-Hilbert $A$-$B$ bimodules, which can be seen by applying \eqref{eqn:functor.to.bimodule} for $X=\cE_1(A)$. To see homotopy equivalence, it suffices to show that the map $\Psi_{A,D}$ in \cref{prp:category.algebra.KK} sends $(\cE_1(A), F_{1,A}) \# (\cE_2(D), F_{2,D})$ to $(\cE_1, F_1) \# (\cE_2,F_2)$. Indeed, by letting $E:=\cE(A)$, $E_1:=\cE_1(A)$, and $E_2:=\cE_2(D)$, the conditions are checked as
			\begin{align*}
				&(V_X^* \otimes 1_E)F_\infty (V_X \otimes 1_E) (a \otimes 1_{E_2}) \\
				=& (V_X^* \otimes 1_E)F_\infty (V_X \otimes 1_E) (a \otimes 1_{E_2}) (W_Z^* \otimes 1_{E_2})(W_Z \otimes 1_{E_2}) \\
				\equiv & (-1)^{|a|} (a \otimes 1_{E_2}) (W_Z^* \otimes 1_{E_2})F_{2,\infty }(W_Z \otimes 1_{E_2}) \quad \text{mod $\cK_B(X \otimes _A E)$},
			\end{align*}
			for $a\in\cK(Z,\cE_1(X))$, where $W_Z \colon Z \to \sH \otimes D$ is a fixed choice of embedding for each $Z \in\Obj\Mod(D)$, and for $a\in\cK(X)$,
			\begin{align*}
				&(a^* \otimes 1)[(V_{X}^* \otimes 1_{E_1})F_{1,\infty} (V_{X} \otimes 1_{E_1}) \otimes 1_{E_2}, (V_X^* \otimes 1_E)F_\infty (V_X \otimes 1_E)](a \otimes 1)\\
				=&(a^*V_X^* \otimes 1)[F_{1,\infty} \otimes 1_{E_2}, F_\infty](V_Xa \otimes 1) \geq 0 \quad \text{mod $\cK_B(X \otimes_A E)$. } \qedhere 
			\end{align*}
		\end{proof}
		
This, together with \cref{thm:Kasparov.product}, shows the well-definedness of the Kasparov product at the level of the module category picture. 
Furthermore, we get an equivalence of Kasparov categories. Let $\Kas^\sC_{\rm cat}$ denote the category whose objects are separable $\sC$-module categories, morphisms are $\KK^\sC$-groups, and the composition is given by the above Kasparov product. 
\begin{prop}\label{prop:Kasparov.category.algebra}
    There is an equivalence of categories $\Kas^\sC$ and $\Kas^\sC_{\rm cat}$ that is an extension of $\Corr^\sC_{\rm es} \cong \Ccat^\sC $ given in \cref{thm:equivalence.category.algebra}.   
\end{prop}
By this equivalence, $\Kas^\sC_{\rm cat}$ inherits nice properties of $\Kas^\sC$ such as \cref{lem:split_exact}.

	Finally, we observe the invariance of the $\KK^\sC$-theory under the monoidal equivalence of the tensor category (cf.~\cref{rem:monoidal.invariance}). 
	This will be extended to weak Morita invariance in \cref{section:Morita}. 
	\begin{prop}\label{prop:KK.monoidal.invariance}
	    A monoidal functor $(\cS , \mathsf{s}) \colon \sC \to \sD$ induces the pull-back homomorphism
	 \[
	 \cS^* \colon \KK^\sD (\sA,\sB) \to \KK^\sC (\cS^*\sA, \cS^*\sB),
	 \]
	 which is an isomorphism if $\cS$ is a monoidal equivalence.
	\end{prop}
	\begin{proof}
	    Following \cref{rem:monoidal.invariance}, we define as $\cS^* (\cE_\pm,F):=((\cE_+ \circ \cS ,\mathsf{v}_{+,\cS(\blank)}), (\cE_- \circ \cS ,\mathsf{v}_{-,\cS(\blank)}), F)$. 
	    Then $\cS^*$ sends a $\sD$-Kasparov $\sA$-$\sB$ bimodule to a $\sC$-Kasparov $\cS^*\sA $-$\cS^*\sB$ bimodule. It induces a homomorphism between KK-groups since it clearly preserves the homotopy. Moreover, if there is a inverse monoidal functor $\cT$ by the natural unitary $u \colon \id \to \cT \circ \cS$, then 
	    \[(\cE_\pm ,F) \mapsto \big( (\id_\sB, \{1 \otimes_\beta u_\pi\}_\pi) \circ (\cE_\pm ,\mathsf{v}) \circ (\id_\sA,\{1 \otimes_\alpha u_\pi\}_\pi )^{-1} , F\big) \]
	    induces a bijection $\bfE^\sC(\sA,\sB) \cong \bfE ^\sC(\cS^*\cT^*\sA, \cS^*\cT^*\sB )$ preserving the homotopy, and hence induces the isomorphism of $\KK^\sC$-groups.
	\end{proof}
	This, together with \cref{thm:RepCQG.KK}, recovers \cite{voigtBaumConnesConjectureFree2011}*{Theorem 8.5} claiming that a monoidal equivalence of the representation categories of two compact quantum groups $G_1$ and $G_2$ induces a categorical equivalence $\Kas^{G_1} \simeq \Kas^{G_2}$.

		\section{Cuntz's picture and the categorical aspect of \texorpdfstring{$\KK^\sC$}{KKC}-theory} \label{section:Cuntz.category}
		
		\subsection{\texorpdfstring{$\sC$}{C}-equivariant specialization}\label{subsection:equiv.stab}
		In this and the next subsection, we discuss a reduction of $\sC$-Kasparov bimodules to the one such that $F$ is trivial, following Cuntz's quasihomomorphism picture of $\KK$-theory. 
		As the first step, we produce a $\sC$-equivariant version of Meyer's specialization trick, given in ~\cite{meyerEquivariantKasparovTheory2000}*{Section 3}, replacing a $G$-Kasparov bimodule to the one whose Fredholm operator is $G$-invariant.  
		Its key idea is to replace a $G$-Kasparov $A$-$B$ bimodule with a $A \otimes \cK(\ell^2(G))$-$B$ bimodule through the Morita equivalence, and then average the Fredholm operator. 
		
		\begin{df}
			Let $A$, $B$ be $\sC$-C*-algebras such that $A$ is realized by endomorphisms. A $\sC$-Kasparov $A$-$B$ bimodule $(E, \phi, \bbmv,F)$ is said to be \emph{special} if the following hold; 
			\begin{enumerate}
                \item $F=F^*$ and $\| F\| \leq 1$, 
				\item each $\bbmv_\pi$ extends to a unitary $\bv_\pi \colon E\to E \otimes_B \beta_\pi$, for $\pi\neq \mathbf{0}_\sC$, and
				\item $F = \bv_\pi^* (F \otimes 1_{\beta_\pi}) \bv_\pi$ for any $\pi \in \Obj \sC\setminus\{\mathbf{0}_\sC\}$. 
			\end{enumerate}
		\end{df}
		We remark that the condition (2) is automatically satisfied if $(\sfE, F)$ is essential.

		Consider the objects
		\begin{align}
			\cH'_\sC:= \bigoplus_{\sigma \in \Irr \sC} \sigma ^{\oplus \infty}, \quad \cH_\sC:=(\cH_\sC')^{\oplus \infty} \in \Obj \tilde{\sC}, \label{eq:universal_repn}
		\end{align}
        where $\tilde{\sC}$ denotes the separable envelope (\cref{rem:separable.envelope}).
        We also fix a family of unitaries 
        \[
		U'_\pi \colon \cH'_\sC \otimes \pi \to \cH'_\sC \qquad \text{for $\pi \in \Obj \sC \setminus \{\mathbf{0}_\sC\}$}
		\]
		by rearranging direct summands of the codomain of $\bigoplus_\sigma U_{\sigma,\pi}^{\oplus\infty}$ (in the way that $U_{\mathbf{1}}=1$), and set $U_\pi:={U'_\pi}^{\oplus\infty}\colon \cH_\sC \otimes \pi \to \cH_\sC$. This choice of $\cH_{\sC}$ and $U_\pi$, as a stabilization of $\cH_{\sC}'$ and $U_{\pi}'$, will be used later in the proof of \cref{thm:quasihom.replacement}. 
		
		For a $\sC$-C*-algebra $A$, the image $\alpha (\cH_\sC)$ of $\cH_\sC$ by the functor $\alpha$ makes sense as a full Hilbert $A$-bimodule. As 
		\begin{align}
			\bU_{A,\pi}:=\fu(\cH_\sC, \pi)^*\alpha(U_\pi)^*  \colon  \alpha(\cH_\sC) \to \alpha(\cH_\sC) \otimes \alpha_\pi
		\end{align}
		is a unitary, we get the endomorphism $\sC$-C*-algebra 
		\begin{align} 
			\kC A = ( \cK(\alpha(\cH_\sC)),  \alpha_\pi^{\rm k} , \fu_{\pi, \sigma}^{\rm k}  ):= (\cK(\alpha(\cH_\sC)), \alpha_\pi^{\alpha(\cH_\sC)}, \fu_{\pi,\sigma}^{\alpha(\cH_\sC)})_{\bU_{A}^*\otimes 1_{\alpha(\cH_\sC)^*}}, 
		\end{align}
		by a cocycle twist (\cref{ex:cocycle.twist}) of the $\sC$-C*-algebra defined as in \eqref{eqn:module.cocycle}. 
		That is, 
		\begin{align*}
			\alpha _\pi^{\rm k}  &:= \Ad (\bU_{A,\pi}^*) \circ \alpha^{\alpha(\cH_\sC)}_\pi  \colon \cK(\alpha(\cH_\sC)) \to \cK(\alpha(\cH_\sC)), \\
			\fu_{\pi, \sigma}^{\rm k} &:=\bU_{A,\pi \otimes \sigma}^* \cdot (1 \otimes \fu_{\pi, \sigma} ^{\alpha(\cH_\sC)}) \cdot (\bU_{A,\pi} \otimes 1_{\alpha_\sigma})\bU_{A,\sigma} \in \cU( \cK(\alpha (\cH_\sC))). 
		\end{align*}
		The proper $\sC$-Hilbert $\kC A$-$A$ bimodule 
		\begin{align} 
			\sfH_{\sC,A}:=(\id_{\cK(\alpha(\cH_\sC))} , \bU_{A}) \otimes_{\cK(\alpha(\cH_\sC))}  (\alpha(\cH_\sC), m, \mathbbm{1}) \cong (\alpha(\cH_\sC), m, \bU_{A}),\label{eqn:hilbert.bimodule.stabilize}
		\end{align}
		where $m$ denotes the multiplication, gives a $\sC$-Morita equivalence of $\kC A$ and $A$ (cf.\ \cref{subsubsection:Morita}). 
		
		Let $(A, \alpha , \fu)$ and $(B, \beta, \fv)$ be $\sigma$-unital $\sC$-C*-algebras, and let $(\sfE,F)=(E,\phi,\bbmv, F)$ be a $\sC$-Kasparov $A$-$B$ bimodule. 
		Let 
		\[
		\phi_{\cH_\sC} \colon \cK(\alpha(\cH_\sC)) \to \cL(E \otimes \beta(\cH_\sC))
		\]
		be the $\ast$-homomorphism as \cref{lem:condbimod} (more precisely, it is defined as the inductive limit of $\phi_{\pi}$ for subobjects $\pi \leq \cH_\sC$ contained in $\sC$), which is characterized by the equality 
		\[
		\phi_{\cH_\sC} (x) |_{E_{\rm es} \otimes \beta(\cH_\sC)}= \bbmv_{\cH_\sC} (x \otimes 1)\bbmv_{\cH_\sC}^*.
		\]
		\begin{lem}\label{lem:specialization}
			Let $A$, $B$, and $(\sfE, F)$ be as above. Then the quadruple
			\begin{align}
				(\sfE, F)^{\rm s}= (E^{\rm s}, \phi^{\rm s}, \bbmv^{\rm s}, F^{\rm s}) :=
				( E \hotimes_B \beta(\cH_\sC), \phi_{\cH_\sC} , 1_E \otimes \bU_{B,\pi} ,  F \otimes 1_{\beta(\cH_\sC)})   \label{eqn:Kasparov.stabilized}
			\end{align}
			is a special $\sC$-Kasparov $\kC A$-$B$ bimodule, whose essential part is unitary equivalent to the Kasparov product $\sfH_{\sC,A} \hotimes_A  (\sfE, F)$ up to locally compact perturbation. 
		\end{lem}
		\begin{proof}
			Let $F' \in \cL(\sfH_{\sC,A} \hotimes _A E)$ be an odd self-adjoint $F$-connection.
			The Kasparov product $\sfH_A \hotimes_A (\sfE,F)$ is unitary equivalent to 
			\[ 
			(E_{\rm es} \hotimes \beta (\cH_\sC), \Ad(\bbmv_{\cH_\sC}) \circ (m \otimes_A 1_E), \bbmv_{\cH_\sC} \bbmv_\pi (\bU_{A,\pi} \otimes 1_E) \bbmv_{\cH_\sC}^* , \bbmv_{\cH_\sC} F'\bbmv_{\cH_\sC}^* )
			\]
			by the unitary $\bbmv_{\cH_\sC} \colon \alpha(\cH_\sC) \otimes_A E \to E_{\rm es} \otimes \beta(\cH_\sC)$. The cocycle relations, i.e., the diagrams in \cref{def:action.of.tensor.category.explicit} and \cref{defn:cocycle.hom}, show that
			\begin{align*}
				\bbmv_{\cH_\sC} \bbmv_\pi (\bU_{A,\pi} \otimes 1_E) 
				&=\bbmv_{\cH_\sC} \bbmv_\pi (\fu(\cH_\sC,\pi)^* \otimes 1_E)(  \alpha(U_\pi)^* \otimes 1_E) \\
				&=(1_E \otimes \fv(\cH_\sC, \pi)^*) \bbmv_{\cH_\sC \otimes \pi} (\alpha(U_\pi)^* \otimes 1_E) \\
				&= (1_E \otimes (\fv(\cH_\sC, \pi)^* \beta(U_\pi)^*)) \bbmv_{\cH_\sC} =(1_E \otimes \bU_{B,\pi}) \bbmv_{\cH_\sC} . 
			\end{align*}
			This shows that $\sfH_{\sC,A} \hotimes \sfE$ coincides with the essential part of $(E \hotimes _B \beta(\cH_\sC), \phi_{\cH_\sC}, 1_E \otimes \bU_{B})$. 
			
			Since $\bU_{B,\pi}$ is $B$-bilinear, $\Ad(1_E\otimes\bU_{B,\pi}^*)(F\otimes1_{\beta(\cH_\sC)\otimes\beta_\pi})=F\otimes1_{\beta(\cH_\sC)}$. 
			Finally, we show that $\bbmv_{\cH_\sC} F'\bbmv_{\cH_\sC}^*$ is an $(F \otimes 1_{\beta(\cH_\sC)})$-connection on 
			\[ \cK(\alpha(\cH_\sC)) \otimes (E \otimes \beta(\cH_\sC )) =(E \otimes \beta(\cH_\sC))_{\rm es} \cong  E_{\rm es} \otimes \beta(\cH_\sC),\]
			which will finish the proof by \cref{lem:KK.unit}. 
            For any subobject $\pi \leq \cH_\sC$ lying in $\Obj \sC$ and $\xi,\eta \in \alpha_\pi $, we have 
				\begin{align*}
					& \bbmv_{\cH_\sC} F' \bbmv_{\cH_\sC}^*\phi_{\cH_\sC}(\xi \otimes \eta^*)  -\phi_{\cH_\sC}(\xi \otimes \eta^*) (F \otimes 1_{\beta(\cH_\sC)}) \\
					= & \bbmv_{\cH_\sC} F'T_\xi (\bbmv_{\cH_\sC}T_\eta)^*  -\bbmv_{\cH_\sC} T_\xi (\bbmv_{\pi}T_\eta)^* (F \otimes 1_{\beta_\pi}) \\
					\equiv  & \bbmv_{\cH_\sC} T_\xi F (\bbmv_{\cH_\sC}T_\eta)^* -\bbmv_{\cH_\sC} T_\xi F (\bbmv_{\pi}T_\eta)^* = 0
				\end{align*}
	        modulo $\cK(E \otimes \beta(\cH_\sC), E_{\rm es} \otimes \beta(\cH_\sC))$. 
		\end{proof}

		\begin{prop}\label{prop:special.perturbation}
			Any essential $\sC$-Kasparov $\kC A$-$B$ bimodule $(\sfE , F)$ is a locally compact perturbation of a special one. 
		\end{prop}
		\begin{proof}
			Apply \cref{lem:specialization} to $(\sfE_0, F_0) := \sfH_{\sC,A} ^* \hotimes_{\kC A} (\sfE, F)$, where the $F$-connection $F_0$ is chosen to be an odd self-adjoint contraction. We get a special $\sC$-Kasparov $\kC A$-$B$ bimodule $(\sfE_0^{\rm s}, F_0^{\rm s})$ that is unitary equivalent to $\sfH_{\sC,A} \hotimes_A \sfH_{\sC ,A}^* \hotimes_{\kC A} (\sfE , F) \simeq (\sfE , F)$ up to locally compact perturbation. 
		\end{proof}
        We also remark that such a special perturbation is unique up to operator homotopy, since the set of operators $F$ making $(\sfE,F)$ special is a convex subset of $\cL(E)$.

		\begin{prop}\label{cor:specialization}
			Let $(\phi,\bbmv)\colon A\to B$ be a cocycle $\sC$-$*$-homomorphism of $\sigma$-unital $\sC$-C*-algebras. 
			Then $\kC(\phi,\bbmv) := (\phi_{\cH_\sC}, \mathbbm{1}) \colon \kC A\to \kC B$ is a $\sC$-$*$-homomorphism such that the $\sC$-Hilbert $\kC A $-$B$ bimodule $(\beta(\cH_\sC), \phi_{\cH_\sC} , \mathbbm{1}) $ satisfies
			\[ (\beta(\cH_\sC), \phi_{\cH_\sC} , \mathbbm{1})_{\rm es} \cong \sfH_{\sC,A}\otimes_A(B,\phi,\bbmv). \] 
		\end{prop}
		\begin{proof}
		By applying \cref{lem:specialization} to $(B,\phi,\bbmv,0)$, we obtain that $(\beta(\cH_\sC), \phi_{\cH_\sC} , \bU_{B,\pi})$ is a unitary $\sC$-Hilbert $\kC A$-$B$ bimodule, in other words, $(\phi_{\cH_\sC}, \bU_B)$ is a unitary cocycle $\sC$-$\ast$-homomorphism from $\kC A$ to $\cK(\beta(\cH_\sC))$. 
		Therefore, $(\phi_{\cH_\sC}, \mathbbm{1}) = (\id , \bU_B^*) \circ (\phi_{\cH_\sC}, \bU_B)$ is as well a unitary cocycle $\sC$-$\ast$-homomorphism from $\kC A$ to $\kC B$. 
		
		The latter claim is a special case of \cref{lem:specialization} for $(\sfE, F):=(B,\phi,\bbmv,0)$. 
		\end{proof}

\subsection{\texorpdfstring{$\sC$}{C}-quasihomomorphisms}
Hereafter, we write as $\sH_B:=\sH \otimes_\bC B$, where $\sH \cong \ell^2(\bZ)$ denotes the separable infinite dimensional Hilbert space.
\begin{df}
	Let $(A, \alpha, \fu)$ and $(B, \beta, \fv)$ be $\sC$-C*-algebras such that $A$ is realized by endomorphisms. 
	\begin{enumerate}
	    \item If $(\sH_B,\phi,\bv)$ forms a unitary $\sC$-Hilbert $A$-$B$ bimodule (in the sense of \cref{defn:trivial.cocycle} (1)), we call a pair $(\phi,\bv)$ a \emph{$\sC$-$\ast$-representation} of $A$ on $\sH_B$. 
	    \item A \emph{$\sC$-quasihomomorphism} from $A$ to $B$ consists of a triple $(\phi_+, \phi_-, \bv)$, or $(\phi_\pm, \bv)$ in short, if both $(\phi_+,\bv)$ and $(\phi_-,\bv)$ are $\sC$-$\ast$-representations of $A$ on $\sH_B$ and $\phi_+(a)-\phi_-(a) \in \cK(\sH_B)$ for any $a \in A$. 
	    We write $\qHom ^\sC(A,B)$ for the set of $\sC$-quasihomomorphisms from $A$ to $B$. 
	\end{enumerate}
	Two $\sC$-$\ast$-representations $(\phi_i,\bv_i)$ (resp.~$(\phi_{i,\pm},\bv_i) \in \qHom^\sC(A,B)$) are said to be homotopic if there is a $\sC$-$\ast$-representation $(\tilde{\phi},\tilde{\bv})$ of $A$ onto $\sH_{B[0,1]}$ (resp.~$(\tilde{\phi}_\pm,\tilde{\bv}) \in \qHom^\sC(A,B[0,1])$) whose evaluations at $t = i$ are identical with $(\phi_i,\bv_i)$ (resp.~$(\phi_{i,\pm},\bv_i)$) for $i=0,1$.
\end{df}

		\begin{rem}
			Let $(\phi_\pm, \bv )$ be a $\sC$-quasihomomorphism from $A$ to $B$ and let $(\psi ,\bw )$ be a unitary cocycle $\sC$-$\ast$-homomorphism from $B$ to $D$. Assume that $\psi$ is essential. Then the composition
			\[(\phi_\pm, \bv) \otimes_B (\psi, \bw) = (\psi,\bw) \circ (\phi_\pm, \bv):=( \phi_\pm \otimes_\psi 1_D,  (1\otimes_B\bw)(\bv\otimes_\psi 1_D ) )\]
			is also a cocycle $\sC$-quasihomomorphism. Similarly, for a unitary cocycle $\sC$-$\ast$-homomorphism $(\psi, \bw) $ from $D$ to $A$, the composition $(\psi,\bw) \otimes_A (\phi_\pm, \bv) $ is also a $\sC$-quasihomomorphism if $\phi_\pm$ are essential. 
		\end{rem}

		To a $\sC$-quasihomomorphism $(\phi_\pm, \bv)$, the special $\sC$-Kasparov $A$-$B$ bimodule 
		\begin{align}
			\Phi (\phi_\pm, \bv):=\bigg( \sH_B \oplus \sH_B^{\op}, \phi_+ \oplus \phi_-, \begin{pmatrix}\bv & 0 \\ 0 & \bv \end{pmatrix}, \begin{pmatrix}0 & 1 \\ 1 & 0 \end{pmatrix} \bigg)   \label{eq:KK_quasihom}
		\end{align}
		is associated. Since a homotopy of $\sC$-quasihomomorphisms is mapped to that of the associated $\sC$-Kasparov bimodules, this $\Phi $ gives a map 
		\[\Phi  \colon \qHom ^\sC (A,B)/\sim_{h} \, \to \KK^\sC(A,B), \]
        where $\sim_h$ denotes the homotopy equivalence relation. 
		
\begin{lem}\label{lem:quasihom_degeneration}
	Let $A$, $B$ be $\sC$-C*-algebras such that $A$ is realized by endomorphisms. 
	\begin{enumerate}
		\item Two $\sC$-$\ast$-representations $(\phi_0,\bv_0)$ and $(\phi_1, \bv_1)$ of $A$ on $\sH_B$ are homotopic if the essential part of $(\sH_B, \phi_0,\bv_0)$ and $(\sH_B, \phi_1,\bv_1)$ are equivalent. Moreover, if $B$ is also realized by endomorphisms and $\bv_0$, $\bv_1$ are trivial, then the homotopy is chosen to be with trivial cocycle.
		\item Two $\sC$-quasihomomorphisms $(\phi_{0,\pm},\bv_0)$ and $(\phi_{1,\pm},\bv_1)$ are homotopic if the essential part of $\Phi (\phi_{0,\pm},\bv_0)$ and $\Phi (\phi_{1,\pm},\bv_1)$ are unitarily equivalent up to locally compact perturbation.
	\end{enumerate}
\end{lem}
\begin{proof}
	We first show (1). 
	By the Kasparov stabilization theorem, there is a unitary $s \colon C([0,1] , \sH) \oplus C_0((0,1],\sH) \to C([0,1],\sH)$, corresponding to the family $s_t:=s \otimes_{\eval_t}1$. By taking the product with $s_0^*$ if necessary, we may choose it as $s_0 = 1_\sH$. 
    Set $S_t:=s_t \otimes \id_B$. 
    Then, for $i=0,1$ and any family $\bw=(\bw_\pi)_{\pi\in\Obj\sC\setminus\{\mathbf{0}\}}$ of unitaries $\bw_\pi\colon \sH_B\to \sH_B\otimes\beta_\pi$, we get a homotopy 
    \[ 
    \bigr(\phi_i, \bv_i) \sim (\Ad(S_1) \circ (\phi_i \oplus 0), (S_1\otimes 1_{\beta})(\bv_i \oplus \bw)S_1^* \bigl).
    \] 

	Let $W \colon \overline{\phi_{0}(A)\sH_B} \to \overline{\phi_{1}(A)\sH_B}$ be a unitary intertwiner. 
	Then
	\[
	\tilde{\phi}
	\begin{pmatrix}
		a_{11} & a_{12} \\
	    a_{21} & a_{22}
	\end{pmatrix}
	= 
	S_1\begin{pmatrix} 
		\phi_{0}(a_{11}) &  W^* \phi_{1}(a_{12}) \\ 
		W \phi_{0}(a_{21}) & \phi_{1}(a_{22}) 
	\end{pmatrix}S_1^*,
	\quad \tilde{\bv}_\pi = (S_1 \otimes1_{\beta_{\pi}}) \begin{pmatrix} \bv_{0,\pi} & 0 \\ 0 & \bv_{1,\pi} \end{pmatrix}S_1^*
	\]
	forms a unitary cocycle $\sC$-$\ast$-homomorphism from $\bM_2(A)$ to $\cL(\sH_B)$. Here, note that $W\phi_0(a)\colon \sH_B\to \sH_B$ is adjointable with $(W\phi_0(a))^*=W^*\phi_1(a^*)$ for $a\in A$. 
	Composed with the corner embeddings $A \to \bM_2(A)$, we get a homotopy connecting $(\Ad(S_1) \circ (\phi_{0} \oplus 0), (S_1 \otimes 1_\beta)(\bv_0 \oplus \bv_1)S_1^*)$ and $(\Ad(S_1) \circ (0 \oplus \phi_{1}), (S_1 \otimes 1_\beta) (\bv_0 \oplus \bv_1)S_1^*)$. This finishes the proof. 
	Indeed, the last claim follows from the construction, since both $S_t$ and the rotation matrix $R_t \in \bM_2 \subset \bM_2(\cM(A))$ are $\sC$-invariant (cf.~\cref{ex:inner.auto}). Here, the former for $t \in (0,1]$ precisely means that $\beta_\pi(S_t)=S_t \in \cL_B(\sH_B^{\oplus 2}, \sH_B)$ for any $\pi \in \Obj \sC \setminus \{ \mathbf{0}\}$.

				For (2), consider the above construction for both $(\phi_{i,+}, \bv_i)$ and $(\phi_{i,-}, \bv_i)$ simultaneously by using a unitary intertwiner $W=W_+ \oplus W_- $ of Kasparov bimodules up to locally compact perturbation. 
				Then the resulting $\sC$-$\ast$-homomorphisms $(\tilde{\phi}_\pm, \tilde{\bv})$ form a $\sC$-quasihomomorphism from $\bM_2(A)$ since 
				\[ 
				W_+ \phi_{0,+}(a) - W_- \phi_{0,-}(a) \in \cK(\sH_B), \quad W_+^* \phi_{1,+}(a) - W_-^* \phi_{1,-}(a) \in \cK(\sH_B),
				\]
                which proves the claim. 
				Indeed, for $i=0,1$, an $\big( \begin{smallmatrix}0 & 1 \\ 1 & 0 \end{smallmatrix}\big)$-connection $\Big( \begin{smallmatrix}0 & F_{i}^* \\ F_{i} & 0 \end{smallmatrix}\Big) $ of the essential part $(A,\id,\mathbbm{1}) \hotimes_A \Phi(\phi_{i,\pm},\bv_i)$ satisfies $\phi_{i,-}(a) - F_{i} \phi_{i,+}(a), \phi_{i,+}(a) - F_{i}^* \phi_{i,-}(a) \in \cK(\sH_B)$. Thus, for any $a,b \in A$, we get
                \[ W_+ \phi_{0,+}(ab) \equiv W_+ F_{0}^* \phi_{0,-}(ab) = F_1^* \phi_{1,-}(a)W_-\phi_{0,-}(b) \equiv \phi_{1,-}(a)W_-\phi_{0,-}(b) = W_- \phi_{0,-}(ab) \]
                modulo $\cK(\sH_B)$. The latter is checked in the same way. 
			\end{proof}

		\begin{prop}\label{lem:quasihom_represent}
			For any separable $\sC$-C*-algebras $A$ and $B$, the map $\Phi$ in \eqref{eq:KK_quasihom} for the pair $\kC A$, $B$ is an isomorphism.
		\end{prop}

\begin{proof}
	We construct the inverse $\Psi$ of $\Phi$. 
	We fix a family $\bw = (\bw_\pi)_\pi$ of unitaries $\bw_\pi \colon \sH_B \to \sH_B \otimes \beta_\pi $. Let $(\sfE, F) \in \bfE^\sC ( \kC A,B)$. 
	By \cref{prop:special.perturbation}, we may choose its essential part $(\sfE_{\rm es}, F_{\rm es})$ as a special $\sC$-Kasparov bimodule. 
	Moreover, we may assume $\| F_{\rm es} \| \leq 1$ by \cref{rem:normalization.F}. 
	Let 
	\begin{align*}
		(E_1,\phi_1,\bv_1, F_1):=& (\sfE_{\rm es}, F_{\rm es}) \oplus \bigg( \sH_B\oplus\sH_B^{\op} ,0,\begin{pmatrix} \bw & 0 \\ 0 & \bw \end{pmatrix} ,0 \bigg), \\
		(E_2,\phi_2,\bv_2, F_2):=& \bigg( E_1 \oplus E_1^{\rm op}, \phi_1 \oplus 0, \begin{pmatrix}\bv_1 & 0 \\ 0 &  \bv_1\end{pmatrix}, \begin{pmatrix}F_1 & (1-F_1^2)^{1/2} \\  (1-F_1^2)^{1/2} & -F_1 \end{pmatrix} \bigg)_{\textstyle ,}\\
		(E_3,\phi_3,\bv_3, F_3):=& (E_2,\phi_2,\bv_2,F_2) \oplus (E_2^{\rm op}, 0, \bv_2, -F_2).
	\end{align*}
	Then $(\sfE_3, F_3)$ is a locally compact perturbation of the sum of $(\sfE_{\rm es} , F_{\rm es} )$ and degenerate $\sC$-Kasparov $\kC A$-$B$ bimodules. 
	Since $E_1$ contains $\sH_B \oplus \sH_B^{\rm op}$ as a direct summand, the Kasparov stabilization theorem shows that there is an even unitary 
	\[ T \colon E_3 \cong E_2 \oplus E_2^{\rm op} \to \sH_B \oplus \sH_B^{\rm op}\]
	such that $T \big( \begin{smallmatrix} 0 & 1 \\ 1 & 0 \end{smallmatrix}\big)T^* = \big( \begin{smallmatrix} 0 & 1 \\ 1 & 0 \end{smallmatrix}\big)$, which is unique up to multiplication with $\diag(U,U) \in \cU(\sH_B \oplus \sH_B^{\rm op})$. 
	Hence we get a unitary equivalence 
	\begin{align*}
		(\sfE_3, F_3) \simeq & \bigg( E_2 \oplus E_2^{\rm op}, \Ad(U_{F_2}) \circ \phi_3, U_{F_2}\begin{pmatrix}\bv_2 & 0 \\ 0 & \bv_2 \end{pmatrix}U_{F_2}^*, \begin{pmatrix}0 & 1 \\ 1 & 0 \end{pmatrix} \bigg) \\
		\simeq & \bigg( \sH_B \oplus \sH_B^{\rm op} , \Ad(TU_{F_2}) \circ \phi_2,TU_{F_2} \begin{pmatrix}\bv_2 & 0 \\ 0 & \bv_2 \end{pmatrix}U_{F_2}^* T^*,  \begin{pmatrix} 0 & 1 \\ 1 & 0 \end{pmatrix} \bigg)
	\end{align*}
	where $U_{F_2}:=\frac{1}{\sqrt{2}}\big( \begin{smallmatrix} 1 & -F_2 \\ F_2 & 1 \end{smallmatrix}\big)$. Since this $\sC$-Kasparov bimodule is special, we have
	\[ T U_{F}\diag(\bv_2, \bv_2)U_{F}^* T^* = \diag(\bv_4, \bv_4) \]
	for some unitary cocycle $\bv_4 $, and hence the right hand side is of the form 
	\[
	(\sfE_3, F_3):=
	\bigg( \sH_B \oplus \sH_B^{\rm op} , \phi_{4,+} \oplus \phi_{4,-}, \begin{pmatrix} \bv_4 & 0 \\ 0 & \bv_4 \end{pmatrix},  \begin{pmatrix} 0 & 1 \\ 1 & 0 \end{pmatrix} \bigg) = \Phi (\phi_{4,\pm}, \bv_4) ,\]
	where $\phi_{4,\pm}$ are the even and odd parts of $\Ad (TU_{F_1}) \circ \phi_3$. Then $\Psi([\sfE,F]):=[\phi_{4\pm},\bv_4]$ determines a well-defined map $\Psi \colon \KK^\sC(A,B) \to \qHom^\sC (\kC A ,B)/\sim_h$, which is a right inverse of $\Phi$. 
	Indeed, the homotopy class of the above $(\sfE_3,F_3)$ does not depend on the choice of $F_{\rm es}$ and $T$ (by the remark below \cref{prop:special.perturbation}). Moreover, 
	if $(\widetilde{\sfE}, \widetilde{F}) \in \bfE^\sC(\kC A,B[0,1])$ is a homotopy of $(\sfE,F)$ and $(\sfE',F)$, then $\Psi(\widetilde{\sfE}, \widetilde{F})$ gives a homotopy connecting $\Psi(\sfE,F)$ and $\Psi(\sfE',F')$, again by the remark below \cref{prop:special.perturbation}. 
			
	The remaining task is to show $\Psi \circ \Phi = \id$. For $(\phi_\pm, \bv) \in \qHom (\kC A,B)$, let $(\sfE, F):= \Phi(\phi_\pm, \bv)$. 
	Starting from it, we construct $(\sfE_1, F_1)$, $(\sfE_2, F_2)$, $(\sfE_3, F_3)$, and $(\phi_{4,\pm},\bv_4)$ as in the above paragraph. 
	Then, by construction, the essential part of $(\sfE_3, F_3)=\Phi (\phi _{4,\pm}, \bv_4)$ is unitary equivalent to $(\sfE_{\rm es}, F_{\rm es})$, i.e., the essential part of $\Phi(\phi_\pm, \bv)$. Therefore, by \cref{lem:quasihom_degeneration}, we obtain that $(\phi_\pm, \bv)$ and $(\phi_{4, \pm}, \bv_4)$ are homotopic. This finishes the proof.  
		\end{proof}
  
		\subsection{Free product over \texorpdfstring{$\sC$}{C}}
		\begin{df}
			Let $(A,\alpha,\fu)$ and $(B,\beta,\fv)$ be endomorphism $\sC$-C*-algebras. Their \emph{free product over $\sC$} is the quotient C*-algebra 
			\[A \ast_\sC B:= (A \ast B)/ I_\sC ,  \]
			where $I_\sC$ is the C*-ideal generated by the elements of the form 
			\begin{align}
				(a \fu_{\pi, \sigma}) \ast b - a \ast (\fv_{\pi, \sigma}b), \quad (a\alpha_f)  \ast b - a \ast  (\beta_f b), \label{eqn:ideal.generators}
			\end{align}
			for $a\in A$, $b \in B$, $\pi, \sigma \in \Obj \sC$ and $f \in \Hom(\pi, \sigma)$. 
		\end{df}
		The C*-algebra $A \ast_\sC B$ is the completion of the algebraic free product over $\sC$, i.e., the $\ast$-algebra $ A \ast_{\alg, \sC} B:= A \ast_\alg B / I_{\alg, \sC}$ where $I_{\alg, \sC}$ is the algebraic $*$-ideal generated by the elements of the form \eqref{eqn:ideal.generators}. Note that $(a \fu_{\pi, \sigma}^*) \ast b - a \ast (\fv_{\pi, \sigma}^*b) = (a \fu_{\pi, \sigma}^*) \ast (\fv_{\pi, \sigma}\fv_{\pi, \sigma}^*b) - (a \fu_{\pi, \sigma}^*\fu_{\pi, \sigma}) \ast (\fv_{\pi, \sigma}^*b) \in I_{\alg,\sC}$.
		
		\begin{lem}
			Let $(A,\alpha,\fu)$ and $(B,\beta,\fv)$ be endomorphism $\sC$-C*-algebras. \begin{enumerate}
				\item For $\pi \in \Obj \sC$, $\alpha_\pi \ast \beta_\pi \colon A \ast B \to A \ast B $ induces a $\ast$-endomorphism 
				\[ 
				(\alpha \ast_\sC \beta)_\pi \colon A \ast_\sC B \to A \ast_\sC B.
				\]
				\item The linear map $(\alpha \ast_\sC \beta)_f$ on $ A \ast_{\alg,\sC} B$ defined by 
				\[
				(\alpha \ast_{\alg, \sC} \beta)_{f}(x)=
				\begin{cases}
					(\alpha_f a) y & \text{if $x=ay$, where $a \in A$ and $y \in A \ast_{\alg, \sC} B$}, \\
					(\beta_f b)y & \text{if $x=by$, where $b \in B$ and $y \in A \ast_{\alg, \sC} B$,}
				\end{cases}
				\]
				extends to $(\alpha \ast_\sC \beta)_{f} \in \cM(A \ast_\sC B)$.
				\item The linear map $(\fu \ast_\sC \fv)_{\pi, \sigma} \colon A \ast_{\alg, \sC} B \to A \ast_{\alg, \sC} B$ defined by 
				\[
				(\fu \ast_{\alg , \sC} \fv)_{\pi, \sigma}(x)=
				\begin{cases}
					(\fu_{\pi, \sigma}a) y & \text{if $x=ay$, where $a \in A$ and $y \in A \ast_{\alg, \sC} B$}, \\
					(\fv_{\pi, \sigma}b)y & \text{if $x=by$, where $b \in B$ and $y \in A \ast_{\alg, \sC} B$,}
				\end{cases}
				\]
				extends to a unitary $(\fu \ast \fv)_{\pi, \sigma} \in \cM(A \ast_\sC B)$.
				\item The triple $(A \ast_\sC B, \alpha \ast_\sC \beta, \fu \ast_\sC \fv)$ forms a $\sC$-C*-algebra.
			\end{enumerate}
		\end{lem}
		\begin{proof}
			For (1), it suffices to show that $\alpha_\pi \ast \beta_\pi(I_\sC) \subset I_\sC$.  This holds since
			\begin{align*}
				&\alpha_\pi (a \fu_{\sigma, \rho}) \ast \beta_\pi (b) 
				-\alpha_\pi(a) \ast \beta_\pi(\fv_{\sigma, \rho} b) \\
				{=} & \alpha_\pi(a) \fu_{\sigma \otimes \rho,\pi}^* \alpha_{\ass(\sigma,\rho,\pi)}^*\fu_{\sigma, \rho\otimes\pi} \fu_{\rho,\pi} \ast \beta_\pi(b) -  \alpha_\pi(a ) \ast \fv_{\sigma \otimes \rho,\pi}^* \beta_{\ass(\sigma,\rho,\pi)}^*\fv_{\sigma, \rho\otimes\pi} \fv_{\rho,\pi} \beta_\pi(b), \\
				&\alpha_\pi (a\alpha_f) \ast \beta_\pi(b) - \alpha_\pi (a) \ast \beta_\pi(\beta_f b) \\
				{=}& \alpha_\pi(a)\fu_{\rho,\pi}^* \alpha_{f \otimes 1} \fu_{\sigma,\pi} \ast \beta_\pi(b) - \alpha_\pi(a) \ast \fv_{\rho,\pi}^* \beta_{f \otimes 1} \fv_{\sigma,\pi} \beta_\pi(b),
			\end{align*}
			are contained in $I_\sC$ for any $a \in A$, $ b \in B$, $\pi, \sigma, \rho \in \Obj \sC$ and $f \in \Hom_\sC (\sigma, \rho)$.

			Next, we show (2) and (3). Since $(\fu \ast_{\alg , \sC} \fv)_{\pi, \sigma}$ is a unitary on $A \ast_{\alg, \sC} B$, it is bounded and extends to a unitary on the C*-completion. 
			In the same way, if $f \in \Hom(\pi,\sigma)$ is a partial isometry, then $(\alpha \ast_{\alg, \sC} \beta)_f$ also satisfies the relation of partial isometry (i.e., $(v^*v)^2=v^*v$ and $(vv^*)^2=vv^*$), and hence is bounded. For general $f$, consider the polar decomposition $f = s|f|$ and write $|f|$ as a linear combination of two unitaries.

			Finally, the relations of \cref{def:Caction.endo} are all satisfied on the dense subalgebra $A \ast _{\alg, \sC} B$, which concludes (4). 
		\end{proof}
		
		\begin{prop}\label{prop:universality.free.product}
			Let $A_1$, $A_2$, and $B$ be $\sC$-C*-algebras with $A_1$ and $A_2$ being realized by endomorphisms. 
			Let $(E,\phi_i, \bv)$ be a $\sC$-Hilbert $A_i$-$B$ bimodule for $i=1,2$ that shares the underlying Hilbert module $E$ and unitary cocycles $\{\bv_\pi\}$. 
			Then, there is a unique $*$-homomorphism 
			\[ \phi_1 \ast_\sC \phi_2 \colon A_1\ast_\sC  A_2 \to \cL(E)\]
			that makes $(E,\phi_1 \ast_\sC \phi_2,\bv)$ a $\sC$-Hilbert $A_1 \ast_\sC  A_2$-$B$-bimodule with unitary cocycles, and coincides with $\phi_i$ on the C*-subalgebra $A_i$.  
			Moreover, this Hilbert bimodule is proper if so are both $\sfE_1$ and $\sfE_2$. 
		\end{prop}
		
		\begin{proof}
			We write $(\alpha_i,\fu_i)$ and $(\beta,\fv)$ for the $\sC$-actions on $A_i$ and $B$. The $*$-homomorphism $\phi_1 \ast_\sC \phi_2$ is induced from $\phi_1 \ast \phi_2 \colon A_1\ast A_2\to \cL(E)$ since it vanishes on $I_\sC$. 
			Indeed, for any $\pi,\rho\in\Obj\sC$, $a \in A_1$, $b\in A_2$, $\xi,\eta\in E$, and $f\in\Hom_\sC(\pi, \rho)$, we see 
			\begin{align*}
				&
				\langle \phi_1(\fu_{1,\pi,\rho}a)\xi, \phi_2(b)\eta\rangle 
				= \langle \bv_{\pi\otimes\rho}\phi_1(\fu_{1,\pi,\rho}a)\xi, 
				\bv_{\pi\otimes\rho}\phi_2(b) \eta \rangle
				\\
				={}& \langle (1_E\otimes_B \fv_{\pi,\rho}) \bv_{\pi\otimes\rho}\phi_1(a)\xi, 
				(1_E\otimes_B \fv_{\pi,\rho}) \bv_{\pi\otimes\rho}\phi_2((\fu_{2,\pi,\rho})^*b)\eta\rangle
				\\
				={}&
				\langle \phi_1(a)\xi, \phi_2(\fu_{2,\pi,\rho}^* b)\eta\rangle, \\
				&
				\langle \phi_1(\alpha_{1,f} a)\xi, \phi_2(b)\eta\rangle 
				= \langle \bv_{\rho}\phi_1(\alpha_{1,f} a)\xi, 
				\bv_{\rho}\phi_2(b)\eta\rangle
				\\
				={}& \langle (1_E\otimes_B \beta_{f}) \bv_{\pi} \phi_1(a)\xi, 
				\bv_{\rho}\phi_2(b)\eta \rangle
				=
				\langle \bv_{\pi}\phi_1(a)\xi, 
				(1_E\otimes_B \beta_{f^*}) \bv_{\rho}\phi_2(b)\eta\rangle
				\\
				={}&
				\langle \bv_{\pi}\phi_1(a)\xi, 
				\bv_{\pi}\phi_2(\alpha_{2,f^*}b)\eta\rangle
				=
				\langle \phi_1(a)\xi, \phi_2(\alpha_{2,f^*} b)\eta\rangle. 
			\end{align*}
			Uniqueness follows from the construction, and now it is a routine work to show that $(E,\phi_1 \ast_\sC \phi_2, \bv)$ forms a $\sC$-Hilbert $A_1\ast_\sC A_2$-$B$-bimodule. 
			Finally, if $\phi(A_1),\phi_2(A_2)\subset \cK_B(E)$, then the range of $\phi_1 \ast_\sC \phi_2$ is also contained in $\cK_B(E)$. 
		\end{proof}

		\begin{prop}\label{prop:KK.free.product}
			There is a $\KK^\sC$-equivalence $A \oplus B \sim A \ast_\sC B$. 
		\end{prop}
		\begin{proof}
			This is proved in the same way as the non-equivariant case \cite{cuntzNewLookKK1987}*{Proposition 3.1}.
			Let $\iota_A \colon A \to A \ast_\sC B$ and $\iota_B \colon B \to A \ast_\sC B$ be inclusions, forming $\sC$-$\ast$-homomorphisms. 
			Let $\mathrm{pr}_A \colon A \oplus B \to A$ and $\mathrm{pr}_B \colon A \oplus B \to B$ denote the projections, and $j_A \colon A \to A \oplus B$ and $j_B \colon B \to A \oplus B$ denote the inclusions. 
			By \cref{prop:universality.free.product}, we get a $\sC$-$\ast$-homomorphism
			\[ (j_A \ast_\sC j_B , \mathbbm{1}) \colon A \ast_\sC B\to A \oplus B. \]
			Then each of the compositions 
			\begin{align*}
				\diag(\iota_A \circ \mathrm{pr}_A, \iota_B\circ \mathrm{pr}_A) \circ (j_A \ast _\sC j_B) \colon& A \ast_\sC B \to A \oplus B \to \bM_2(A \ast_\sC B),\\
				\bM_2 (j_A \ast _\sC j_B, \mathbbm{1}) \circ \diag(\iota_A\circ \mathrm{pr}_A, \iota_B\circ \mathrm{pr}_A) \colon& A \oplus B \to \bM_2(A \ast_\sC B) \to \bM_2(A \oplus B), 
			\end{align*}
			are homotopic to the embeddings onto the upper left corner of $\bM_2$. Hence $j_A \ast _\sC j_B$ and $\diag(\iota_A\circ \mathrm{pr}_A, \iota_B\circ \mathrm{pr}_A)$  are mutually $\KK^\sC$-inverse. 
		\end{proof}

		\begin{df}
			For an endomorphism $\sC$-C*-algebra $A$, let $\rQ_\sC A:=A \ast_\sC A$, and let $\mrq_\sC A$ denote the kernel of the $\ast$-homomorphism $\id_A \ast _\sC \id_A \colon \rQ_\sC A \to A$. 
		\end{df}
		
		Since $\mrq_\sC A$ is the kernel of a cocycle $\sC$-$\ast$-homomorphism $(\id_A \ast _\sC \id_A , \mathbbm{1})$, it is $\sC$-invariant, and hence is endowed with the $\sC$-C*-algebra structure induced from that of $A \ast_\sC A $ (cf.\ \cref{subsubsection:ideal.exact}).
		
\begin{lem}\label{lem:qA.KKequiv}
	For any separable $\sC$-C*-algebra $A$, let $\iota_i \colon A \to \rQ_\sC A\to \cM(\mrq_\sC A)$ denote the canonical inclusion
    to the $i$-th component (for $i=1,2$). Then the $\sC$-Kasparov $A$-$\qC A$ bimodule 
	\begin{align*} 
		\Phi(\iota_1,\iota_2,\mathbbm{1}) =\bigg( \mrq_\sC  A \oplus \mrq _\sC A ^{\rm op}, \iota_1 \oplus \iota_2, \diag (\mathbbm{1}, \mathbbm{1}), \begin{pmatrix}0 & 1 \\ 1 & 0 \end{pmatrix} \bigg)		
	\end{align*}
	(where we follow \eqref{eq:KK_quasihom} by abuse of notation)
 induces a $\KK^\sC$-equivalence.
\end{lem}
\begin{proof}
	Let $j_i \colon A \to A \oplus A$ denote the inclusion to $i$-th direct summand. 
	By \cref{prop:KK.free.product}, $j_1 \ast_\sC j_2$ gives a $\KK^\sC$-equivalence of  $\rQ_\sC A$ and $A \oplus A$. 
	Moreover, since the exact sequence $0 \to \qC A \to \rQ _\sC A \to A \to 0$ splits $\sC$-equivariantly by $\iota_2$, the element $[\pr_1] \circ [(j_1 \ast_\sC j_2)|_{\qC A}] \in \KK^\sC(\qC A,A)$ is a $\KK^\sC$-equivalence. 
	On the other hand, we have
	\begin{align*}
		&[\pr_1] \circ [(j_1 \ast_\sC j_2)|_{\qC A}] \circ \bigg[ \mrq_\sC  A \oplus \mrq_\sC A ^{\rm op}, \iota_1 \oplus \iota_2, \diag (\mathbbm{1}, \mathbbm{1}), \begin{pmatrix}0 & 1 \\ 1 & 0 \end{pmatrix} \bigg]  \\
		=&[\pr_1] \circ  \bigg[ A \oplus A^{\rm op}, j_1 \oplus j_2, \diag (\mathbbm{1}, \mathbbm{1}), \begin{pmatrix}0 & 1 \\ 1 & 0 \end{pmatrix} \bigg] = [\id_A] \in \KK^\sC (A , A).  
	\end{align*}
	This finishes the proof. 
\end{proof}

		For endomorphism $\sC$-C*-algebras $A,B$, let $[A,B]^\sC$ (resp.~$[A,B]^\sC_{\rm triv}$) denote the set of homotopy classes of cocycle $\sC$-$\ast$-homomorphisms with unitary (resp.~trivial) cocycles from $A$ to $B$. 
        Here, the homotopy equivalence relation $\phi_0 \simeq \phi_1$ is given by a cocycle $\sC$-$*$-homomorphism $\phi \colon A\to B[0,1]$ (resp.~a $\sC$-$*$-homomorphism) such that $\eval_i \circ \phi=\phi_i$ for $i=0,1$. 
		Note that there are canonical maps $[A,B]_{\rm triv}^\sC \to [A,B]^\sC\to \KK^{\sC}(A,B)$. 
  
        \begin{thm}\label{thm:quasihom.replacement}
			Let $A$ and $B$ be $\sC$-C*-algebras. 
        \begin{enumerate}
            \item There is a bijection between $[\qC \kC A, \kC B]^\sC $ and $\qHom ^\sC(\kC A,B)/\sim_h \cong \KK^\sC(\kC A,B)$. 
            \item The canonical map  $\KK^\sC \colon [\kC \qC \kC A, \kC B]^\sC_{\rm triv} \to \KK^\sC(\kC \qC \kC A,\kC B) \cong \KK^\sC(A,B)$ sending $[\psi,\mathbbm{1}]$ to its $\KK^\sC$-class is an isomorphism. 
        \end{enumerate}
		\end{thm}
\begin{proof}
    We first show (1). We show that $\qHom (A,B)/\sim_h \cong [\qC A, \kC B]$ if $A$ is realized by endomorphisms, and then apply it to $\kC A$ and $B$ in which case \cref{lem:quasihom_represent} can be applied. 
	The map $\Theta \colon \qHom^\sC (A,B) \to \Hom^\sC (\qC A, \kC B)$ is given by using \cref{ex:cocycle.twist} and \cref{prop:universality.free.product} as
	\[ 
	\Theta (\phi_\pm, \bv) := (\id, \bU_B^*) \circ \big( (\phi_+ \ast _\sC \phi_-)|_{\qC A}, \bv \big).
	\]
	Here, we apply $\bU_B$ to $\sH_B$ by fixing an identification $\sH_B \cong \beta(\cH_\sC)$ as Hilbert $B$-modules. 
	Note that the restriction of $(\phi_+ \ast _\sC \phi_-)$ to $\qC A$ takes value in $\cK(\beta(\cH_\sC))$. 

	The inverse $\Xi \colon \Hom^\sC (\qC A, \kC B) \to \qHom^\sC (A,B)$ is given by 
	\[ \Xi (\psi, \bw) := (\Ad (U) \circ (\psi \circ \iota_1  \oplus 0),\Ad (U) \circ (\psi \circ \iota_2  \oplus 0) , U(\diag (\bw, \mathbbm{1}) U^* ), \]
	where $E:=\phi(\qC A ) \cdot \beta(\cH_\sC)$ and $U \colon E \oplus \beta(\cH_\sC) \to \beta(\cH_\sC)$ is a unitary taken by the Kasparov stabilization theorem. Note that $\Xi (\psi,\bw)$ is well-defined only up to unitary equivalence coming from the ambiguity of the choice of $U$. 
	
	The equivalence $\Xi \circ \Theta (\phi_{\pm},\bv)\sim_{h} (\phi_{\pm},\bv)$ follows from \cref{lem:quasihom_degeneration}. Similarly, $\Theta \circ \Xi (\psi,\bw) \sim_{h} (\psi,\bw)$ also follows from the same homotopy, considered for $\sC$-$\ast$-homomorphisms from $\qC A$ to $\cK(\beta (\cH_\sC))$.

    To see (2), we first give an isomorphism similar to (1). Pick a unitary $V \colon \sH \otimes \cH_\sC \to \cH_\sC$ such that $U_\pi V = V (1 \otimes U_{\pi})$ for any $\pi$ by using $\cH_\sC \cong (\cH_\sC')^{\oplus \infty}$ and $U_\pi = (U_\pi')^{\oplus \infty}$. 
    Then, $\beta (V) \in \cL(\beta(\cH_\sC))$ satisfies $ \beta(V) = \bU_{B,\pi}^*(\beta(V) \otimes 1_{\beta_\pi}) (1_\sH \otimes \bU_{B,\pi})$ for any $\pi \in \Obj \sC \setminus \{ \mathbf{0}\}$, and hence $(\Ad (\beta(V)), \mathbbm{1})$ forms a $\sC$-$\ast$-isomorphism of $\cK \otimes \kC B \to \kC B$ by \cref{ex:inner.auto}. 
    Now, the map $\Theta_{\rm triv} \colon \qHom^\sC (\kC A,B) \to \Hom^\sC _{\rm triv}(\kC \qC \kC A , \kC B)$ is given (where $\Hom^\sC_{\rm triv}(\blank , \blank)$ denotes the set of $\sC$-$\ast$-homomorphisms) by using \cref{cor:specialization} as
	\[ 
	\Theta_{\rm triv} (\phi_\pm, \bv) := (\Ad \beta(V) , \mathbbm{1}) \circ \kC\big( (\phi_+ \ast _\sC \phi_-)|_{\qC\kC A}, \bv \big).
	\]
    The map $\Xi_{\rm triv}$ of converse direction is defined as
	\[ \Xi_{\rm triv}(\psi, \mathbbm{1}) := \Xi ((\id,\bU_B) \circ (\psi,\mathbbm{1}) \circ (\id , \bU_{\qC\kC A}^*) \circ (\iota, \mathbbm{1})) , \]
	where $(\iota , \mathbbm{1}) \colon \qC \kC A \to \cK( \sH_{\qC \kC A} ) $ denote the full corner embedding into a rank $1$ subspace. 
	Again by \cref{lem:quasihom_degeneration} (1), they are inverse to each other up to homotopy. 
	
	Now, the desired isomorphism follows from
	\begin{align*}
	    	&\Phi(\iota_1,\iota_2,\mathbbm{1}) \hotimes_{\qC \kC A} \sfH_{\sC, \qC \kC A}^* \hotimes_{\kC \qC \kC A} \big( \KK^\sC \circ \Theta_{\rm triv} (\phi_{\pm},\bv)\big)  \hotimes_{\kC B} \sfH_{\sC , B}\\
	    	={} & \Phi(\iota_1,\iota_2,\mathbbm{1}) \hotimes_{\qC \kC A} [\sH_B, \phi_+ \ast \phi_-, \bv]
	    	= \Phi (\phi_\pm,\bv),
	\end{align*} 
    together with bijectivity of $\Theta_{\rm triv}$ and $\Phi$ (cf.~\cref{lem:quasihom_represent}).
\end{proof}

		\subsection{Universality of the \texorpdfstring{$\KK^\sC$}{KKC}-functor}
		We consider the following three categories of $\sC$-C*-algebras: 
		\begin{itemize}
		    \item The category $\Corr^\sC_{\rm es}$ of separable $\sC$-C*-algebras and unitary equivalence classes of essential proper $\sC$-Hilbert bimodules (\cref{defn:category.CorrC}).
		    \item The category $\Calg^\sC$ of separable $\sC$-C*-algebras and cocycle $\sC$-$\ast$-homomorphisms (\cref{defn:category.Calg}).
		    \item The subcategory $\Ctriv^\sC$ of $\Calg^\sC$ consisting of separable endomorphism $\sC$-C*-algebras and $\sC$-$*$-homomorphisms (with trivial cocycle). 
		\end{itemize}
		They are all essentially small categories. 
		
		We have canonical functors 
        \[ \Ctriv^\sC \to \Calg^\sC \to \Corr^\sC_{\rm es} \xrightarrow{\KK^\sC} \Kas^\sC, \]
		where the last functor $\KK^\sC $ is given by $\KK^\sC(\sfE, \phi,\bbmv) := [\sfE, \phi,\bbmv] \in \KK^\sC(A,B)$. 
		We use the same letter $\KK^\sC$ for the compositions $\Calg^\sC \to \Kas^\sC$ and $\Ctriv^\sC \to \Kas^\sC$. 
        
		Let $\fC$ be either $\Calg^\sC_{\rm triv}$, $\Calg^\sC$ or $\Corr_{\rm es}^\sC$. A covariant functor $\rF \colon \fC  \to \fA$ to an additive category $\fA$ is said to be
		\begin{itemize}
			\item C*-stable if $\rF (\phi, \bbmv)$ is an isomorphism for any morphism $(\phi,\bbmv)$ with a full corner embedding $\phi$, 
			\item homotopy invariant if $\rF (\eval_0,\mathbbm{1}) = \rF (\eval_1,\mathbbm{1})$, where $\eval_t \colon A[0,1] \to A$ denotes the evaluation at $t \in [0,1]$, and 
			\item split exact if $F(\iota)\oplus F(s)\colon \rF(I) \oplus \rF(A/I)\to \rF(A)$ is an isomorphism for any sequence of morphisms $0 \to I \xrightarrow{\iota} A \to A/I \to 0$ that is split exact with a morphism $s\colon A/I\to A$.
		\end{itemize}
		For example, the functors $\KK^\sC$, $\KK^\sC(A, \blank )$, and $\KK^\sC(\blank, B)$ satisfy these properties taking values in the categories $\Kas^\sC$, $\mathfrak{Ab}$, and $\mathfrak{Ab}^{\op}$, respectively.

\begin{lem}\label{lem:Morita.invert}
    Let $\fC$ be either $\Calg^\sC$ or $\Calg_{\rm triv}^\sC$ and let $\rF \colon \fC \to \fA$ be a C*-stable functor preserving finite direct sums. Then $\rF$ uniquely factors through $\widetilde{\rF} \colon \Corr_{\rm es}^\sC \to \fA$. 
\end{lem}
\begin{proof}
    For separable $\sC$-C*-algebras $A , B $ and a proper essential $\sC$-Hilbert $A$-$B$ bimodule $\sfE=(E,\phi , \bbmv)$, let $\iota \colon B \to \cK(E \oplus B)$ denote the full corner embedding, and define 
    \[ \widetilde{\rF}(\sfE):= \rF (\iota, \mathbbm{1})^{-1} \circ \rF (\phi \oplus 0, \bbmv \oplus 0) \colon \rF(A) \to \rF(B).\] 
    Note that unitary equivalent $\sC$-Hilbert bimodules determine the same morphism since, if $(E_i,\phi_i,\bbmv_i)$ are unitary equivalent by $U$, the following diagram commutes;
    \[
    \xymatrix@C=4em@R=1.5ex{
    &\rF(\cK(E_1 \oplus B)) \ar[dd]^{\Ad(U \oplus 1)}& \\
    \rF(A) \ar[ru]^{\rF(\phi_1 \oplus 0,\bbmv_1 \oplus 0)  \hspace{2em}} \ar[rd]_{\rF(\phi_2 \oplus 0,\bbmv_2 \oplus 0) \hspace{2em} } && \rF(B). \ar[lu]_\cong  \ar[ld]^\cong  \\
    & \rF(\cK(E_2 \oplus B))&
    }
    \]
    It is seen that $\widetilde{\rF}$ is functorial by applying $\rF $ to the following commutative diagram;
    \[ 
    \xymatrix{
    A \ar[r] \ar@/_18pt/[rrd] & \cK(E_1 \oplus B) \ar[rd] & B \ar[l] \ar[r] & \cK(E_2 \oplus D) \ar[ld] & D \ar@/^18pt/[lld] \ar[l] \\
    && \cK(E_1 \otimes _B E_2 \oplus E_2 \oplus D). && 
    }
    \]

    The same holds for C*-stable functors from $\Calg_{\rm triv}^\sC$, by using the $\sC$-C*-algebra $D$ given in \cref{lem:linking.trivcocycle} instead of $\cK(E \oplus B)$.
\end{proof}
Also, it is not hard to see if moreover $\rF$ is split exact and homotopy invariant, so is $\widetilde{\rF}$. 
\begin{thm}\label{thm:Kasparov.universality}
	Let $\fC$ be either $\Calg^\sC_{\rm triv}$, $\Calg^\sC$ or $\Corr_{\rm es}^\sC$. 
	The functor $\KK^\sC \colon \fC \to \mathfrak{KK}^\sC $ is universal among functors to an additive category which is C*-stable, homotopy invariant, and split exact. 
\end{thm}
\begin{proof}
	By \cref{lem:Morita.invert} and the remark after that, the problem is reduced to the case of $\fC=\Corr_{\rm es}^\sC$. Take a functor $\rF \colon \Corr_{\rm es}^\sC \to \fA$ to an additive category $\fA$ that is stable, homotopy invariant, and split exact. 

	As in \cref{lem:qA.KKequiv}, set $\varpi_A:= \pr_1 \circ (j_1\ast _\sC j_2)|_{\qC \kC A} \colon \qC \kC A \to \kC A$ for a $\sC$-C*-algebra $A$. 
	Note that, by the assumption on $\rF$, the same argument as \cref{prop:KK.free.product} and \cref{lem:qA.KKequiv} shows that $\rF (\varpi_A)$ is an isomorphism. 
	Put the proper $\sC$-Hilbert $\kC \qC \kC A$-$A$ bimodule 
	\[ 
	\sfK_A:= \sfH_{\sC,\qC \kC A} \otimes ({}_{\varpi _A} \kC A ) \otimes  \sfH_{\sC,A}.
	\] 
	By \cref{prop:KK.free.product}, it induces a $\KK^\sC$-equivalence. In the same way, we also have that $\rF(\sfK_B ) = \rF(\sfH_{\sC,\qC \kC A}) \circ \rF (\varpi _A , \mathbbm{1}) \circ \rF(\sfH_{\sC,A})$ is an isomorphism. 

	Now, the map 
	\[ 
	\mathcal{KK}^\sC :=(\sfK_A)^{-1} \hotimes_{\kC \qC \kC A} \KK^\sC( \blank ) \hotimes_{\kC \qC \kC B} \sfK_B \colon [\kC\qC\kC A, \kC\qC\kC B]_{\rm triv}^\sC \to \KK^\sC( A, B)
	\]
	is defined, and is bijective by \cref{thm:quasihom.replacement} (2). 
    The desired functor is constructed as 
	\[ 
	\overline{\rF} (\xi ):= \rF (\sfK_B) \circ \rF((\mathcal{KK}^\sC) ^{-1}(\xi)) \circ \rF(\sfK_A)^{-1}.  
	\]
    This is indeed an extension of $\rF$, i.e., $\overline{\rF}\circ\KK^\sC = \rF$. 
    To see this, notice that $\kC$ and $\qC$ are functorial on $\Calg_{\rm triv}^\sC$ (cf.~\cref{cor:specialization} and \cref{prop:universality.free.product}), and they satisfy $\sfK_B \circ \kC\qC\kC (\phi,\mathbbm{1})\cong (\phi,\mathbbm{1})\circ\sfK_A$ in $\Corr^{\sC}_{\rm es}$. This shows that $\overline{\rF}([\phi,\mathbbm{1}]) = \rF (\phi,\mathbbm{1})$.  
    More generally, for a proper essential $\sC$-Hilbert $A$-$B$ bimodule $(E,\phi,\bbmv)$, we have $\sfE \otimes_D(\iota, 1) \cong (\phi ,\mathbbm{1} )$ as  \cref{lem:linking.trivcocycle}, 
    and hence 
    \begin{align*}
        \overline{F}([E,\phi,\mathbbm{v}]) 
        = \overline{\rF}([\iota,\mathbbm{1}])^{-1} \circ \overline{\rF}([\phi ,\mathbbm{1}]) = \rF(\iota,\mathbbm{1})^{-1} \circ \rF(\phi,\mathbbm{1}) = \rF(E,\phi,\bv).
    \end{align*}
    
    Finally, the uniqueness of $\overline{\rF}$ also follows from the construction.
\end{proof}
		
		We also describe the universality in terms of the module category picture. 
		
		\begin{cor}\label{cor:universality.categorical}
			A functor ${\mathrm F} \colon \Ccat^\sC \to \fA$ (cf.~\cref{defn:category.Ccat}) to an additive category that is homotopy invariant, and split exact factors through $\KK^\sC \colon \Ccat^\sC \to \Kas^\sC_{\rm cat} $.
		\end{cor}
		\begin{proof}
			The composition $\Corr_{\mathrm{es}}^\sC \xrightarrow{\Mod(\blank )} \Ccat^\sC \xrightarrow{\rF} \fA$
			is a functor that is C*-stable, homotopy invariant, and split exact. 
			Now the corollary follows from \cref{thm:Kasparov.universality} and \cref{prop:Kasparov.category.algebra}.
		\end{proof}

		\subsection{Triangulated category structure}
		Following the work of Meyer--Nest \cite{meyerBaumConnesConjectureLocalisation2006}*{Appendix A}, we prove that the category $\Kas^\sC$ is endowed with the structure of the triangulated category.   
		Recall that a triangulated category is an additive category $\fT$ equipped with the additive automorphism $\Sigma \colon \fT \to \fT$ called the suspension, and a class of triangles (a sequence of the form $A \to B \to C \to \Sigma A$) called distinguished triangles, satisfying some axioms.
		
		The stabilized Kasparov category $\widetilde{\Kas}{}^\sC$ is the category whose objects are the pairs $(A,n) \in \Calg \times \bZ$ and $\Hom _{\widetilde{\Kas}{}^\sC} ((A,n),(B,m)):= \KK^\sC_{n-m}(A,B)$. Note that the inclusion $A \mapsto (A,0)$ gives a category equivalence $\Kas^\sC \cong \widetilde{\Kas}{}^\sC$. 
		The suspension automorphism is defined by $\Sigma(A,n)=(A,n+1)$ and the set of distinguished triangles of the Kasparov category $\Kas^\sC$ as the ones equivalent to the mapping cone triangle (cf. \cref{ex:mapping.cone})
		\[ SB \to  \cone (\phi, \bbmv) \to A \xrightarrow{(\phi, \bbmv)} B. \]
		
		\begin{thm}\label{thm:triangulated}
			The above $\Sigma$ and distinguished triangles determine a triangulated category structure on $\widetilde{\mathfrak{KK}}{}^\sC$. 
		\end{thm}
		\begin{proof}
			We check the following axioms of the triangulated category. 
			\begin{description}
				\item[TR0] By definition, a triangle that is equivalent to a distinguished one is also distinguished, and a triangle of the form $A \xrightarrow{\id_A} A \to 0 \to \Sigma A$ is distinguished. 
				\item[TR1] By \cref{thm:quasihom.replacement}, any morphism $f \in \widetilde{\Kas}{}^\sC(A,B)$ can be represented by a cocycle $\sC$-$\ast$-homomorphism $(\phi,\bbmv)$, and hence $f$ is a part of the distinguished triangle $A \xrightarrow{f} B \to S\cone (\phi,\bbmv) \to \Sigma A$.  
				\item[TR2] We show that a triangle $A \xrightarrow{f} B \xrightarrow{g} D \xrightarrow{h}\Sigma A$ is distinguished if and only if so is $B \xrightarrow{-g} D \xrightarrow{-h}\Sigma A \xrightarrow{-\Sigma f} \Sigma B$. To see this, it suffices to show that $SA \to SB \to \cone(\phi, \bbmv) \to A$ is equivalent to the triangle $SA \to \cone (\cone(\phi,\bbmv) \to A) \to \cone(\phi,\bbmv) \to A$, which follows from the homotopy equivalence of $SB \to \cone (\cone (\phi,\bbmv) \to A)$.
				\item[TR3] For any distinguished triangles $A \to B \to D \to \Sigma A$ and $A' \to B' \to D' \to \Sigma A'$ and $f \colon A \to A'$, $g \colon B \to B'$, there exists $h \colon D \to D'$ such that the diagram 
				\[
				\xymatrix{A \ar[r] \ar[d]^f & B \ar[d]^g \ar[r] & D \ar[r] \ar@{.>}[d]^{\exists h}  & \Sigma A \ar[d] \\ A' \ar[r] & B' \ar[r] & D' \ar[r] & \Sigma A'}
				\]
				commutes. To see this, we introduce a variation of the mapping cone.
				For $(\phi,\bbmv) \colon A \to B$ and $(\psi, \bbmw) \colon B \to D$, we define
				\begin{align*}
					&\mathsf{Cyl}((\phi, \bbmv), (\psi, \bbmw)):= A \oplus _A B[0,1] \oplus _B \cone D\\
					=& \{ (a,f,g) \in A \oplus B[0,1] \oplus D[0,1] \mid f(0)=\phi(a) , f(1)=\psi(g(0)), g(1)=0 \},
				\end{align*}
				which is equipped with the $\sC$-C*-algebra structure as fibered sum.
				Note that $\mathsf{Cyl}((\phi, \bbmv), (\psi,\bbmw))$ is homotopy equivalent to the mapping cone of $(\psi,\bbmw) \circ (\phi,\bbmv)$, and there is an exact sequence
				\begin{align*}
					0 \to \cone (\phi,\bbmv) \to \mathsf{Cyl}((\phi,\bbmv), (\psi,\bbmw)) \to \cone(\psi,\bbmw) \to 0.
				\end{align*} 
				Now, for $[\sfE ,F] \in \KK^\sC(A,A')$ and $[\sfE',F'] \in \KK^\sC(B,B')$, a homotopy of Kasparov $A$-$B'$ bimodules $(\phi,\bbmv) \hotimes _{A'}(\sfE',F')$ and $(\sfE,F) \hotimes_B (\phi',\bbmv')$ gives rise to a Kasparov bimodule in 
				$\bfE^\sC( \mathsf{Cyl}(\id_A, (\phi,\bbmv)), \mathsf{Cyl}((\phi',\bbmv'), \id_{B'}) )$.
				\item[TR4] The cocycle $\sC$-$\ast$-homomorphisms $(\phi,\bbmv) \colon A \to B$ and $(\psi,\bbmw) \colon B \to D$ extends to the commutative diagram
				\[\xymatrix{
					S^2D \ar[r] \ar[d] & S\cone (\psi,\bbmw) \ar[r] \ar[d] & SB \ar[r] \ar[d] & SD \ar[d] \\
					0 \ar[r] \ar[d] & \cone (\phi,\bbmv)  \ar[r] \ar[d] &  \cone (\phi,\bbmv) \ar[r] \ar[d] & 0 \ar[d] \\
					SD \ar[r] \ar[d] &  \mathsf{Cyl}((\phi,\bbmv),(\psi,\bbmw)) \ar[r] \ar[d] & A \ar[r] \ar[d] & D \ar[d] \\
					SD \ar[r]  & \cone (\psi,\bbmv) \ar[r] & B \ar[r] & D
				}\]
				such that the vertical and horizontal sequences are all distinguished triangles. 
				This diagram satisfies the octahedral axiom, i.e., the compositions $SB \to \cone (\phi,\bbmv) \to \mathsf{Cyl}((\phi,\bbmv),(\psi,\bbmw))$ and $SB \to SD \to \mathsf{Cyl}((\phi,\bbmv),(\psi,\bbmw))$ are identical in $\Kas^\sC$ (indeed, they are homotopic). \qedhere
			\end{description}
		\end{proof}
		
\begin{rem}
	Since the mapping cone of $\sC$-module functors is defined as a $\sC$-module category (see \cref{subsubsection:ideal.exact,subsubsection:fiber.sum}), the triangulated category structure on $\widetilde{\Kas}{}^\sC_{\rm cat}$ is also defined in the internal language of $\sC$-module categories.
\end{rem}

In the spirit of Meyer--Nest \cite{meyerBaumConnesConjectureLocalisation2006}, this triangulated structure could potentially be used for formulating the Baum--Connes type conjecture in terms of semi-orthogonal decomposition of the category. 
Recall that a relative homological algebra of the restriction-induction adjunction \cites{meyerHomologicalAlgebraBivariant2008,meyerHomologicalAlgebraBivariant2010} provides a semi-orthogonal decomposition. We just take the first step toward this direction.

\begin{prop}\label{prop:adjunction.tensor}
	Let $\sC$ be a rigid C*-tensor full subcategory of $\sC$. 
	Then we have a left adjoint $\Ind^\sC$ to the restriction functor $\Res_\sC\colon \Ccat^\sC\to \Ccat$ forgetting the $\sC$-action. 
	The pair $(\Ind^\sC, \Res_\sC)$ induces an adjoint pair of Kasparov categories. 
\end{prop}
This yields a semi-orthogonal decomposition $(\langle \cT \cI\rangle _{\loc}, \cT \cC )$ of $\Kas^\sC$, where $\cT\cI$ denotes the class of objects of the form $\Ind^\sC A$ for some $A \in \Obj \Calg$, and $\cT\cC$ denotes the full subcategory of separable $\sC$-C*-algebras with $\Res_\sC A \sim_{\KK} 0$.

\begin{proof}
	We let $\Ind^\sC\colon \Ccat\to \Ccat^\sC$ as $\sA\mapsto \sA\boxtimes \tilde{\sC}$, where $\tilde{\sC}$ denotes the separable envelope of $\sC$ (\cref{rem:separable.envelope}). It induces the functor of Kasparov categories $\Ind^\sC \colon  \Kas\to\Kas^\sC $ by universality (\cref{thm:Kasparov.universality}). 
	For any nonunital C*-category $\sA$ and any $\sC$-module category $\sB$,  we take 
	\begin{align*}
	    &&\eta_{\sA} &{}\colon \sA \to \sA \boxtimes \tilde{\sC} , && \eta_{\sA} (X) = X \boxtimes \mathbf{1}_\sC, && \\
	    &&\epsilon_{\sB} &{}\colon \sB \boxtimes \tilde{\sC} \to \sB, && \epsilon_{\sB} (X \boxtimes \pi):=X \otimes_\beta \pi. &&
	\end{align*}
    The latter assignment $\epsilon_\sB$ forms a $\sC$-module functor by the coherence map 
    \[ 
    \mathsf{v}_{\pi,X} \colon \epsilon_{\sB}(X\boxtimes (\rho \otimes \pi))=X\otimes_\beta (\rho \otimes \pi) \xrightarrow{\fv_{\rho,\pi}^*} (X \otimes_\beta \rho) \otimes_\beta \pi = \epsilon_{\sB}(X \boxtimes \rho) \otimes_\beta \pi.
    \]
	It is routine to check $\eta_{\sA}$ and $\epsilon_{\sB}$ satisfy the unit and counit relations. 
\end{proof}

\section{Weak Morita invariance of the Kasparov category}\label{section:Morita}
In this section, we prove the following theorem.
\begin{thm}\label{thm:Morita.invariance}
	Let $\sC$ and $\sD$ be weakly Morita equivalent rigid C*-tensor categories. Then there exists a categorical equivalence
	\[ \Kas^\sC \simeq \Kas^\sD.\]
\end{thm}
\begin{cor}
	Let $G,H$ be compact quantum groups. If they are categorically Morita equivalent, i.e., $\Rep(G)$ and $\Rep(H)$ are weakly Morita equivalent,  then $\Kas^{G}$ and $\Kas^{H}$ are categorically equivalent.
\end{cor}

\subsection{Weak Morita equivalence of tensor categories}
We start with a short review of the generalities of the weak Morita equivalence.
Just for simplicity of discussion, in the rest of this section, we assume that the C*-tensor category $\sC$ and any nonunital $\sC$-module category $\sA$ is strict, by taking their strictifications if necessary. 
Indeed, for a C*-tensor category $\sC$ and its strictification $\sC^{\wr}$ (cf.~\cite{etingofTensorCategories2015}*{Theorem 2.8.5}), the categories $\Ccat^{\sC}$ and $\Ccat^{\sC^{\wr}}$ are equivalent by \cref{rem:monoidal.invariance}. Moreover, a $\sC^{\wr}$-module category $\sA$ can be strictified (cf.~\cite{etingofTensorCategories2015}*{Remark 7.2.4}, whose proof works for our C*-setting). 
\begin{df}[{\cite{neshveyevCategoricallyMoritaEquivalent2018}*{Definition~3.1}}]
	Let $\sC$ and $\sD$ be rigid C*-tensor categories. They are said to be \emph{weakly Morita equivalent} if there exists a rigid C*-$2$-category $\tilde \sC$ such that
	\begin{enumerate}
		\item The $0$-morphism of $\tilde \sC$ is $\{0,1\}$,
	    \item $\tilde \sC(0,0) \simeq \sC$ and $\tilde \sC(1,1) \simeq \sD$ as C*-tensor categories.
		\item $\tilde \sC(0,1) \neq 0$.
	\end{enumerate}
\end{df}
Weak Morita equivalence is rephrased in terms of $Q$-system (see e.g.~\cite{bischoffTensorCategoriesEndomorphisms2015}), namely, an algebra object $\rho \in \Obj \sC$ (with multiplication $m_\rho$ and unit $\eta_\rho$) that satisfies the Frobenius condition $(m_\rho \otimes 1_\rho)(1_\rho \otimes m_\rho^*) = m_\rho^*m_\rho = (1_\rho \otimes m_\rho)(m_\rho^* \otimes 1_\rho )$ and $m_\rho m_\rho ^* \in \bR_{>0} \cdot 1_\rho$. In this paper, we normalize $m_\rho$ and always assume that a $Q$-system satisfies $m_\rho m_\rho^* =1_\rho$. 

A right $\rho$-module in $\sC$ is an object $X \in \Obj \sC$ with a coisometry $m_X \colon X \otimes \rho \to X$ such that the diagrams
    \begin{align}
    \begin{split}
    \xymatrix@C=2.5em{X \otimes \rho \otimes \rho \ar[r]^{m_X \otimes 1_\rho} \ar[d]^{1_X \otimes m_\rho} & X \otimes \rho \ar[d]^{m_X} \\ X \otimes \rho \ar[r]^{m_X} & X,} \quad  
	\xymatrix@C=1em{X \otimes 1 \ar[rr]^{1_X \otimes \eta_\rho} \ar[rd] && X \otimes \rho \ar[ld]^{m_X} \\ & X, & } \quad
	\xymatrix@C=2.5em{X \otimes \rho  \ar[r]^{m_X^* \otimes 1_\rho \; \; } \ar[d]^{m_X} & X \otimes \rho \otimes \rho \ar[d]^{1_X \otimes m_\rho} \\ X  \ar[r]^{m_X^*} & X \otimes \rho,}
	\end{split}
	\label{eqn:Qsystem.module}
	\end{align}
commute. 
A morphism $T \colon X \to Y$ between right $\rho$-modules is said to be \emph{$\rho$-equivariant} if it satisfies $ T m_X = m_Y(T \otimes 1_\rho)$. 
The notion of right $\rho$-module and $\rho$-$\rho$ bimodule are also defined in the same way. 
We write $\sMod_{\sC}^l (\rho)$, $\sMod_\sC^r(\rho)$, and $\sBimod_\sC(\rho)$ for the C*-category of left, right $\rho$-modules and $\rho$-$\rho$ bimodule objects in $\sC$ respectively. 
\begin{prop}[{\cite{neshveyevCategoricallyMoritaEquivalent2018}*{Section 2.2, Theorem~3.2}}]\label{prop:weak.Morita}
	The following are equivalent:
	\begin{enumerate}[(1)]
		\item The tensor categories $\sC$ and $\sD$ are weakly Morita equivalent.
    	\item There exists a non-zero semisimple cofinite $\sC$-module (unital) C*-category $\sM$ such that $\sD \simeq \Func_\sC(\sM,\sM)$. In this case, the C*-$2$-category is given so that $\tilde \sC(0,1) = \sM$ and $\tilde{\sC}(1,0) = \sM^*:= \Func_\sC(\sM,\sC )$.
    	\item There exists a $Q$-system $(\rho, m_\rho, \eta_\rho)$ in $\sC$ such that $\sD \simeq  \sBimod(\rho)$. In this case, the C*-$2$-category is given so that $\tilde \sC(0,1)=\sMod^{l}_\sC (\rho)$ and $\sC(1,0)=\sMod^r_\sC (\rho)$. 
	\end{enumerate}
\end{prop}
        Both the identifications $\sBimod (\rho) \simeq \Func_\sC(\sM,\sM)$ and $\sMod_\sC^r(\rho) \simeq \Func_\sC(\sM , \sC)$ are given by sending a module $\pi$ to the tensor product functor $\pi \otimes_\rho \blank$ (cf.~\cref{lem:Qsystem.module.tensor}). 

\begin{rem}\label{rem:Qsystem.Morita}
    In \cref{prop:weak.Morita}, the role of $\sC$ and $\sD$ are symmetric. In (2), one can identify $\sC$ with $\Func_\sD (\sM^*,\sM^*)$ by sending $\pi \in \Obj \sC$ to the functor $\Obj \sM^* \ni F \mapsto (\pi \otimes \blank ) \circ F $. Then we have $\sM \simeq \Func_{\sD}(\sM^*,\sD)$  as $\sD$-$\sC$ bimodule categories. 

	In (3), consider the internal endomorphism object
	\[ 
	\iEnd(\rho):=(\rho \otimes \rho , \underline{m}_{\iEnd(\rho)}:=1_\rho \otimes (\|\eta_\rho\|^{-1} \eta_\rho ^* m_\rho) \otimes 1_\rho , \eta_{\iEnd(\rho)}:= m_\rho^* \eta_\rho), 
	\]
	which is a priori a $Q$-system in $\sC$. Since it respects the $\rho$-bimodule structure, $\underline{m}_{\iEnd(\rho)}$ reduces to $\iEnd(\rho) \otimes_\rho \iEnd(\rho) \to \iEnd(\rho)$ (we also remark that $\underline{m}_\rho \colon \rho \otimes_\rho \iEnd(\rho) \to \rho$ is defined in the same way). 
	This gives a $Q$-system in $\sBimod_\sC (\rho)$, denoted by $\iiEnd(\rho)$ in this paper. 
	This $Q$-system recovers the same C*-$2$-category by the following lemma. 
\end{rem}

\begin{lem}\label{lem:linking.bimodule}
	The left (resp.~right) $\sD$-module categories $\sMod^{l}_\sC(\rho)$ and $\sMod^{r}_\sD (\iiEnd (\rho))$ (resp.~$\sMod^{r}_\sC(\rho)$ and $\sMod^{l}_\sD (\iiEnd (\rho))$) are equivalent. Moreover, these equivalences are also $\sC$-equivariant under the monoidal equivalence $\sC \simeq \sBimod (\iiEnd (\rho ) )$.
\end{lem}
\begin{proof}
	The desired categorical equivalence is given by 
	\[
	\sMod^{l}_\sC(\rho) \to \sMod^{r}_\sD (\iiEnd (\rho)), \quad (X,m_X) \mapsto ((X \otimes \rho , m_X \otimes 1_\rho, 1_X \otimes m_\rho), 1_X \otimes \underline{m}_\rho).
	\]
	Indeed, its inverse functor is given by 
	\[ 
	\sMod^{r}_\sD (\iiEnd (\rho)) \to \sMod^{l}_\sC(\rho), \quad ((Y,m_Y^l,m_Y^r), \underline{m}_Y) \mapsto (Y \otimes_{\iEnd(\rho)} \rho , m_Y^l\otimes 1_{\rho}).
	\]
	In the same way, $\sC \to \sBimod_\sC(\iiEnd(\rho))$ is given by sending $\pi \in \Obj \sC$ to $\rho \otimes \pi \otimes \rho$. Now the $\sD$-$\sC$ equivariance of the above functor is obvious from the definition.   
\end{proof}

		\begin{ex}\label{ex:Qsystem.group}
			Let $\Gamma$ be a countable discrete group and set $\sC=\Hilb_\Gamma^{\rm f}$. 
			Let $\Lambda$ be a finite subgroup of $\Gamma$. Then $\rho = \bC[\Lambda] \in\Obj\sC$ is a $Q$-system. 
			In this case, the bimodule category $\sBimod(\rho)$ is identified with the category of $\Gamma$-graded $\Lambda$-bimodules. 
			In particular, when $\Lambda = \Gamma$ is a finite group, the Q-system gives a weak Morita equivalence between $\Rep(\Gamma)$ and $\Hilbg$. 
			\cref{thm:Morita.invariance} for this weak Morita equivalence is just the Takesaki--Takai (Baaj--Skandalis) duality of equivariant $\KK$-theory \cite{baajCalgebresHopfTheorie1989}. 
		\end{ex}
		\begin{ex}
			Let $\sC$ be a fusion C*-category. 
			Then the Deligne tensor product $\sC \boxtimes \sC^{\rm rev}$ is weakly Morita equivalent to the categorical center $Z(\sC)$, via the left $\sC \boxtimes \sC^{\rm rev}$-module category $\sC$ (see e.g.~\cite{etingofTensorCategories2015}*{Section 7.13}). 
		\end{ex}
		\begin{ex}
		There are some examples of compact quantum groups with weakly Morita equivalent representation categories. 
		\begin{enumerate}
		    \item Two finite groups $\Gamma_1, \Gamma_2$ may have Morita equivalent representation categories (a necessary and sufficient condition is given in \cite{naiduCategoricalMoritaEquiavlence2007}*{Corollary~5.9}). 
		    In particular, this happens if there is an abelian normal subgroup $\Lambda \leq \Gamma_1$ such that $\hat{\Lambda} \rtimes \Gamma_1/\Lambda \cong \Gamma_2$ and $\omega =0$, where $\omega \in H^3(\hat{\Lambda} \rtimes \Gamma_1/\Lambda ; \bT)$ denotes the $3$-cocycle defined later in \eqref{eqn:cocycle.Takesaki}. Note that this is an extension of the construction of isocategorical groups given by \cite{etingofIsocategoricalGroups2001}.
		    \item For a finite quantum group $F$, a compact quantum group $H$, and an automorphism $\mathsf{m} \in \Aut(L^\infty(F)\otimes L^\infty(H))$ is called a \emph{matching} if $(\Delta_F\otimes\id)\mathsf{m}=\mathsf{m}_{23}\mathsf{m}_{13}(\Delta_F \otimes \id)$ and $(\id \otimes \Delta_H) \mathsf{m} = \mathsf{m}_{13} \mathsf{m}_{12} (\id \otimes \Delta_H )$. Associated to it, there are two constructions of a new compact quantum group, the bicrossed product  \cite{vaesExtensionsLocallyCompactQuantum2003} and the double crossed product \cite{baajDoubleCrossedProduct2005} (denoted by $G_1$ and $G_2$ here), whose underlying function C*-algebras are $C(G_1):=F \ltimes C(H)$ and $C(G_2):=C(F)\otimes C(H)$ respectively. 
		    They have weakly Morita equivalent representation categories. Indeed, 
		    $F^{\op}\ltimes C(G_2)$ with the left $G_1$-action by \cite{baajDoubleCrossedProduct2005}*{Proposition~6.1} and the right $G_2$-action by $\Delta_{G_2}$, forms a Morita--Galois object in the sense of \cite{neshveyevCategoricallyMoritaEquivalent2018}*{Definition 3.3} by \cite{neshveyevCategoricallyMoritaEquivalent2018}*{Theorem~3.7} and a similar argument as in \cite{neshveyevCategoricallyMoritaEquivalent2018}*{Example~4.6}. 
		\end{enumerate}
		\end{ex}

		Our \cref{thm:Morita.invariance} on KK-theory reduces to the following theorem on weakly Morita equivalent C*-tensor categories, which will be proved in the next subsection.
		\begin{thm}\label{thm:weakMorita}
			Let $\sC$ and $\sD$ be weakly Morita equivalent tensor categories with the implementing bimodule categories $\sM := \tilde{\sC}(0,1)$ and $\sM^*:= \tilde{\sC}(1,0)$. 
			\begin{enumerate}
				\item For a separable nonunital $\sC$-module category $\sA$, the $\sC$-module functor category $\Func _\sC(\sM, \sA)$ is equipped with the structure of a separable nonunital $\sD$-module category. 
				\item There is a natural isomorphism $\Func _\sD( \sM^*, \Func _\sC (\sM, \blank)) \cong \id$ of functors on $\Ccat^\sC$. 
			\end{enumerate}
		\end{thm}
        In (2), for a $\sC$-module category $\sA$, $ \Func _\sD( \sM^*, \Func _\sC (\sM, \sA))$ is regarded as a $\sC$-module category by $\sC \simeq \Func_{\sD}(\sM^*,\sM^*)$ stated in \cref{rem:Qsystem.Morita}.
		
		\begin{proof}[Proof of \cref{thm:Morita.invariance}]
			By \cref{thm:weakMorita} (1), the assignment $\sA \mapsto \Func_\sC(\sM,\sA)$ gives a functor 
			\[ \Func_\sC(\sM , \blank ) \colon \Ccat^\sC \to \Ccat^\sD.\]
			It has an inverse $\Func_\sD(\sM^*, \blank)$ by \cref{thm:weakMorita} (2). 
            In particular, the functor $\Func_{\sC}(\sM,\blank)$ preserves split exact sequences of module categories in the sense of \cref{def:exact.sequence.categories} thanks to the argument after \eqref{eq:exact.Cmodule.category} because the categorical equivalence preserves equalizers if exist. 
            
			Thus, the theorem follows from the universality of the $\KK^\sC$-functor (\cref{cor:universality.categorical}) since $\Func_\sC(\sM,\blank)$ preserves homotopies and split exact sequences. 
			Indeed, as will be checked in \cref{rem:Qsystem.algebra} (3), $\Func_\sC(\sM,\blank) $ respects the structure of $\Ccat$-module, i.e., the natural transform $\Func_\sC(\sM,\sA) \boxtimes \sB  \to \Func_\sC(\sM,\sA \boxtimes \sB)$ of bifunctors from $\Ccat^\sC \times \Ccat$ to $\Ccat^\sD$ is an isomorphism. 
			Homotopy invrariance of $\KK^\sD \circ \Func_\sC(\sM,\blank) $ is checked by applying this for $\sB = \Mod (C[0,1])$ and the evaluation $\ast$-homomorphisms. 
		\end{proof}
		
\subsection{C*-category of modules over a \texorpdfstring{$Q$}{Q}-system}
Here we prove \cref{thm:weakMorita} by comparing the module category picture and the Q-system picture of the weak Morita equivalence. 

As in the previous subsection, let $(\rho, m_\rho, \eta_\rho)$ be a $Q$-system in $\sC$ and let $\sM:=\sMod_\sC^l(\rho)$ be the connected cofinite C*-$\sC$-module category of left $\rho$-modules in $\sC$, which has a strict right $\sC$-module structure. 
For a separable nonunital right $\sC$-module category $\sA$, a $\rho$-module object $(X,m_X)$ in $\sA$ is defined by the same relation as \eqref{eqn:Qsystem.module}.
\begin{df}
    The category of right $\rho$-modules and $\rho$-equivariant morphisms in $\sA$ is denoted by $\sA_\rho$.
\end{df}
\begin{lem}\label{lem:Qmodule.nonunital}
	The C*-category $\sA_\rho$ is endowed with a structure of nonunital C*-category with the categorical ideal $\cK_\rho(X,Y)$ defined by the set of compact equivariant morphisms in $\cL_\rho(X,Y):=\cL_{\sA_\rho}(X,Y)$. 
	Moreover, if $\sA$ is separable, then so is $\sA_\rho$. 
\end{lem}
\begin{proof}
	We need to show $\cL_\rho(X,Y)$ is the multiplier of $\cK_\rho(X,Y)$.
	For any $x \in \cL_\sA (X,Y)$, the morphism $m_Y(x \otimes 1_\rho) m_X^*$ is $\rho$-equivariant by 
	\begin{align*}
		\big( m_Y (x \otimes 1) m_X^*\big) m_X &{}= m_Y(x \otimes 1_\rho ) (1_X \otimes m_\rho)(m_X^* \otimes 1_\rho)  \\
		&{}= m_Y (1_X \otimes m_\rho)(x \otimes 1_\rho \otimes 1_\rho )(m_X^* \otimes 1_\rho)\\
		&{}= m_Y\big( (m_Y(x \otimes 1_\rho )m_X^*) \otimes 1_\rho \big).
	\end{align*}
	Hence we get a subspace $m_Y(\cK_\sA (X,Y) \otimes 1_\rho)m_X^* \subset \cK_\rho (X,Y) \subset \cK_\sA (X,Y)$, whose approximate unit strictly converges to the identity when $X=Y$. 
	This shows that the inclusion $\cK_\rho(X,Y) \subset \cK_\sA (X,Y)$ extends to $\cM (\cK_\rho (X,Y)) \to \cL_\sA(X,Y)$, whose image is contained in $\cL_\rho(X,Y)$. 
	The resulting functor $\cM(\cK_\rho) \to \cL_\rho $ is bijective since $\cK_\rho \subset \cL_\rho$ is an essential categorical ideal.
			
	The last claim follows from that, if $Y \in \Obj \sA$ is dominant,  then so is $Y \otimes \rho \in \Obj \sA_\rho$. 
	Indeed, any $(X,m_X) \in \Obj \sA_\rho$ is embedded to $Y^\infty \otimes \rho$ by the composition of $m_X^* \colon (X, m_X) \to (X \otimes \rho, 1_X \otimes m_\rho)$ and $i \otimes 1_\rho$ for some isometry $i \colon X \to Y^\infty$. 
\end{proof}

\begin{lem}[{\cite{neshveyevDrinfeldCenterRepresentation2016}*{Remark 6.5}}]\label{lem:Qsystem.module.tensor}
	Let $(X,m_X) \in \Obj \sA_\rho$ and a let $\pi$ be a left $\rho$-module in $\sC$.
	\begin{enumerate}
		\item There exists an object $X \otimes_\rho \pi$ in $\sA$, regarded as a direct summand of $X\otimes \pi$ via a coisometry $p_{X,\pi} \colon X \otimes \pi \to X \otimes_\rho \pi$ with $p_{X,\pi}(m_X \otimes 1_\pi) = p_{X,\pi}(1_X \otimes m_\pi)$ that satisfies the following universality:
		For any $Y \in\Obj\sA$ and any morphism $f \colon X \otimes \pi \to Y$ such that $f(m_X \otimes 1_\pi) = f(1_X \otimes m_\pi)$, there exists a unique morphism $\bar f \colon X \otimes_\rho \pi \to Y$ with $\bar f p_{X,\pi} = f$.
		\item The induced map $\bar{m}_X \colon X \otimes_\rho \rho \to X$ is a unitary.
	\end{enumerate}
\end{lem}
\begin{proof}
	For each $X$ and $\pi$ above, the element $p_{X,\pi} = (m_X \otimes 1_{\pi})(1_X \otimes m_{\pi})^* \in \cL(X \otimes \pi)$ is a projection since
	\begin{align*}
				p_{X,\pi}^* p_{X,\pi} & = (1_X \otimes m_{\pi})(m_X \otimes 1_{\pi})^* (m_X \otimes 1_{\pi})(1_X \otimes m_{\pi})^* \\
				&= (1_X \otimes m_{\pi})(m_X \otimes 1_\rho \otimes 1_{\pi})(1_X \otimes m_\rho \otimes 1_{\pi})^*(1_X \otimes m_{\pi})^* \\
				&= (m_X \otimes 1_\pi)(1_X \otimes 1_\rho \otimes m_{\pi})(1_X \otimes 1_\rho \otimes m_{\pi})^*(1_X \otimes m_{\pi})^* \\
				&= p_{X,\pi}.
	\end{align*}
	Now take the direct summand $X \otimes_\rho \pi$ of $X \otimes \pi$ corresponding to $p_{X,\pi}$. By abuse of notation, the projection from $X \otimes \pi \to X \otimes_\rho \pi$ is also written by $p$. It is routine to show that $p_{X,\pi}(m_X \otimes 1_\pi) = p_{X,\pi}(1_X \otimes m_\pi)$.
	For $f \colon X \otimes \pi \to Y$ with the condition, put $\bar f$ to be the restriction of $f$ to $X \otimes_\rho \pi$. Then $f p_{X,\pi} = f$ shows the desired universality.
			
	For (2), since $m_X^* \colon X \to X \otimes \rho$ is isometry, the Frobenius condition $p_{X,\pi} = m_X m_X^*$ shows $X \otimes_\rho \rho \simeq X$.
\end{proof}

For a $\rho$-bimodule $\pi$ in $\sC$ and a right $\sC$-module functor $(F,\mathsf{v}) \colon \sMod^l_\sC(\rho) \to \sA$, $F(\pi)$ is equipped with the right $\rho$-action by $F(m_\pi^r) \circ \mathsf{v}_{\pi,\rho}^* \colon F(\pi) \otimes \rho \to F(\pi \otimes \rho) \to F(\pi)$, where $m_\pi^r$ denotes the right $\rho$-action on $\pi$. 

\begin{lem}\label{lem:bimodrholinear}
	For any $\sC$-module functor $(F, \mathsf{v}) \colon \sM \to \sA$, $\rho$-bimodule $\pi$ and left $\rho$-module $\sigma$, there is a unitary $\mathsf{u}_{F,\pi,\sigma} \colon F(\pi \otimes_\rho \sigma) \to F(\pi ) \otimes_\rho \sigma$. The family $\{ \mathsf{u}_{F,\pi,\sigma}\}$ forms a natural unitary between trifunctors from $\Func_\sC(\sM,\sA) \times \sBimod(\rho) \times \sM$ to $\sA$. 
\end{lem}

\begin{proof}
    By using the projection in \cref{lem:Qsystem.module.tensor} (1), set
	\[
	\mathsf{u}_{F,\pi,\sigma} := p_{F(\pi),\sigma } \mathsf{v}_{\pi,\sigma} F(p_{\pi,\sigma} ) \colon F(p_{\pi,\sigma}) \cdot F(\pi \otimes \sigma ) \to 
	p_{F(\pi),\sigma } \cdot (F(\pi ) \otimes \sigma ).
	\] 
    Then its domain and the codomain are canonically isomorphic to $F(\pi \otimes_\rho \sigma )$ and $F(\pi) \otimes_\rho \sigma$ respectively. 
    This is natural by construction, and is a unitary since
	\begin{align*}
	    \mathsf{v}_{\pi,\sigma} 
		F((m_\pi\otimes1_{\sigma})(1_\pi\otimes m_\sigma)^*) 
		&{}=
		(m_{F(\pi)}\otimes1_{\sigma})(\mathsf{v}_{\pi,\rho}\otimes1_\sigma)
		\mathsf{v}_{\pi\otimes\rho,\sigma} 
		F(1_\pi\otimes m_\sigma)^* \\
		&{}=
		(m_{F(\pi)}\otimes1_{\sigma})\mathsf{v}_{\pi,\rho\otimes\sigma}
		F(1_\pi\otimes m_\sigma)^* \\
		&{}=
		(m_{F(\pi)}\otimes1_{\sigma})(1_{F(\pi)}\otimes m_\sigma)^* \mathsf{v}_{\pi,\sigma}. \qedhere 
	\end{align*}
\end{proof}
		
		\begin{lem}\label{lem:functor.Qsystem}
			The $\ast$-functor
			\[\mathcal{F} \colon \Func_\sC(\sM,\sA) \to \sA_\rho, \quad F \mapsto F(\rho),\]
			is an equivalence of (unital) $\sD$-module categories  with the inverse functor
			\[\mathcal{G} \colon \sA_\rho \to \Func_\sC(\sM,\sA), \quad X \mapsto X \otimes_\rho \blank .\]
            Here, an object $\sigma $ of $\sD \cong \sBimod (\rho)$ acts on $\sA_\rho$ by $\blank \otimes _\rho \sigma $. 
		\end{lem}
		\begin{proof}
			\cref{lem:Qsystem.module.tensor} (2) shows $\mathcal{F} \mathcal{G}(X) = X \otimes_\rho \rho \simeq X$. 
			Conversely, a similar consideration as in (2) shows $\rho \otimes_\rho \pi \simeq \pi$. Now applying $F$ in view of \cref{lem:bimodrholinear}, we get a natural unitary $\mathcal{G} \mathcal{F}(F)(\pi) = F(\rho) \otimes_\rho \pi \to F(\pi)$. 
            Similarly, for $\sigma \in \Obj \sD$, the composition functor $F \circ \sigma $ is sent by $\cF$ to $F(\rho \otimes_\rho \sigma) \cong F(\rho)\otimes_\rho \sigma$. Thus $\cF$ preserves the $\sD$-module structure.
		\end{proof}
		Under this categorical equivalence, the categorical ideal $\cK_\rho$ is sent to compact $\sC$-equivariant natural transformations. Together with \cref{lem:Qmodule.nonunital}, this proves \cref{thm:weakMorita} (1).

Next, we show \cref{thm:weakMorita} (2). 
Notice that \cref{lem:bimodrholinear} induces a $\sD$-module structure on $\Func_\sC(\sM,\sA)$, which is inherited by $\sA_\rho$ via \cref{lem:functor.Qsystem}. More explicitly, the tensor product of $X \in \Obj \sA_\rho$ and $\pi \in \Obj \sD$ is given by $X \otimes_\rho \pi$ on which $\rho$ acts from the right by $1 \otimes_\rho m_\pi$.

		By \cref{lem:functor.Qsystem} and \cref{lem:linking.bimodule}, we get an equivalence of (unital) $\Func_\sD(\sM^*,\sM^*)$-module categories
		\begin{align*}
			\Func_\sD(\sM^* , \Func_\sC(\sM, \sA)) \simeq & \Func_\sD(\sMod^l_\sD(\iiEnd(\rho)) , \Func_\sC(\sM, \sA))\\
			\simeq & (\Func_\sC (\sM, \sA))_{\iiEnd (\rho)} \cong (\sA_{\rho})_{\iiEnd (\rho)}
		\end{align*}
        sending a functor $F$ to 
        $(F(\rho))(\rho)$ via $\sM^* \simeq \sMod_\sC^r(\rho)$. 
    Note that the action of $\pi \in \Obj \sC$, identified with the $\iiEnd(\rho)$-bimodule $\rho \otimes \pi \otimes \rho$, on $(\sA_\rho)_{\iiEnd (\rho)}$ is given by $\blank \otimes _{\iiEnd(\rho)} (\rho \otimes \pi \otimes \rho)$.
		\begin{prop}
			There exists an equivalence of nonunital $\sC$-module categories $\sA \simeq (\sA_\rho)_{\iiEnd(\rho)}$, which gives a natural transform of functors $\id \xrightarrow{\sim} (\blank_\rho)_{\iiEnd(\rho)}$.
		\end{prop}
		\begin{proof}
			We show that the functor $\sA \to (\sA_\rho)_{\iiEnd(\rho)}$ given by $X \mapsto X \otimes \rho$ gives the desired categorical equivalence. 
			First, let us consider the forgetful functor 
			\[ 
			(\sA_\rho)_{\iiEnd ( \rho ) } \to \sA_{\iEnd ( \rho ) }, \quad ((X,m_X), \underline{m}_X) \mapsto (X,\underline{m}_X \cdot  p_{X,\iEnd(\rho)}),
			\]
			where $m_X$ and $\underline{m}_X$ denotes the action of $\rho$ and $\iiEnd(\rho)$, and $p_{X,\iEnd(\rho)}$ denotes the projection of \cref{lem:Qsystem.module.tensor} (1). 
			This is a categorical equivalence.
			Indeed, $m_X$ is recovered from $\underline{m}_X \cdot p_{X,\iEnd(\rho)}$ as $m_X = \underline{m}_X(1_X \otimes m_\rho^*)$ 
			by the commutativity of the diagram
			\begin{align*}
				\xymatrix@C=13ex{
					X \otimes \rho \ar[r]^-{1_X \otimes m_\rho^* \eta_\rho \otimes 1_\rho} \ar[rd]_{1_X \otimes m_\rho ^* } & X \otimes \rho \otimes \rho \otimes \rho \ar[r]^{\underline{m}_X \otimes 1_\rho} \ar[d]^{1_X \otimes 1_\rho \otimes m_\rho} & X \otimes \rho \ar[d]^{m_X} \\ 
					& X \otimes \rho \otimes \rho \ar[r] ^{\underline{m}_X}& X.
				}
	\end{align*}
	This shows that the functor is both fully faithful and essentially surjective.
	Besides, the composition functor $\sA \to (\sA_\rho)_{\iiEnd(\rho)} \to \sA_{\iEnd ( \rho ) }$ is also a categorical equivalence, with the inverse functor
	\[
	\sA_{\iEnd ( \rho ) } \to \sA, \quad Y \mapsto Y\otimes_{\iEnd ( \rho ) } \rho. 
	\]
	These categorical equivalences all preserve the $\sC$-module structures via the monoidal equivalence $\sC \simeq \sBimod^\sD (\iiEnd(\rho))$ given by $\pi\mapsto \rho\otimes\pi\otimes \rho$, which is checked as e.g.~ $(X \otimes \pi) \otimes \rho \simeq (X \otimes \rho) \otimes_{\iiEnd(\rho)} (\rho \otimes \pi \otimes \rho) $. 

    Finally, the naturality of the equivalence $\sA \simeq (\sA_{\rho})_{\iiEnd(\rho)}$ with respect to a $\sC$-module functor $(F,\mathsf{v}) \colon \sA \to \sB$ is obvious from the construction, as $F(X \otimes \rho) \in \Obj ((\sB_\rho)_{\iiEnd(\rho)})$ is identified with $F(X) \otimes \rho$ via $\mathsf{v}_{X,\rho}$.
\end{proof}

We finish this subsection by a remark on the $\sC$-C*-algebra picture of this weak Morita invariance, which also shows the remaining part of the proof of \cref{thm:Morita.invariance}.
	\begin{lem}\label{rem:Qsystem.algebra}
    Let $(A, \alpha,\fu)$ be a $\sC$-C*-algebra. We assume $(\alpha, \fu)$ is strict, i.e., $\alpha_{\pi \otimes \sigma}=\alpha_\pi \otimes_A \alpha_\sigma$ and $\fu_{\pi,\sigma} =1$ for any $\pi, \sigma \in \Obj \sC$.
    Let $\rho$ be a $Q$-system in $\sC$. Set $B:=\cK_\rho(\alpha_\rho)$. Then the following hold.
    \begin{enumerate}
        \item The following functor is a categorical equivalence;
    \[
    \Mod(B) \to \Mod(A)_\rho, \quad X \mapsto (X \otimes_B \alpha_\rho, 1_X \otimes \alpha_{m_\rho}).
    \]
    \item The map $\lambda_\rho \colon \alpha_\rho \to B$ given by $\lambda_\rho ( \xi ) =\alpha_{m_\rho}( \xi \otimes \blank)$ is bijective.
    \item For a C*-algebra $D$, there is an equivalence $\Mod(A)_\rho \boxtimes \Mod (D) \simeq \Mod(A \otimes D)_\rho$ as  $\sBimod(\rho)$-module categories. 
    \end{enumerate}
\end{lem}
\begin{proof}
    The claim (1) follows from \cref{thm:equivalence.category.algebra} and \cref{lem:Qmodule.nonunital}. For (2), the inverse map is explicitly given by $\lambda_\rho^{-1}(T) = T \alpha_{\eta_\rho}$ for $T\in \cK_\rho(\alpha_\rho)$. For (3), notice that the categorical equivalence $\Mod(A)_\rho \boxtimes \Mod(D) \simeq \Mod (\cK_\rho(\alpha_\rho) \otimes D) \simeq \Mod (\cK_\rho (\alpha_\rho \otimes D)) \simeq \Mod(A \otimes D)_\rho$ sends $(X,m_X) \boxtimes Y$ to $(X \otimes Y, m_X \otimes 1)$, and hence is $\sBimod(\rho)$-equivariant. 
\end{proof}

\subsection{Partial Takesaki--Takai duality}\label{section:partial:TT}
An application of weak Morita equivalence, in which a tensor category not coming from a compact quantum group is involved, is the following generalization of the Baaj--Skandalis--Takesaki--Takai duality. 

Let $\Gamma$ be a countable group and let $\Lambda$ be a finite abelian normal subgroup of $\Gamma$. 
As is discussed in \cref{ex:Qsystem.group}, the group algebra $\rho:=\bC[\Lambda]$ determines a $Q$-system in $\Hilbg$. 
Then the C*-tensor category $\sBimod (\rho)$ is equivalent to that of $\Gamma$-graded $\Lambda$-bimodules, namely, finite dimensional Hilbert spaces equipped with a $\Gamma$-grading and compatible $\Lambda$-actions from the left and the right (i.e., ${g_2} \cdot \sH_h \cdot {g_1} \subset \sH_{g_2 h g_1}$). 
This category is known to be equivalent to a $3$-cocycle twist of the semi-direct product group $\hat{\Lambda} \rtimes \Gamma/\Lambda$ with respect to the action $(\chi,g) \mapsto {}^g \chi := \chi(s(g)^{-1} \cdot \blank \cdot  s(g))$.  (cf.\ \cref{ex:group.action}), where $s \colon \Gamma/\Lambda \to \Gamma$ is a set-theoretic section.
\begin{prop}[{\cite{naiduCategoricalMoritaEquiavlence2007}*{Theorem 4.9}}]
	There is an equivalence of C*-tensor categories
	\[
	\sBimod(\rho) \simeq \operatorname{Hilb^f_{\hat{\Lambda} \rtimes \Gamma/\Lambda ,\omega}},
	\]
    where the $3$-cocycle $\omega \in Z^3(\hat{\Lambda} \rtimes \Gamma/\Lambda ; \bT)$ is defined by using a fixed section $s \colon \Gamma/\Lambda \to \Gamma$ as
	\begin{align}
		\omega((\chi_1 , g_1),(\chi_2 , g_2),(\chi_3 , g_3)) = \chi_3(s(g_1)s(g_2)s(g_1 g_2)^{-1}).\label{eqn:cocycle.Takesaki}
	\end{align}
\end{prop}
    Note that \eqref{eqn:cocycle.Takesaki} is written by the cup product as $\omega =\varphi (\nu \cup \theta)$, where $\nu \in Z^2(\hat{\Lambda} \rtimes \Gamma/\Lambda;\Lambda)$ is the pull-back of the extension class $\nu' \in Z^2(\Gamma/\Lambda;\Lambda)$ explicitly given by $\nu'(g,h):=s(g)s(h)s(gh)^{-1}$, $\theta \in Z^1(\hat \Lambda \rtimes \Gamma/\Lambda ; \hat \Lambda)$ is the $1$-cocycle given by $\theta(\chi,g):=\chi$, and $\varphi \colon \Lambda \otimes_{\mathbb Z} \hat \Lambda \to \bT$ denotes the pairing.

We just recall the construction of the above monoidal equivalence.
For each $(\chi,g) \in \hat{\Lambda} \rtimes \Gamma/\Lambda $, define $\sH_{\chi ,g} := \ell^2(s(g) \Lambda) \in \sBimod(\rho)$, equipped with
\begin{itemize}
	\item the $\Gamma$-grading given by $\ell^2(s(g) \Lambda) = \bigoplus_{h \in s(g) \Lambda} \bC \delta_h$, and
	\item the left and right $\Lambda$-actions given by $(x \xi y)(h) = \chi(x) \xi(x^{-1} h y^{-1})$.
\end{itemize}
As $\sH_{\chi,g}$ exhausts all irreducible objects in $\sBimod(\rho)$, we obtain a unitary categorical equivalence $F \colon \operatorname{Hilb^f_{\Gamma/\Lambda \times \hat \Lambda,\omega}} \to \sBimod(\rho)$. Moreover, the natural unitary
    \[
    \mathsf{v}_{(\chi_1,g_1), (\chi_2,g_2)} \colon \sH_{\chi_1,g_1} \otimes_{\bC[\Lambda]} \sH_{\chi_2,g_2} \to \sH_{\chi_1 \cdot ({}^{g_1} \chi_2) , g_1 g_2},
    \]
given by $\delta_{s(g_1)x_1} \otimes \delta_{s(g_2)x_2} \mapsto \chi_2(x_1) \delta_{s(g_1)x_1s(g_2)x_2}$ for any $x_1,x_2 \in \Lambda$, makes $(F,\mathsf{v})$ a tensor functor.

We make \cref{thm:Morita.invariance} for this weak Morita equivalence more explicit, especially in the case that $\sA$ comes from a $\Gamma$-C*-algebra $A$ as $\sA =\Mod(A)$ (\cref{ex:group.action}). Let $\Lambda $ be as above. 
We use the same letter $\rho$ for the $Q$-system $\bC[\Lambda]$ in this reversed tensor category. 

The $Q$-system $(\alpha_\rho , \alpha_{m_\rho}\fu_{\rho,\rho})$ in $\Bimod(A)$ is isomorphic to the crossed product $\bigoplus_{t \in \Lambda} {}_{\alpha_t} A \cong \Lambda \ltimes_{\alpha} A$. 
By \cref{rem:Qsystem.algebra}, $\Mod(A)_\rho$ is equivalent to $\Mod(\Lambda \ltimes A)$, on which $\sBimod(\rho)\simeq (\Hilb^{\rm f}_{\hat{\Lambda}\rtimes\Gamma/\Lambda,\omega})^{\rm rev}$ acts. 
The corresponding $(\hat{\Lambda}\rtimes \Gamma/\Lambda,\omega)$-action (see \cref{defn:twisted.action}), denoted by $(\hat{\alpha}, \hat{\fu})$, is described by 
\begin{align} \begin{split}
		\hat{\alpha}_{\chi,g} (a u_x) :=& \chi(x)\alpha_{s(g)}(a) u_{s(g)xs(g)^{-1}}
		\\
		\hat{\fu}_{(\chi_1,g_1), (\chi_2,g_2)} :=& u_{s(g_1)s(g_2)s(g_1g_2)^{-1}}
	\end{split} \label{eqn:twisted.action.TT}
\end{align}
for $x \in \Lambda$ and $a\in A$. 

With this identification, an essential Hilbert $A$-$B$ bimodule $\sfE=(E,\phi)$ with a $\Gamma$-action $\gamma$ corresponds to 
\[ \Lambda \ltimes \sfE := (\Lambda \ltimes E, \tilde{\phi}, \tilde{\bbmv} ) \in \Bimod_{\sBimod(\rho)} (\Lambda \ltimes A , \Lambda \ltimes B),  \]
where the essential Hilbert $(\Lambda \ltimes A)$-$(\Lambda \ltimes B)$-bimodule $(\Lambda \ltimes E, \tilde{\phi})$ with $\tilde{\bbmv}$ corresponding to the $(\hat{\Lambda}\rtimes\Gamma/\Lambda,\omega)$-action as the corner of $\Lambda\ltimes \cK(E\oplus B)$ (cf.~\cite{kasparovEquivariantKKTheory1988}).

For $T \in \cL(E)$, we write $\widetilde{T} \in \cL(\Lambda \ltimes E)$ for the operator given by $(\widetilde{T}\xi)(g)=T (\xi(g))$.

In the following theorem, the $(\Hilb_{\hat{\Lambda}\rtimes \Gamma/\Lambda , \omega}^{\rm f})^{\rm rev}$-equivariant Kasparov category is simply written as $\Kas^{\hat{\Lambda}\rtimes \Gamma/\Lambda, \omega}$. This is a notation that will be introduced in \cref{section:group.cocycle}.
\begin{thm}\label{thm:partial.TakesakiTakai}
	Let $\Gamma$ be a countable group and let $\Lambda$ be its finite abelian normal subgroup. 
	For separable $\Gamma$-C*-algebras $A$ and $B$, the crossed product by $\Lambda $ gives the descent functor
	\[
	j^{\Lambda} \colon \KK^\Gamma (A,B) \xrightarrow{\cong} \KK^{\hat{\Lambda} \rtimes \Gamma /\Lambda , \omega } (\Lambda \ltimes A, \Lambda \ltimes B), \quad j^{\Lambda}[\sfE,F]:=[\Lambda \ltimes \sfE, \tilde{F}],
	\]
	which induces an equivalence of categories $\Kas^{\Gamma } \cong \Kas^{\hat{\Lambda} \rtimes \Gamma /\Lambda, \omega}$. 
\end{thm}
\begin{proof}
    It is observed in the same way as \cite{kasparovEquivariantKKTheory1988}*{Theorem 3.11} that the above $j^\Lambda$ gives a well-defined functor. As is seen above, it is an extension of the equivalence given in the proof of \cref{thm:Morita.invariance}.
\end{proof}
In particular, \cref{thm:partial.TakesakiTakai} for $\Lambda = \Gamma$ is the same thing as the Baaj--Skandalis--Takesaki--Takai duality  \cite{baajCalgebresHopfTheorie1989}.
		
\begin{ex}
	In the case of $\Gamma = \bZ_4$ and $\Lambda =\bZ_2$, the associated $3$-cocycle  $\omega \in H^3(\bZ_2 \times \bZ_2; \bT)$ is the generator of $H^2(\bZ_2;\bZ_2) \otimes H^1(\bZ_2;\bZ_2) \cong \bZ_2$. For this $\omega$, we get an equivalence of categories $\Kas^{\bZ_4} \cong \Kas^{\bZ_2\times \bZ_2, \omega}$ given by $A \mapsto \bZ_2 \ltimes A$. 
\end{ex}

\section{3-cocycle twist of discrete groups}\label{section:group.cocycle}
In this section, we study the equivariant KK-theory of $\Hilb_{\Gamma, \omega}^{\rm f}$, the tensor category of finite dimensional $\Gamma$-graded Hilbert spaces with the associator twisted by $\omega \in Z^3(\Gamma , \bT)$. 
Under the assumption that $\Gamma$ satisfies the property $\gamma=1$, the equivariant Kasparov category (denoted by $\Kas^{\Gamma, \omega}$) is shown to be unchanged $\omega$ is replaced by another one with the assumption $\beta \omega =0$. 
In particular, a Baum--Connes type result holds for this $\Kas^{\Gamma, \omega}$. 

\subsection{Actions of 3-cocycle twists of discrete groups}
Let $\Gamma$ be a discrete group and let $\omega \in Z^3(\Gamma;\bT)$. 
We consider the associated tensor category $(\Hilb_{\Gamma, \omega}^{\rm f})^{\rm rev}$ of $\Gamma$-graded Hilbert spaces with the associator twisted by $\omega$ (cf.\ \cref{ex:cocycle.twist}) with the reversed tensor product. 
As a generalization of \cref{ex:group.action}, an action of it on a C*-algebra $A$ corresponds to an \emph{anomalous action} of $\Gamma$ in the sense of  ~\cite{jonesRemarksAnomalousSymmetries2021}. 
Let $(A,\alpha,\fu)$ be a $(\Hilb_{\Gamma,\omega})^{\rm rev}$-C*-algebra $(A,\alpha,\fu)$ realized by endomorphisms. 
Then, each $\alpha_g$ is identified with an automorphism $\alpha_g \in \Aut(A)$, and $\overline{\fu}_{g,h}:=\fu_{h,g}^*$ is a unitary in $\cM(A)$ satisfying the relation $\alpha_g \circ \alpha_h = \Ad(\bar{\fu}_{g,h})\alpha_{gh}$ and the cocycle condition twisted by $\omega$ in the following way. 
\begin{df}[{\cite{jonesRemarksAnomalousSymmetries2021}*{Definition 1.1}}]\label{defn:twisted.action}
	A (left) \emph{$(\Gamma , \omega) $-C*-algebra} consists of a triple $(A, \alpha, \fu )$, where $A$ is a C*-algebra, and $\alpha \colon \Gamma \to \Aut (A)$ and $\fu \colon \Gamma \times \Gamma \to \cU(\cM (A))$ are maps satisfying the relations  
		\begin{align}\begin{split}
			\alpha_g \alpha_h &= \Ad (\fu_{g,h}) \circ \alpha_{gh}, \\
			\fu_{g,h}\fu_{gh,k}&= \omega(g,h,k) \alpha_g(\fu_{h,k})\fu_{g,hk} . 
		\end{split} \label{eq:cocycle} \end{align} 
\end{df}

\begin{ex}\label{ex:Z2.cocycle}
	When $\Gamma = \bZ/2\bZ$, the third cohomology $H^3(\Gamma; \bT)$ is isomorphic to $\bZ/2\bZ$ and the nontrivial $3$-cocycle is given by
	\[
	\omega(g,h,k) = \begin{cases} 1 & \text{ if $ghk = 0$,} \\ -1 & \text{ if $g = h = k=1$.} \end{cases}
	\]
	In this case, a $(\Gamma,\omega)$-action on a C*-algebra $A$ corresponds to a pair $(\alpha, \fu)$ where $\alpha$ is an automorphism of $A$ and $\fu = \fu_{1,1} \in \cM(A)$ is a unitary such that $\alpha^2 = \Ad(\fu)$, $\alpha(\fu) = -\fu$.
\end{ex}
\begin{ex}
    As in \cref{section:partial:TT}, let $\Gamma$ be a discrete group and let $\Lambda $ be its finite abelian normal subgroup. Then the crossed product $B:=\Lambda \ltimes A$ has the structure of $(\hat{\Lambda} \rtimes \Gamma /\Lambda  , \omega)$-C*-algebra given by $\beta_{\chi,g}$ and $\fv_{(\chi_1,g_1), (\chi_2,g_2)}$ as \eqref{eqn:twisted.action.TT}. 
\end{ex}
Many other examples besides them are found in \cite{jonesRemarksAnomalousSymmetries2021}.

If the $3$-cocycle $\omega$ is a coboundary on $\cZ\cM(A)$ in the following sense, a $(\Gamma, \omega)$-action is constructed from a cocycle action of $\Gamma$. 
Let $X$ be a locally compact Hausdorff space equipped with a $\Gamma$-action. 
We say that a $(\Gamma, \omega)$-C*-algebra $(A,\alpha,\fu)$ is a \emph{$(\Gamma \ltimes X,\omega)$-algebra} if it is equipped with an essential $*$-homomorphism $C_0(X) \to \cZ \cM(A)$ that is $\Gamma$-equivariant (note that $\Gamma$ acts on the center $\cZ \cM(A)$ as a genuine action).
Let $C(X; \bT)$ denote the abelian group of $\bT$-valued continuous functions on $X$ and let $i \colon \bT \to C(X;\bT)$ denote the homomorphism mapping $z \in \bT$ to the constant function $z \cdot 1_X$. 
We also assume that the $3$-cocycle 
    \[
	p_X^*\omega \in Z^3(\Gamma  ; C(X; \bT) ), \quad \text{where $p_X \colon X \to \pt$,}  
	\]
is a coboundary, i.e., there is $\tau \in C^2(\Gamma ; C(X;\bT))$ such that 
	\[
	\tau (gh, k) \tau(g,h) \tau(g,hk)^{-1}\alpha_g(\tau(h,k)).
	\]
Then, by letting $\fu_{\tau , g,h}:=\tau(g,h)\fu_{g,h} \in \cU(\cM (A))$, the triple $(A, \alpha, \fu_\tau)$ forms a $(\Gamma, \omega)$-C*-algebra. 
Conversely, for a $(\Gamma, \omega)$-C*-algebra $(A, \alpha, \fu)$ and a $\Gamma$-space $X$ as above, the triple $(A,\alpha, \fu_{\tau^{-1}})$ is a cocycle action of $\Gamma$ by putting $\fu_{\tau^{-1} ,g,h}:=\tau(g,h)^{-1}\fu_{g,h}$.

\subsection{Proper \texorpdfstring{$(\Gamma, \omega)$}{(Gamma,omega)}-actions}
Let $\Gamma$ be a discrete group and let $\omega \in Z^3(\Gamma;\bT)$. 
We say that a $(\Gamma,\omega)$-algebra is \emph{proper} if it is $(\Gamma \ltimes X,\omega)$-C*-algebra for some locally compact proper $\Gamma$-space $X$. 
In this subsection, we discuss a sufficient condition that $\omega $ is a coboundary on any proper $(\Gamma,\omega)$-C*-algebra.

For this, we use a model of the equivariant cohomology group $H_\Gamma^*(X;\cA)$ via the simplicial--\v{C}ech double complex, where $\cA$ be either the $\sfA$-valued constant sheaf or the sheaf of $\sfA$-valued continuous functions $\underline{\sfA}$ for either $\sfA=\bZ, \bR, \bT$.
Let $\fU = \{ U_i\}_{i \in I}$ be a $\Gamma$-equivariant open cover, i.e., $\Gamma$ acts on the index set $I$ and $U_{gi} = \alpha_g(U_i)$ for any $i \in I$ and $g \in \Gamma$. 
The group $\check{H}_\Gamma^*(\fU ; \cA)$ is defined as the cohomology of the double complex
	\[
	C^p(\Gamma ; \check{C}^q(\fU, \cA) ) := \mathrm{Map} \bigg( \Gamma ^p , \prod _{i_0, \cdots, i_q} C(U_{i_0, \cdots, i_q} , \cA) \bigg), 
	\]
with the differential $d_\Gamma + (-1)^p\delta_\fU $, where 
    \begin{align*}
	(d_\Gamma \varphi) (g_1, \cdots, g_{p+1}; i_0, \cdots, i_q):= & g_1^*(\varphi (g_2, \cdots , g_{p+1} ; g_1^{-1}i_0 , \cdots, g_1^{-1}i_{q})) \\
	& + \sum\limits_{j=1}^{p} (-1)^j \varphi (g_1, \cdots, g_jg_{j+1} , \cdots, g_{p+1}; i_0, \cdots, i_q) \\
	&+ (-1)^{p+1} \varphi(g_1, \cdots, g_p;i_0, \cdots, i_q ), \\
	(\delta_\fU \varphi) (g_1, \cdots , g_p; i_0, \cdots, i_{q+1}) := & \sum_{j=0}^{q+1} (-1)^j  \varphi(g_1, \cdots , g_p; i_0, \cdots,i_{j-1}, i_{j+1}, \cdots,  i_{q+1}) |_{U_{i_0, \cdots, i_{q+1}}}.
	\end{align*}
The group $\check{H}^*_\Gamma (X;\cA)$ is defined by the colimit $\varinjlim _{\fU } \check{H}^*_\Gamma (\fU ; \cA)$ with respect to refinement of $\fU$. 
If a $\Gamma$-equivariant cover $\fU$ is good, i.e., each finite intersection $U_{i_0, \cdots, i_k}:= U_{i_0} \cap \cdots \cap U_{i_k}$ is contractible (this assumption is satisfied if $X$ is a manifold or a simplicial complex, on which $\Gamma$ acts isometrically, and each $U_i$ is a geodesic ball), then $\check{H}^*_\Gamma (X;\cA) \cong \check{H}^*(\fU ; \cA)$ holds. 

We just remark without proof that the above group is isomorphic to the groupoid cohomology $H\*(\Gamma \ltimes X; \cA)$ in the sense of \cites{haefligerDifferentialCohomology2011,tuGroupoidCohomologyExtensions2006}. 
In particular, for a $\Gamma$-CW-complex $X$ and a constant sheaf $\sfA$, the group $\check{H}^*_\Gamma(X ; \sfA)$ is isomorphic to the Borel equivariant cohomology $H^*(E\Gamma \times_\Gamma X ; \sfA )$. 

\begin{rem}\label{rem:pull.isom}
    As above, let $\cA$ be either a constant sheaf or a sheaf of continuous functions.
    If a $\Gamma$-equivariant map $f \colon X \to Y$ induces isomorphisms $f^* \colon H^k(Y;\cA) \to H^k(X; \cA)$ for $k \leq n$, then the mapping cone double complex of $f^* \colon C^*(\Gamma ; C^*(\fU_Y ; \cA)) \to C^*(\Gamma ; C^*(\fU_X ; \cA))$ has vanishing $E_1$-page for $q \leq n$. 
    Therefore, a comparison of the spectral sequence (see the lemma below) shows that $f^* \colon H^k_\Gamma (Y;\cA) \to H^k_\Gamma (X;\cA)$ are also isomorphic for $k \leq n-1$. 
\end{rem}

\begin{lem}
	The $E_2$-page of the spectral sequence of the double complex $C^p(\Gamma ; \check{C}^q(\fU, \underline{\bT}) )$ with respect to the horizontal filtration is  
	\[
	E_2^{pq} \cong H^p(\Gamma ; \check{H}^q(\fU ; \underline{\bT})) \cong 
	\begin{cases}
		H^p(\Gamma ; H^{q+1}(X ; \bZ)) & \text{ if $q \geq 1$,} \\
		H^p(\Gamma ; C(X;\bT)) & \text{ if $q=0$. }
	\end{cases}
	\]
\end{lem}
\begin{proof}
	We determine the group $\check{H}^q(\fU ; \underline{\bT}) \cong \check{H}^q(X;\underline{\bT})$. Since $\underline{\bR}$ is a soft sheaf, we have 
	\[
	\check{H}^q(X, \underline{\bR}) \cong 
	\begin{cases}
	C(X; \bR) & \text{ if $q=0$}, \\
	0 & \text{otherwise.}
	\end{cases} \]
	The long exact sequence associated to $0 \to \underline{\bZ} \to \underline{\bR} \to \underline{\bT} \to 0$ shows that $\check{H}^q(X; \underline{\bT}) \cong H^{q+1}(X;\bZ)$ for $q \geq 1$, and that 
	\[
	0 \to C(X;\bR)/C(X;\bZ) \to C(X;\bT) \cong \check{H}^0(\fU;\underline{\bT}) \to H^1(X;\bZ) \to 0
	\]
	is exact. Here the isomorphism at the middle comes from the fact that $\check{H}^0$ coincides with the global section functor. 
\end{proof}
		
Recall that we are interested in the group $H^3(\Gamma  ; C(X;\bT)) \cong E_2^{3,0}$.
We define the map
	\begin{gather*}
	    \chi \colon C^p(\Gamma ; C(X;\bT)) \to C^p(\Gamma, \check{C}^0(\fU ,\underline{\bT})), \\
	(\chi \varphi) (g_1, \cdots, g_p ; i_0) = \varphi (g_1, \cdots, g_p)|_{U_{i_0}} .
	\end{gather*} 
This is a cochain map and hence induces a map $H^p(\Gamma ; C(X;\bT)) \to H^p_\Gamma ( X; \underline{\bT} )$. 
\begin{lem}
	Assume that $H^k(X; \bZ)=0$ holds for $k=2,3$. Then the map 
	\[
	\chi_* \colon H^3(\Gamma ; C(X;\bT)) \to H^3_\Gamma (X ; \underline{\bT} )
	\]
    is injective.  
\end{lem}
\begin{proof}
    Let $\cE_r^{p,q}$ denote the spectral sequence of the double complex $C^p(\Gamma ; C(X;\bT))$ supported on the row $\cE_r^{p,0}$. Then the cochain map $\chi$ induces $\chi_* \colon \cE_r^{p,q} \to E_r^{p,q}$. 
    On the other hand, as is indicated in \cref{fig:E2.page}, the only non-trivial differentials with the range $E_r^{3,0}$ are the ones with the domain $E_2^{1,1}$ and $E_3^{0,2}$, both of which vanish by assumption. 
    Thus the group $E_\infty^{3,0}$, which is a quotient of $H^3_\Gamma (X;\underline{\bT})$, is isomorphic to $E_2^{3,0}$. Now we get the following commutative diagram:
    \[ 
    \xymatrix{
        \cE_2^{3,0} \ar[r]^\cong \ar[d]^{\chi_*}_\cong  &   \cE_\infty^{3,0}  \ar[d]^{\chi_*} & H^3(\Gamma ; C(X;\bT)) \ar[d]^{\chi_*} \ar[l]_{\cong \hspace{2em}}\\
        E_2^{3,0} \ar[r]^\cong &  E_\infty^{3,0}  & H^3_\Gamma (X;\underline{\bT}). \ar[l] 
    }
    \]
    This shows the injectivity of the right vertical map $\chi_*$ as desired.
\end{proof}

\begin{table}[t]
	\begin{tabular}{c|cccccccc}
    	$q=3$ &  && && && &  \\ 
    	$q=2$ & $H^0(\Gamma; H^3(X;\bZ))$  && && &&     &  \\  
    	$q=1$ &   && $H^1(\Gamma ; H^2(X;\bZ))$ && && &  \\ 
    	$q=0$ &  && && && $H^3(\Gamma ; C(X, \bT))$ &  \\ 
	\hline 
	$E_2^{pq}$ & $p=0$ && $p=1$ && $p=2$ && $p=3$ &
	\end{tabular}
	\caption{$E_2$-page of the spectral sequence.}
	\label{fig:E2.page}
\end{table}

		\begin{lem}
			Let $X$ be a finite dimensional proper $\Gamma$-simplicial complex. Then the Bockstein map for the coefficient sheaf $0 \to \underline{\bZ} \to \underline{\bR} \to \underline{\bT} \to 0$ gives an isomorphism $\beta \colon H^3_\Gamma (X ; \underline{\bT}) \cong H^4_\Gamma (X ; \bZ)$. 
			In particular, $p_X^* \omega =0$ for any locally compact proper $\Gamma$-space $X$ if $\beta \omega =0 \in H^4(\Gamma ; \bZ)$.
		\end{lem}
\begin{proof}
	The exact sequence of sheaves $0 \to \underline{\bZ} \to \underline{\bR} \to \underline{\bT} \to 0$ induces the long exact sequence
	\[ 
	\cdots \to \check{H}^*_\Gamma ( X ; \underline{\bZ}) \to \check{H}^*_\Gamma (X ; \underline{\bR}) \to \check{H}^*_\Gamma (X ; \underline{\bT}) \to H^{*+1}_\Gamma (X ; \underline{\bZ}) \to \cdots .
	\]
	We show that $\check{H}^*_\Gamma ( X ; \underline{\bR}) =0$ by considering the spectral sequence of $C^p(\Gamma ; \check{C}^q(\fU;\underline{\bR}))$. 
	Since $\check{H}^0(X; \underline{\bR})\cong C(X;\bR)$ and $\check{H}^q(X; \underline{\bR}) \cong 0$ for $q \geq 1$, we have $H^q_\Gamma(X ; \underline{\bR}) \cong H^q(\Gamma ; C(X; \bR))$. 
    For $0  \leq n \leq \dim X$, let $X^{(n)}$ denote its $n$-skeleton and let $U^{(n)}:=X^{(n)} \setminus X^{(n-1)} \cong \mathop{\mathrm{int}}\Delta^n \times \Gamma / K_n$. 
    Then we have $H^q(\Gamma ; C_0(U^{(n)},\bR)) \cong H^q(K_n ; C_0(\Delta^n;\bR)) \cong 0$ and $H^q(\Gamma ; C(X^{(0)};\bR)) \cong H^q(K_n ; ;\bR) \cong 0$ for $q \geq 1$. 
	An iterated use of the long exact sequence associated to $0 \to C_0(U^{(n)}  ; \bR) \to C(X^{(n)}, \bR) \to C(X^{(n-1)},\bR) \to 0$ proves that $H^q(\Gamma;C(X^{(n)},\underline{\bR})) = 0 $ for any $n \geq 1$ and $q \geq 1$. 
			
	To see the second claim, consider the $5$-skeleton $\underline{E}\Gamma^{(5)}$ of the universal proper $\Gamma$-space $\underline{E}\Gamma$. 
	We remark that a model of $\underline{E}\Gamma$ constructed by the Milnor construction \cite{baumClassifyingSpaceProper1994}*{Appendix 1} is a $\Gamma$-simplicial complex.
	Then the right vertical map of the following commutative diagram is an isomorphism;
	\[
	\xymatrix{
		H^3(\Gamma; \bT) \ar[r] ^-{\cong} \ar[d]^\beta 
		& H^3_\Gamma (\underline{E}\Gamma ; \bT) \ar[r] \ar[d]^\beta & H^3_\Gamma (\underline{E}\Gamma ;\underline{\bT}) \ar[r]^\cong \ar[d]^\beta & H^3_\Gamma (\underline{E}\Gamma^{(5)};\underline{\bT} )\ar[d]^\beta_\cong \\
		H^4(\Gamma ; \bZ) \ar[r]^-{\cong} & H^4_\Gamma(\underline{E}\Gamma  ;\bZ) \ar@{=}[r] & H^4_\Gamma (\underline{E} \Gamma ; \bZ) \ar[r]^\cong & H^4_\Gamma (\underline{E}\Gamma^{(5)};\bZ).
	}
	\]
	Moreover, the left and the right horizontal maps are isomorphisms by \cref{rem:pull.isom}. Thus, $\beta \omega =0$ implies that $p_{\underline{E}\Gamma}^*\omega =0$, and hence $p_X^*\omega = f_X^* p_{\underline{E}\Gamma}^* \omega =0$ for any proper $\Gamma$-space $X$ with the classifying map $f_X \colon X \to \underline{E}\Gamma$.
\end{proof}

\begin{rem}
	We recall the relation between the universal coefficient theorem and the Bockstein map. The UCT theorem states that 
	\[
	H^p(G; \bT) \cong H^p(G ; \bZ) \otimes_\bZ \bT \oplus \mathrm{Tor}^1_\bZ (H^{p+1}(G; \bZ) , \bT). 
	\]
    This isomorphism is given by the flat resolution $0 \to \bZ \to \bR \to \bT \to 0$ of $\bT$. Therefore, $\mathrm{Tor}^1_\bZ (H^{p+1}(G; \bZ) , \bT)$ is isomorphic to the torsion part $\ker (H^{p+1}(G;\bZ) \to H^{p;1}(G;\bR))$, and the map $H^p(G;\bT) \to \mathrm{Tor}^1_\bZ (H^{p+1}(G; \bZ) , \bT)$ is given by the Bockstein map $\beta$.
\end{rem}
		
\begin{prop}
	Let $X$ be a locally compact proper $\Gamma$-space, and $A$ is a $(\Gamma \ltimes X, \omega)$-C*-algebra. 
	We also assume that $\beta \omega =0 \in H^4(\Gamma ; \bZ)$ (this holds e.g. when $H^4(\Gamma;\bZ)$ has no torsion). Then there exists $\tau \in C^2(\Gamma, C(X; \bT))$ such that $d_\Gamma \tau = \omega$, and hence $(A, \alpha, \fu_{\tau^{-1}})$ is a cocycle action. 
\end{prop}

\subsection{Baum--Connes property for 3-cocycle twists of torsion-free groups}
Now we investigate the structure of the $(\Hilb_{\Gamma, \omega}^{\rm f})^{\rm rev}$-equivariant Kasparov category, simply denoted by $\Kas^{\Gamma, \omega}$ in this subsection, under the assumption that $\Gamma $ is torsion-free and satisfies the property $\gamma =1$.

\begin{prop}
	Let $\Gamma$ be a discrete group and let $\omega_1, \omega_2 \in H^3(\Gamma; \bT)$. Then there is a bifunctor 
	\[
	\Kas^{\Gamma, \omega_1} \times \Kas ^{\Gamma, \omega_2} \to \Kas ^{\Gamma , \omega_1 + \omega_2}. \]
\end{prop}
\begin{proof}
	The desired functor is given by the composition of the exterior tensor product
	\[
	\blank \otimes \blank \colon  \Kas^{\Gamma, \omega_1} \times \Kas ^{\Gamma, \omega_2} \to \Kas ^{\Gamma \times \Gamma , (\omega_1  ,\omega_2)}, 
	\]
	and the restriction of the twisted group action to the diagonal subgroup $\Gamma \leq \Gamma \times \Gamma$. 
\end{proof}

Here, we reformulate the induction-restriction adjunction introduced in \cref{prop:adjunction.tensor} in terms of anomalous actions.
A $(\Gamma , \omega)$-C*-algebras $A$ is said to be trivially contractible if the underlying C*-algebra $A$ is $\KK$-contractible. We write $\cT \cC $ for the full subcategory of $\Kas^{\Gamma,\omega}$ consisting of compactly contractible $(\Gamma,\omega)$-C*-algebras. 		
We also consider the localizing full subcategory $\langle \cT \cI \rangle _{\loc} \subset \Kas^{\Gamma,\omega}$ generated by trivially induced $(\Gamma , \omega)$-C*-algebras, i.e., $(\Gamma, \omega)$-C*-algebras of the form $\Ind ^{\Gamma, \omega} A:=(c_0(\Gamma ) \otimes A , \lambda \otimes \id_A, \tau )$ for some separable C*-algebra $A$. 
Here $\tau$ denotes a $2$-cochain with $d \tau = p_\Gamma ^*\omega \in C^3(\Gamma ; C(\Gamma ; \bT))$. 

By \cref{prop:adjunction.tensor} (or by the same unit and counit as in \cite{meyerBaumConnesConjectureLocalisation2006}), the restriction-induction adjunction 
	\[
	\KK^{\Gamma, \omega} (\Ind^{\Gamma, \omega } A , B) \cong \KK(A, \Res_{\Gamma, \omega} B ) 
	\]
holds. Hence the relative homological algebra of \cite{meyerHomologicalAlgebraBivariant2010} gives a semi-orthogonal decomposition $(\langle \cT \cI \rangle _{\rm loc}, \cT \cC)$. By definition, the tensor product of $B \in \Kas^{\Gamma , \omega} $ with $A \in \langle \cT \cI \rangle_\loc \subset \Kas^\Gamma$ is contained in $\langle \cT \cI \rangle_\loc \subset \Kas^{\Gamma , \omega}$, as well as the tensor product with $A \in \cT \cC \subset \Kas^{\Gamma}$ is contained in $\cT \cC \subset \Kas^{\Gamma , \omega}$.

\begin{thm}\label{thm:untwist.Kasparov}
	Let $\Gamma$ be a countable discrete group. 
	\begin{enumerate}
		\item The C*-tensor category $\Hilbgo$ satisfies $\cT \cC=0$ if $\Gamma$ does (e.g.\ if $\Gamma$ is torsion-free and has the Haagerup property \cite{higsonTheoryKKTheory2001}).
		\item If $\Gamma$ has the Haagerup property and $\beta \omega =0 \in H^4(\Gamma; \mathbb{Z})$, then there exists a categorical equivalence $\Kas^{\Gamma,\omega} \simeq \Kas^\Gamma$ preserving the restriction functor.
	\end{enumerate}
\end{thm}
\begin{proof}
	The claim (1) follows from the above observation. Indeed, for any $A \in \Kas^{\Gamma, \omega }$, we have $A = \bC \otimes A \in \langle \cT \cI \rangle_\loc \in \Kas^{\Gamma, \omega}$ since $\bC \in \Kas^{\Gamma} = \langle \cT \cI \rangle _\loc$ by the assumption $\cT \cC=0$ and the semi-orthogonal decomposition 
 $(\langle \cT \cI \rangle _{\rm loc}, \cT \cC)$ of $\Kas^\Gamma$.  
			
	To see (2), let $\cA_\Gamma$ be the Higson--Kasparov $\Gamma$-C*-algebra \cite{higsonTheoryKKTheory2001}, i.e., an $(\Gamma \ltimes X)$-C*-algebra for some locally compact proper $\Gamma$-space $X$ that is $\KK^\Gamma$-equivalent to $\bC$. By assumption, we can take $\tau \in C^2(\Gamma ; C(X;\bT))$ such that $d_\Gamma\tau = p_X^* \omega$. 
    The functors 
	\begin{align*}
		\Kas^{\Gamma} \to \Kas^{\Gamma ,\omega}&, \quad (A,\alpha,\fu) \mapsto (A \otimes \cA_\Gamma , \alpha \otimes \gamma, \fu \otimes \tau), \\
		\Kas^{\Gamma , \omega} \to \Kas^{\Gamma}&, \quad (A,\alpha,\fu) \mapsto (A \otimes \cA_\Gamma , \alpha \otimes \gamma,  \fu \otimes \tau^{-1}), 
	\end{align*}
	are mutually inverse since $(\tau \otimes 1) + (1 \otimes  \tau^{-1}) \in C^2(\Gamma; C(X \times X; \bT))$ is a null-cohomologous $\bT$-valued $2$-cocycle. This is because $\tau$ is the pull-back of $\tilde{\tau} \in C^2(\Gamma ; C(\underline{E}\Gamma , \bT) )$ and $[\tilde{\tau} \otimes 1] = [1 \otimes \tilde{\tau}] \in H^3_\Gamma(\underline{E}\Gamma \times \underline{E}\Gamma ; \bT)$ since they coincide after pulling back to the diagonal.        
\end{proof}
		
		\begin{ex}
		    For $\Gamma = \bZ^3$, we have $H^3(\Gamma ; \bT) \cong \bT$ and $H^4(\Gamma ; \bZ) \cong 0$. Therefore, \cref{thm:untwist.Kasparov} shows the categorical equivalence $\Kas^{\bZ^3, \omega } \simeq \Kas^{\bZ^3}$ for any $\omega \in H^3(\Gamma ; \bT)$.
		\end{ex}

\appendix
		
\section{Toolkit for \texorpdfstring{$\sC$}{C}-module categories and \texorpdfstring{$\sC$}{C}-C*-algebras}
Here we summarize $\sC$-equivariant versions of basic materials of the theory of C*-algebra, which form the basis in operator $\K$-theory and $\KK$-theory.

\subsection{Ideal, quotient and exact sequence}\label{subsubsection:ideal.exact}
    We discuss the $\sC$-equivariant version of the correspondence of ideals and short exact sequences of C*-algebras. 
		
    For a nonunital C*-category $\sA$, its \emph{ideal} is a family of closed subspaces $\cK_\sI(X,Y) \subset \cK_\sA(X,Y)$ for $X,Y \in \Obj \sA$ satisfying the conditions of categorical ideals, i.e., $\cK_\sI (X,Y) \circ \cK_\sA(Y,Z) \subset \cK_\sI(X,Z)$ and $\cK_\sA(Z,X) \circ \cK_\sI(X,Y) \subset \cK_\sI(Z,Y)$. Note that it is automatically $\ast$-closed, i.e., $\cK_\sI(X,Y)^* = \cK_{\sI}(Y,X)$. 
    An ideal again forms a nonunital C*-category by letting $\Obj \sI$ the class of objects $X \in \Obj \sA$ with $\cK_\sI(X) = \cK_{\sA}(X)$, and setting $\cL_\sI(X,Y):= \cL_\sA(X,Y)$. 

	If $\sI \subset \sA$ is an ideal, the quotient category $\sA / \sI$ is defined by the idempotent completion of the category $(\sA/\sI)_0$ defined by $\Obj (\sA/\sI)_0= \Obj \sA$, $\cK_{(\sA/\sI)_0}(X,Y):= \cK_\sA(X,Y)/\cK_{\sI}(X,Y)$ and $\cL_{(\sA/\sI)_0}(X,Y):= p_Y \cM(\cK_{\sA / \sI}(X \oplus Y))p_X$. 
	We note that, if a nonunital C*-category $\sA$ is separable, then so are its ideal $\sI$ and the associated quotient $\sA/\sI$.
	Note that an ideal of $\sA = \Mod(A)$ corresponds to an ideal of a separable C*-algebra $A$. 
	
	\begin{df}\label{def:exact.sequence.categories}
	An exact sequence of nonunital C*-categories is defined by a sequence
	\[ 
	0 \to \sJ \xrightarrow{\cE} \sA \xrightarrow{\cF} \sB \to 0, 
	\]
	where 
	\begin{enumerate}
		\item $\cE$ is a proper fully faithful functor, 
		\item $\cF$ is a proper full functor such that any $X \in \Obj \sB$ is a direct summand of $\cF(Y)$ for some $Y \in \Obj \sA$, and 
		\item for $X \in \Obj \sA$, $\cF(X) \cong 0$ if and only if $X \cong \cE(Z)$ for some $Z \in \Obj \sJ$.  
	\end{enumerate}
	\end{df}
	For an exact sequence of nonunital C*-categories, $\cK_\sI(X,Y):= \ker \cF_{X,Y}$ defines an ideal $\sI$ such that the nonunital C*-categories $\sI$ and $\sA/\sI$ are categorically equivalent to $\sJ$ and $\sB$ respectively. 
	Conversely, if $\sI $ is an ideal of $\sA$, then $0 \to \sI \to \sA \to \sA/\sI \to 0$ is an exact sequence. 
	
	Let $\sA$ be a nonunital $\sC$-module category. An ideal $\sI \subset \sA$ is said to be $\sC$-invariant if $X \otimes \alpha_\pi \in \Obj \sI$ for any $X \in \Obj \sI$. If $\sI$ is a $\sC$-invariant ideal of $\sA$, $\sI$ and $\sA/\sI$ are also endowed with the $\sC$-action.  
	An exact sequence of $\sC$-module categories is a sequence 
	\begin{align}
		0 \to \sJ \xrightarrow{(\cE,\mathsf{v})} \sA \xrightarrow{(\cF,\mathsf{w})} \sB \to 0   \label{eq:exact.Cmodule.category} 
	\end{align}
	of $\sC$-module proper functors such that the underlying sequence of nonunital C*-categories is exact in the above sense. 
	The associated ideal $\sI$ is $\sC$-invariant, and the induced $\sC$-module functors $(\cE, \mathsf{v}) \colon \sJ \to \sI$ and $(\cF,\mathsf{w}) \colon \sA/\sI \to \sB$ are equivalence of $\sC$-module categories. 
    This argument shows that $\sJ \xrightarrow{(\cE,\mathsf{v})} \sA$ is the equalizer of $\sA\xrightarrow{(\cF,\mathsf{w})} \sB$ and $\sA\xrightarrow{0}\sB$ in the category $\Ccat^\sC$ because so is the canonical inclusion $\sI\subset \sA$ by definition.

	We also discuss the $\sC$-C*-algebra picture of the same materials.
	\begin{df}
		Let $(A,\alpha,\fu)$ be a $\sC$-C*-algebra. The ideal $I \triangleleft A$ is $\sC$-invariant if $I \otimes_A \alpha_\pi$ becomes a Hilbert $I$-module, i.e., the inner product takes value in $I$, for each $\pi \in \Obj \sC$.   
	\end{df}
	For such $I$, the triple $(I, I \otimes_A \alpha_\pi, \fu)$ forms a $\sC$-C*-algebra such that the inclusion $\iota \colon I \to A$ together with the trivial cocycle is a $\sC$-$\ast$-homomorphism. 
	Moreover, the quotient C*-algebra $A/I$ is also endowed with a $\sC$-action $(\beta, \fv)$ in the following way: The bimodules are defined as
	\[
	\beta_\pi := \alpha_\pi \otimes _q A/I \in \Bimod(A/I),
	\]
	which is a priori a $A$-$A/I$ bimodule but the left $A$-action actually factors through $A/I$. The unitary cocycles are defined as
	\[ 
	\fv_{\pi,\sigma } \colon (\alpha_\pi \otimes_q A/I) \otimes_{A/I} (\alpha_\sigma \otimes A/I) \cong \alpha_\pi \otimes_A \alpha_\sigma \otimes_q B \xrightarrow{\fu_{\pi,\sigma}} \alpha_{\pi \otimes \sigma} \otimes_q A/I.
	\]
	
	An exact sequence of $\sC$-C*-algebra is defined by a sequence 
	\begin{align}
		0 \to J \xrightarrow{(\phi, \bbmv)} A \xrightarrow{(\psi, \bbmw)} B \to 0 \label{eq:exact.C.algebra}
	\end{align}
	such that the underlying sequence of C*-algebras is exact. 
	Then the ideal $I:=\ker \phi$ is $\sC$-invariant since
	\[ 
	I \otimes_I \alpha_\pi \otimes _q B = I \otimes_I \bbmv_\pi^*(q \otimes_B \beta_\pi) = \bbmv_\pi^*(I \otimes_I q \otimes_B \beta_\pi) =0. 
	\]
	Now the cocycle $\sC$-$\ast$-homomorphisms $(\phi, \bbmv)$ and $(\psi, \bbmv)$ induces the $\sC$-$\ast$-isomorphisms $I \cong J $ and $A/I \cong B$. 
	Conversely, if $I$ is a $\sC$-invariant ideal of $A$, then $0 \to I \to A \to A/I \to 0$ is an exact sequence of $\sC$-C*-algebras.
	
	We say that the exact sequence \eqref{eq:exact.Cmodule.category} equivariantly splits if there is a $\sC$-module functor $(\sS, s) \colon \sB \to \sA$ such that the composition $(\sF,f) \circ (\sS, s)$ is naturally isomorphic to $(\id_\sB, \mathbbm{1})$. 
	Correspondingly, we say that the exact sequence \eqref{eq:exact.C.algebra} equivariantly splits if there is a cocycle $\sC$-$\ast$-homomorphism $(s,\bbmx) \colon B \to A$ such that $(q,\bbmw)  \circ (s,\bbmx) = (\id_B, \mathbbm{1})$. We remark that \eqref{eq:exact.C.algebra} equivariantly splits if and only if so does the corresponding exact sequence of module categories. Indeed, if there is a splitting functor $(\sS, \bx) \colon \Mod(B) \to \Mod(A)$, corresponding to a cocycle $\sC$-$\ast$-homomorphism $(s,\bbmx)$, satisfies that $(q,\bbmw)  \circ (s,\bbmx)$ is unitarily equivalent to the identity map. By replacing $(s,\bbmx)$ with $(s,\bbmx) \circ ((q,\bbmw)  \circ (s,\bbmx))^{-1}$, we get a C*-algebraic splitting.

	\subsection{Fibered sum}\label{subsubsection:fiber.sum}
	Let $\sA_i$ ($i=1,2$) and $\sB$ be nonunital C*-categories. For a proper functor $\cF_i \colon \sA_i \to \sB$, the \emph{fibered sum} $\sA_1 \oplus_{\sB} \sA_2$ is the separable envelope (in the sense of \cref{rem:separable.envelope}) of the C*-category
	\begin{itemize}
		\item whose object is a triplet $(X_1,X_2,U)$, where $X_i \in \Obj \sA_i$ and $U$ is a unitary $\cF_1(X_1) \to \cF_2(X_2)$ and
		\item whose compact morphism ideal is
		\begin{align*}
			&\cK_{\sA_1 \oplus_\sB \sA_2}((X_1,X_2,U),(Y_1,Y_2,V)) \\
			=& \{(x_1,x_2) \in \cK_{\sA_1}(X_1,Y_1) \oplus \cK_{\sA_2}(X_2,Y_2) \mid \cF_1(x_1) = V^* \cF_2(x_2) U \}.
		\end{align*}				
	\end{itemize}
	When $\sA_i$ and $\sB$ are $\sC$-module categories and $\cF_i$ are $\sC$-module functors, a $\sC$-module structure can be imposed on $\sA_1 \oplus _{\sB} \sA_2$ under the assumption that $\cF_2$ satisfies (2) of \cref{def:exact.sequence.categories}. 
	We describe it in terms of $\sC$-C*-algebras.          
	
	Let $(A_i,\alpha_i,\fu_i)$ ($i=1,2$) and $(B, \beta, \fv)$ be $\sC$-C*-algebras and 
	let $(\phi_i,\bbmv _i) \colon A_i \to B$ ($i=1,2$) be cocycle $\sC$-$\ast$-homomorphisms. 
    We write $\phi_{i,\pi} \colon \alpha_{i,\pi} \to \beta_\pi$ for the maps given by $\phi_{i,\pi}(\xi):=\bbmv_{i,\pi}T_\xi \in \cK(B,\beta_\pi) \cong \beta_\pi$.
	The fibered sum 
	\[ 
	A_1 \oplus _B A_2:= \{ (a_1,a_2) \in A_1 \oplus A_2 \mid \phi_1(a_1)=\phi_2(a_2) \}
	\]
	will be equipped with the $\sC$-C*-algebra structure $(\alpha _1 \oplus_B \alpha_2 , \fu_1 \oplus_B \fu_2)$, where 
	\begin{align*}
		&(\alpha_1 \oplus _B \alpha_2)_\pi := \{ (\xi, \eta) \in \alpha_{1,\pi} \oplus \alpha_{2,\pi} \mid \phi_{1,\pi}(\xi) = \phi_{2,\pi}(\eta)  \}, \\
		&(\alpha_1 \oplus_B \alpha_2)_f:=\alpha_{1,f} \oplus _B \alpha_{2,f}, \\
		&(\fu_1 \oplus_B \fu_2)_{\pi, \sigma} ((\xi_\pi,\eta_\pi) \otimes_{A_1 \oplus_B A_2} (\xi_\sigma, \eta_\sigma)) := (\fu_{1,\pi,\sigma}(\xi_\pi \otimes_{A_1} \xi_\sigma), \fu_{2,\pi,\sigma}(\eta_\pi \otimes_{A_2} \eta_\sigma)).
	\end{align*}
    Note that $\Mod(A_1 \oplus _B A_2) \cong \Mod(A_1) \oplus_{\Mod(B)} \Mod(A_2)$ holds. 
	Indeed, a routine computation shows
    \[ 
    \bbmv_{1,\sigma} T_{\alpha_1(f)\xi_\pi} = \bbmv_{2,\sigma}T_{\alpha_2(f) \eta_\pi},  
    \quad
    \bbmv_{\pi \otimes \sigma}T_{\fu_{\pi,\sigma}(\xi_\pi \otimes_A \xi_\sigma) } =\bbmv_{\pi \otimes \sigma}T_{\fv_{\pi,\sigma}(\eta_\pi \otimes_A \eta_\sigma)},  
    \]
    for any $(\xi_\pi,\eta_\pi) \in (\alpha_1 \oplus_B \alpha_2)_\pi$, $f \in \Hom(\pi,\sigma)$ and $(\xi_\sigma,\eta_\sigma) \in (\alpha_1 \oplus_B \alpha_2)_\sigma$. 
    \begin{lem}
        Assume that $\phi_2 \colon A_2 \to B$ is surjective. Then, each $(\fu_1 \oplus_B \fu_2)_{\pi, \sigma}$ is a unitary. Consequently, the above data determines an $\sC$-action on $A_1 \oplus_B A_2$. 
    \end{lem}
    \begin{proof}
    First of all, we may reduce the problem to the case that $\phi_1$ is also surjective. Indeed, $B':=\phi_1(A_1)$, $\beta_\pi':=\Im (\phi_{1,\pi})$, and the restriction $\fv_{\pi,\sigma}'$ of $\fv_{\pi,\sigma}$ forms a $\sC$-C*-algebra that is isomorphic to $A_1 / \ker \phi_1$ (defined in \cref{subsubsection:ideal.exact}). By letting $A_2':= \phi_2^{-1}(B')$, $\alpha_{2,\pi}' :=\phi_{2,\pi}^{-1}(\beta_\pi')$, the restriction $\fu_{2,\pi,\sigma}'$ of $\fu_{2,\pi,\sigma}$ gives a unitary $\alpha_{2,\pi}' \otimes_{A_2'} \alpha_{2,\sigma}' \cong \alpha_{2,\pi \otimes \sigma}'$ forming a $\sC$-C*-algebra $(A_2',\alpha_2',\fu'_2)$. 
    Now, $A_1 \oplus_{B'} A_2'$ is the same thing as the desired fibered sum $A_1 \oplus_{B} A_2$.

    Now we show that the naturally induced map 
    \[ 
    (\alpha_{1,\pi} \oplus_B \alpha_{2,\pi}) \otimes_{A_1 \oplus _B A_2} (\alpha_{1,\sigma} \oplus_B \alpha_{2,\sigma}) \to (\alpha_{1,\pi} \otimes_{A_2} \alpha_{1,\sigma}) \oplus_B (\alpha_{2,\pi} \otimes_{A_2} \alpha_{2,\sigma})
    \]
    is a unitary, which concludes the proof. 
    Let $V_0 \colon \bigoplus_{\pi \in \Obj \sC} \beta_\pi^{\oplus \infty} \to \sH_B$ be a unitary. By the surjectivity of $\phi_i$, one may lift it to unitaries $V_i \colon \bigoplus_\pi \alpha_{i,\pi}^{\infty} \to \sH_{A_i}$ such that $V_i \otimes_{\phi_i} 1 = V_0 \circ \diag (\bbmv_{i,\pi})_{\pi}$.  
    Let $\psi_0 \colon B \to \cL(\sH_{B})$ and $\psi_i \colon A_i \to \cL(\sH_{A_i})$ be the left actions imposed through the identifications $V_i$ ($i=0,1,2$). Then we have $\psi_i(a_i) \otimes_{\phi_i} 1 = \phi_i(a_i) \otimes_{\psi_0} 1$ for $a_i \in A_i$.
    Let $p_{i,\pi}$ denote the projection onto the first copy of $\alpha_{i,\pi}$ in $\bigoplus_\pi \alpha_{i,\pi}^{\infty} \cong \sH_{A_i}$. Then both $p_{1,\pi} \otimes_{\phi_1} 1 = p_{2,\pi } \otimes_{\phi_2} 1$ are the support projections of the first copy of $\beta_{\pi}$ in $\bigoplus_\pi \beta_{\pi}^\infty$. 
    Thus we get
        \[ 
        \alpha_{i,\pi} \otimes_{A_i} \alpha_{i,\sigma} = (p_{i,\pi} \sH_{A_i}) \otimes_{A_i} (p_{i,\sigma}\sH_{A_i}) \cong (p_{i,\pi} \otimes_{\psi_i} 1) \cdot (1_\sH \otimes p_{i,\sigma}) \cdot (\sH \otimes \sH \otimes A_i). 
        \]  
    Therefore, by setting $P_{i,\pi,\sigma}:=(p_{i,\pi} \otimes_{\psi_i} 1) \cdot (1 \otimes p_{i,\sigma}) $, we conclude that
    \begin{align*}
        & (\alpha_{1,\pi} \otimes_{A_1} \alpha_{1,\sigma}) \oplus_B (\alpha_{2,\pi} \otimes_{A_2} \alpha_{2,\sigma})\\ 
        \cong {} & P_{1,\pi,\sigma} (\sH \otimes \sH \otimes A_1) \oplus_B P_{2,\pi,\sigma}( \sH \otimes \sH \otimes A_2) \\
        \cong{} & (P_{1,\pi,\sigma}, P_{2,\pi,\sigma}) \cdot (\sH \otimes \sH \otimes (A_1 \oplus_B A_2)) \\
        ={}& ( (p_{1,\pi}, p_{2,\pi}) \otimes _{(\psi_1,\psi_2)} 1) \cdot (1 \otimes (p_{1,\sigma},p_{2,\sigma} )) \cdot (\sH \otimes \sH \otimes (A_1 \oplus_B A_2)) \\
        \cong {}& (p_{1,\pi}, p_{2,\pi}) \sH_{A_1 \oplus_B A_2} \otimes_{A_1 \oplus_B A_2} (p_{1,\sigma}, p_{2,\sigma}) \sH_{A_1 \oplus_B A_2}\\
        \cong {}& (\alpha_{1,\pi} \oplus_B \alpha_{2,\pi}) \otimes _{A_1 \oplus_B A_2} (\alpha_{1, \sigma} \oplus_B \alpha_{2,\sigma}). \qedhere
    \end{align*}
    \end{proof}

	\begin{ex}\label{ex:mapping.cone}
		Let $(\phi,\bbmv) \colon A \to B$ be a cocycle $\sC$-$\ast$-homomorphism. The mapping cone $\cone (\phi,\bbmv) $ is defined by the fibered sum $\cone (B) \oplus_B A$. Since $\cone (B) \to B$ is surjective, $\cone(\phi,\bbmv)$ is equipped with the $\sC$-action. Note that the mapping cone sequence $0 \to S B \to \cone(\phi,\bbmv) \to A \to 0$ is an exact sequence of $\sC$-C*-algebras. Similarly, we can also consider the mapping cylinder of $(\phi,\bbmv)$. 
	\end{ex}

	\subsection{Morita equivalence}\label{subsubsection:Morita}
	\begin{prop}\label{lem:cocycleMoritainverse}
		Let $(A,\alpha,\fu)$ and $(B,\beta,\fv)$ be $\sC$-C*-algebras. 
		The following conditions are equivalent; 
		\begin{enumerate}
			\item there is a $\sC$-equivariant categorical equivalence $\Mod(A) \simeq \Mod(B)$, 
			\item there is $\mathsf{E} \in \Bimod^\sC(A,B)$ and $\mathsf{E}^* \in \Bimod^\sC(B,A)$ such that $\mathsf{E}^* \otimes_A \mathsf{E} \cong B$ and $\mathsf{E} \otimes_B \mathsf{E}^* \cong A$,
			\item there is a $\sC$-Hilbert $A$-$B$ bimodule $\mathsf{E}=(E,\phi,\bbmv )$ such that 
			\begin{itemize}
				\item the underlying Hilbert $B$-module $E$ is full, i.e., $\langle E,E \rangle \subset B$ is dense, and 
				\item $(\phi , \bbmv)$ induces a $\ast$-isomorphism $A \to \cK(E)$.
			\end{itemize} 
		\end{enumerate}
	\end{prop}
	
	We say that $(A,\alpha,\fu)$ and $(B,\beta,\fv)$ are \emph{$\sC$-equivariantly Morita equivalent} if one of the above conditions (1), (2), (3) is satisfied. In this case, the bimodule $\mathsf{E}$ is called the ($\sC$-equivariant) imprimitivity $A$-$B$ bimodule. 
	\begin{proof}
		The equvialence (1)$\Leftrightarrow$(2) is obvious from \cref{prop:functor_bimodule}. (2)$\Rightarrow$(3) is also clear from the known fact in non-equivariant Morita theory. We show (3)$\Rightarrow$(2), by defining the $\sC$-action onto the $B$-$A$ bimodule $E^*:= \cK(E, B)$, endowed with the Hilbert $A$-module structure given by $Ta:=T \cdot \phi(a)$ and $\langle T, S \rangle :=\phi^{-1}(T^*S)$. 
		To this end, we define the cocycles $\bbmv_\pi^* \colon \beta_\pi \otimes_B E^* \to E^* \otimes_A \alpha_\pi$ as the composition
		\begin{align*}
			\beta_\pi \otimes_B E^* \cong (E^* \otimes_A E) \otimes_B \beta_\pi \otimes_B E^* \xrightarrow{(1_E \otimes \bbmv_\pi^*) \otimes 1_{E^*}} E^* \otimes_A \alpha_\pi \otimes E \otimes_B E^* \cong  E^* \otimes_A \alpha_\pi.  
		\end{align*}  
		It is straightforward to check that $\bbmv^*$ makes the diagrams in \cref{defn:cocycle.hom} commutative. 
		Moreover, the multiplication $\mathsf{E}^* \otimes_A \mathsf{E} \to B$ and the isomorphism $\phi \colon A \to \cK(E) \cong \sfE \otimes_B \sfE^*$ are both unitary intertwiners of $\sC$-Hilbert bimodules. 
	\end{proof}

    \begin{ex}
    Let $A$, $B$ be endomorphism $\sC$-C*-algebras. 
    Let $(\iota, \mathbbm{1}) \colon A \to B$ be a $\sC$-$\ast$-homomorphism $(\iota, \mathbbm{1})$ such that $\iota$ is a full-corner embedding onto $pBp$ for some projection $p \in \cM(B)$. 
    Then, the $A$-$B$ imprimitivity bimodule $E:=pB$ is made $\sC$-equivariant by the trivial cocycle ${}_{\alpha_\pi}A \otimes_{\iota} pB \to pB \otimes_B ({}_{\beta_\pi}B)$. By \cref{lem:cocycleMoritainverse} (3), this gives a $\sC$-Morita equivalence. 
    \end{ex}

    Any proper essential $\sC$-Hilbert bimodule is decomposed into a cocycle $\sC$-$\ast$-homomorphism (with trivial cocycles) and the `inverse' of a Morita equivalence in the following sense.
	\begin{lem}\label{lem:linking.trivcocycle}
	Let $A$, $B$ be endomorphism $\sC$-C*-algebras and $\sfE=(E,\phi,\bbmv)$ be a proper essential $\sC$-Hilbert $A$-$B$-bimodule. 
	Then, there is an endomorphism $\sC$-C*-algebra $D$, a $\sC$-$\ast$-homomorphisms $(\phi , \mathbbm{1}) \colon A\to D$, and a full-corner embedding $(\iota , \mathbbm{1}) \colon B\to D$ such that $\sfE \otimes_B D  \cong A\otimes_\phi D$ as $\sC$-Hilbert $A$-$D$ bimodules.
\end{lem}

\begin{proof}
    We write $\bbmv_\pi^{\rm t} \colon E^*\to E^*$ as the transposition of $E \ni \xi \mapsto \bbmv^*_{\pi} (\xi\otimes_{\beta_\pi}1_{B}) \in A\otimes_{\alpha_\pi}E=E$.
    Set 
    \[
	    (D,\delta,\fw)=\bigg( (\cK(E\oplus B), 
	    \begin{pmatrix} \Ad \bbmv_\pi^* (\blank\otimes_{\beta_\pi}\id_B)&\bbmv_\pi^*(\blank \otimes_{\beta_\pi}1_B)\\\bbmv_\pi^{\rm t}&\beta_\pi \end{pmatrix}, 
	    \begin{pmatrix} \phi(\fu_{\pi, \sigma})&0\\0&\fv_{\pi, \sigma} \end{pmatrix}  \bigg) .
	\] 
	It is routine to check $D$ is well-defined as an endomorphism $\sC$-C*-algebra. 
	By definition, $\phi\colon A\to D$ and $B\subset D$ are $\sC$-$*$-homomorphisms with trivial cocycles. Note that $\overline{BD}=(E^*\oplus B, m, {}_{\beta_\pi}(E^*\oplus B)\subset {}_{\delta_\pi}D )$ and $\overline{\phi(A)D}=(\cK(E)\oplus E, \phi, {}_{\phi\alpha_\pi}(\cK(E)\oplus E)\subset {}_{\delta_\pi\phi}D )$, where $m$ denotes the multiplication inside $D$. Thus 
	$\sfE\otimes_BD = \bigl( E\otimes_B(E^*\oplus B), \phi\otimes_B1_{E^*\oplus B}, \bbmv\otimes_B1_{E^*\oplus B} \bigr) \cong A\otimes_\phi D$ by $m$ since the following diagram is commutative, 
	\begin{align*}
		&
		\xymatrix{
			{}_{\phi\alpha_\pi} E\otimes_B (E^*\oplus B)
			\ar[r]^-{\bbmv_\pi\otimes1}\ar[d]_-{1_{\alpha_\pi}\otimes m}& 
			{}_{\phi}E\otimes_{\beta_\pi} (E^*\oplus B)
			\ar[r]^-{=}& {}_\phi E\otimes_B D\otimes_{\delta_\pi}D\ar[d]^-{m\otimes1_{\delta_\pi}} \\
			{}_{\phi\alpha_\pi}(\cK(E)\oplus E)\ar[r]^-{\subset}&{}_{\delta_\pi\phi}D&{}_\phi (\cK(E)\oplus E)\otimes_{\delta_\pi}D\ar[l]_-{\supset},
		}
	\end{align*}
	where the composition ${}_\phi E\otimes_{\beta_\pi}(E^*\oplus B)\to {}_{\delta_\pi\phi}D$ of the right arrows is given by $m(\delta_\pi|_E\otimes \id_D)=m(\bbmv_\pi^*\otimes \id_D)$. 
\end{proof}

	\subsection{External tensor product}\label{subsubsection:external.tensor} 
        In this subsection, $\bullet$ stands for either $\min$ or $\max$. 
        
        For nonunital C*-tensor categories $\sA$ and $\sB$, their minimal or maximal Deligne tensor products $\sA \boxtimes_{\bullet} \sB$ are defined (cf.~\cite{antounBicolimitsCategories2020}*{Subsection 3.1}), which are again nonunital C*-tensor categories by \cite{antounBicolimitsCategories2020}*{Proposition 3.2}. 
	Let $\sC$ and $\sD$ be rigid C*-tensor categories with countable objects, and let $\sA$ and $\sB$ be a nonunital $\sC$-module category and a nonunital $\sD$-module category respectively. 
    Their tensor product $\sA \boxtimes_{\bullet} \sB$ is equipped with the action of the Deligne tensor product $\sC \boxtimes \sD$ by $(X \otimes Y) \otimes_{\alpha \otimes \beta}(\pi \otimes \sigma):=(X \otimes_\alpha \pi) \otimes (Y \otimes_\beta \sigma)$. 
    
	Equivalently, in terms of $\sC$-C*-algebras, the external tensor product of $\sC$-C*-algebras $(A, \alpha, \fu)$ and $(B, \beta, \fv)$ is defined by 
	\[ A \otimes_{\bullet} B:= (A \otimes_{\bullet} B, \{ \alpha _\pi \otimes \beta_\sigma \}_{(\pi, \sigma) \in\Obj (\sC \boxtimes \sD)} , \{\fu_{\pi, \pi'} \otimes \fv_{\sigma, \sigma'}\}_{(\pi, \sigma),(\pi', \sigma') \in \Obj (\sC \boxtimes  \sD)} ). \]
    This construction is related to the above paragraph through the isomorphism $\Mod (A) \boxtimes_\bullet \Mod (B) \cong \Mod (A \otimes_\bullet B)$ (\cite{antounBicolimitsCategories2020}*{Proposition 3.5}).

	Let $\sA$, $\sB$ be $\sC$-module categories and let $\sA'$, $\sB'$ be $\sD$-module categories. 
	Then, for equivariant functors $(\cE, \mathsf{v}) \colon \sA \to \sB$ and $(\cF, \mathsf{w}) \colon \sA' \to \sB'$, their external tensor product 
	\[(\cE \boxtimes_\bullet \cF, \mathsf{v} \boxtimes \mathsf{w} ) \colon \sA \boxtimes_\bullet  \sA' \to \sB \boxtimes_\bullet  \sB' \]
	is defined by $(\cE \boxtimes_\bullet \cF)(X \boxtimes_\bullet  Y):=\cE(X) \boxtimes_\bullet \cF(Y)$ and $(\mathsf{v} \boxtimes_\bullet \mathsf{w})_{X \boxtimes_\bullet Y, \pi \boxtimes_\bullet \sigma}:=\mathsf{v}_{X,\pi} \boxtimes_\bullet \mathsf{w}_{Y,\sigma}$.
	Its C*-algebraic counterpart is as follows. 
	Let $A$, $B$ be $\sC$-C*-algebras $A$ and $B$ an $A'$, $B'$ $\sD$-C*-algebras. Then, for $\mathsf{E}_1\in\Bimod^\sC(A,B)$ and $\mathsf{E}_2\in\Bimod^\sD(A',B')$, their external tensor product is a $\sC \boxtimes \sD$-equivariant $A \otimes_\bullet A'$-$B \otimes_\bullet B'$ bimodule defined as
	\[ \mathsf{E}_1 \otimes_\bullet \mathsf{E}_2:=(E_1\otimes_{\bullet} E_2, \phi_1 \otimes \phi_2, \bbmv_1\otimes\bbmv_2). \]
	It is a routine work to check that $\mathsf{E}_1 \otimes_\bullet \mathsf{E}_2$ satisfies the conditions of tensor category equivariant Hilbert bimodule. 
	
	\subsection{\texorpdfstring{$\bZ_2$}{Z2}-grading}\label{subsubsection:graded}
	The notions concerned with $\bZ_2$-grading are described only in terms of $\sC$-C*-algebra. 
	Let us recall that, for $\bZ_2$-graded C*-algebras $A$ and $B$, $\bZ_2$-graded Hilbert $A$-module $E_1$ and $B$-module $E_2$, and a graded $\ast$-homomorphism $\phi \colon A \to \cL(E_2)$, their graded exterior product $E_1 \hotimes_A E_2$ is the Hilbert $B$-module $E_1 \otimes_A E_2$ equipped with the $\bZ_2$-grading $\xi \hotimes \eta \mapsto \gamma_{E_1}(\xi) \hotimes \gamma_{E_2}(\eta)$.
 	\begin{df}
		Let $A$, $B$ be $\sC$-C*-algebras. 
		\begin{enumerate}
			\item A $\sC$-C*-algebra $(A, \alpha, \fu)$ is \emph{$\bZ_2$-graded} if $A$ is equipped with a $\bZ_2$-grading, i.e., a $\ast$-automorphism $\gamma_A \colon A \to A$ with $\gamma_A^2=\id$, and each $\alpha_\pi$ is equipped with a $\bZ_2$-grading $\gamma_{\alpha_\pi}$ making it a $\bZ_2$-graded Hilbert $A$-$A$ bimodule (i.e., $\gamma_{\alpha_\pi}(a\xi b) = \gamma_A(a)\gamma_{\alpha_\pi} (\xi) \gamma_A(b)$ and $\langle \gamma_{\alpha_\pi}(\xi),\gamma_{\alpha_\pi}(\eta) \rangle  = \gamma_A(\langle \xi,\eta \rangle) $ hold) and the morphisms $\alpha_f$ for $f \in \Hom_\sC(\pi,\sigma)$ and $\fu_{\pi,\sigma} \colon \alpha_\pi \hotimes \alpha_\sigma \to \alpha_{\pi \otimes \sigma}$ are even. 
			\item For $\bZ_2$-graded $\sC$-C*-algebras $A$ and $B$, a $\sC$-Hilbert $A$-$B$ bimodule $\sfE:=(E,\phi, \bbmv)$ is graded if $E$ is equipped with a $\bZ_2$-grading $\gamma_E$ such that $\phi$ is even (i.e., $\gamma_E \phi(\gamma_A(a)) = \phi(a)\gamma_E$) and each $\bbmv _\pi \colon \alpha_\pi \hotimes_\phi E \to E \hotimes_\phi \beta_\pi$ are even.
		\end{enumerate}
	\end{df}
	To these $\bZ_2$-gradings, the eigenspace decompositions $A = A^{0} \oplus A^1$, $\alpha_\pi = \alpha_\pi^0 \oplus \alpha_\pi^1$, and $E=E^0 \oplus E^1$ are associated. An element of $A^i$, $\alpha_\pi^i$ or $E^i$ for $i=0,1$ is said to be homogeneous. The induced $\bZ_2$-grading $T \mapsto \gamma_E T \gamma_E$ on $\cL(E)$ and $\cK(E)$ are denoted by $\gamma_{\cL(E)}$ and $\gamma_{\cK(E)}$.
	
    \begin{rem}
        A $\bZ_2$-grading of a $\sC$-C*-algebra $(A, \alpha, \fu)$ is equivalent to a cocycle $\sC$-$\ast$-homomorphism $(\gamma, \mathsf{g}) \colon A \to A$ such that $(\gamma , \mathsf{g})^2=(\id_A , \mathbf{1})$. Indeed, $\mathsf{g}_\pi$ and $\gamma_{\alpha_\pi}$ are identified through the canonical isomorphisms $\alpha_\pi \otimes_A ({}_\gamma A) \cong \alpha_\pi$ and $({}_\gamma A) \otimes_A \alpha_\pi \cong \alpha_\pi$. 
        Similarly, a $\bZ_2$-grading of an essential $\sC$-Hilbert $A$-$B$ bimodule $(E,\phi,\bbmv)$ is equivalent to an intertwiner $(\gamma_A, \mathsf{g}_A) \otimes_A \sfE \to \sfE \otimes_B  (\gamma_B , \mathsf{g}_B )$. 
    \end{rem}

    For a $\bZ_2$-graded $\sC$-Hilbert $A$-$B$ bimodule $\sfE_1:=(E_1,\phi_1,\bbmv_1)$ and a $\bZ_2$-graded $\sC$-Hilbert $B$-$D$ bimodule $\sfE_2:=(E_2,\phi_2,\bbmv_2)$, their composition 
    \[ \sfE_1 \hotimes_B \sfE_2:=(E_1 \hotimes_{\phi_2} E_2, \phi_1 \hotimes 1, (1 \hotimes \bbmv_2)(\bbmv_1 \hotimes 1) )\]
    is again $\bZ_2$-graded. Here, isometric maps $\bbmv_{1,\pi} \hotimes 1$ and $1 \hotimes \bbmv_{2,\pi}$ are defined to be $\bbmv_{1,\pi} \otimes 1$ and $1 \otimes \bbmv_{2,\pi}$, which are both even by definition. 

    Recall that the (maximal or minimal) graded tensor product $A \hotimes B$ of $\bZ_2$-graded C*-algebras $A$, $B$ is the C*-completion of the $A \hotimes_{\rm alg} B :=A \otimes_{\rm alg} B$ with the algebra structure $(a \hotimes a') \cdot (b \hotimes b'):=(-1)^{|a'| \cdot |b|} ab \hotimes a'b'$ for homogeneous $a,a' \in A$ and $b, b' \in B$. 
    For a $\bZ_2$-graded $\sC$-C*-algebra $(A,\alpha,\fu)$ and a $\bZ_2$-graded $\sD$-C*-algebra $(B,\beta,\fv)$, the (maximal or minimal) graded tensor product $A \hotimes B$ is equipped with the canonical structure of $\sC \boxtimes \sD$-C*-algebra by $ (\alpha \hotimes \beta )_{(\pi , \pi')} := \alpha_\pi \hotimes \beta_\pi$ and
    \begin{gather*}
        (\fu \hotimes \fv)_{(\pi , \pi'), (\sigma,\sigma')} \colon (\alpha_\pi \hotimes \beta_{\pi'}) \hotimes_{A \hotimes B} (\alpha_\sigma \hotimes \beta_{\sigma'} ) \to \alpha_{\pi \otimes \sigma} \hotimes \beta_{\pi \otimes \sigma},\\
        (\fu \hotimes \fv)_{(\pi , \pi'), (\sigma,\sigma')} ((\xi \hotimes \xi' ) \hotimes_{A \hotimes B} (\eta \hotimes \eta')):=(-1)^{|\eta|\cdot |\xi'|} \fu_{\pi,\sigma} (\xi,\eta) \hotimes \fv_{\pi',\sigma'} (\xi ',\eta') .
    \end{gather*}
    A typical and motivating example of a $\bZ_2$-graded $\sC$-C*-algebra is $(A \hotimes \bC \ell_n, \alpha \hotimes \id_{\bC \ell_n},\fu \hotimes 1)$, where $(A ,\alpha, \fu)$ be a $\sC$-C*-algebra and  $\bC \ell_n$ is the Clifford algebra. 

    A $\sC$-C*-algebra is regarded as trivially $\bZ_2$-graded, i.e.,  $\bZ_2$-graded via $\gamma_A = \id_A$ and $\gamma_{\alpha_\pi} = \id_{\alpha_\pi}$.
    When both $A$ and $B$ are trivially graded, a $\bZ_2$-graded Hilbert $A$-$B$ bimodule $(E,\phi,\bbmv)$ is nothing else than the direct sum $(E^0,\phi^0,\bbmv^0) \oplus (E^1,\phi^1,\bbmv^1)$ of two $\sC$-Hilbert $A$-$B$ bimodules. In this case, we define the oppositely graded bimodule $\sfE^{\rm op}:=(E^{\rm op},\phi^{\rm op},\bbmv^{\rm op})$ by $E^{\rm op}:=E$, $\phi^{\rm op}:=\phi$, $\bbmv_\pi^{\rm op}:=\bbmv_\pi$ and $\gamma_{E^{\rm op}}:=-\gamma_E$.

	\section{Kasparov's technical theorem}
	Here we record a variant of Kasparov's technical theorem below (\cref{thm:technicalappendix}). 
	They can be proved in a completely same way as \cite{baajCalgebresHopfTheorie1989}*{Section 4}, but we include them here for the convenience of readers. 
	
	\begin{thm}[cf. {\cite{baajCalgebresHopfTheorie1989}*{Theoreme 4.3}}]\label{thm:technicalappendix}
		Let $I$ be a countable index set. Let 
		\begin{itemize}
			\item $\cJ$ be a $\bZ_2$-graded C*-algebra, 
			\item $\{ (\cJ_i, \sigma_i)\}_{i \in I} $ be a family of $\bZ_2$-graded C*-algebras $\cJ_i$ and graded $\ast$-homomorphisms $\sigma_i \colon \cJ \to \cJ_i$ that are all essential, 
			\item $\cA_1 $ be a $\sigma$-unital $\bZ_2$-graded essential C*-subalgebra of $\cM(\cJ)$, 
			\item $\{ \cA_{2,i}\}_{i \in I} $ be a family of $\sigma$-unital $\bZ_2$-graded C*-subalgebra of $\cM(\cJ_i)$ such that $\sigma_i(\cA_1) \cdot \cA_{2,i} \subset \cJ_i$, and 
			\item $\{ \Delta _i\}_{i \in I}$ be a family of separable $\bZ_2$-graded closed linear subspace of $\cM(\cJ_i)$ such that $[\sigma_i(\cA_1), \Delta_i] \subset  \overline{\sigma_i(\cA_{1})\cM(\cJ_i)\sigma_i(\cA_{1})}$.
		\end{itemize}
		Then there is an even operator $M \in \cM(\cJ)$ such that $0 \leq M \leq 1$ and 
		\[
		M \cA_1 \subset \cJ, \quad (1-\sigma_i(M))\cA_{2,i} \subset \cJ_i, \quad [\sigma_i(M), \Delta_i] \subset \cJ_i
		\]
		holds for any $i \in I$. 
	\end{thm}
	Note that this is the same thing as the original Kasparov technical theorem in the case that $I=\{ 0 \}$, $\cJ_0 = \cJ$, and $\sigma_0 = \id$. 
	For $\bZ_2$-graded C*-algebras $\cB$, $\cD$ and an essential graded $\ast$-homomorphism $\tau \colon \cB \to \cD$, we write $\mathrm{Der}(\cB,\cD;\tau)$ for the set of bounded linear maps $\delta \colon \cB \to \cD$ such that 
	\[\delta (ab)=\delta (a)\tau(b) + (-1)^{|\delta| \cdot |a|} \tau(a)\delta (b) \]
	holds. 
        Note that the inner derivation $\ad (\blank) \circ \tau \colon \cM(\cD) \to \mathrm{Der}(\cB,\cD ;\tau)$ is continuous in the norm topology. 
	
	\begin{lem}[cf. {\cite{baajCalgebresHopfTheorie1989}*{Lemma 4.1}}]\label{lem:sublemma.technical}
		Let $\cB$, $\cD$ be $\bZ_2$-graded C*-algebras and $\tau \colon \cB \to \cD$ be an essential graded  $\ast$-homomorphism. 
		Let $K \subset \Der (\cB, \cD; \tau) $ be a compact subset in the norm topology. 
		Then, for any $\varepsilon > 0$, $ b \in \cB$, $d \in \cD$, 
		and $e_0\in \cB$ with $0\leq e_0$ and $\| e_0 \|<1$, 
		there exists $e\in \cB$ such that 
		\begin{itemize}
			\item $e_0\leq e$, $\|e\|<1$, and $e$ is even,  
			\item $\|b-eb\|, \|d - \tau(e)d\|, \| \delta (e) \| < \varepsilon $ for any $ \delta \in K$. 
		\end{itemize}
	\end{lem}
	
	\begin{proof}
		Consider the bounded linear map 
		\begin{align*}
			\Phi_{b,d,K} \colon \cB \to \cB \oplus \cD \oplus  C(K ; \cD ), \quad 			
			\Phi_{b,d,K}(x):=(b-xb,d - \tau(x)d, \delta \mapsto \delta (x) ).  
		\end{align*}
		This uniquely extends to the map $\Phi_{b,d,K} \colon\cM(\cB ) \to \cM(\cB \oplus \cD \oplus C(K;\cD))$ that is strictly continuous on the unit ball. Indeed, for $T \in \cM(\cB)$, its image $\delta(T) \in \cM(\cD)$ characterized by the relation
		\[ \delta (T)\tau (x)=\delta(T x) - (-1)^{|\delta | \cdot |T|} \tau(T) \delta(x) \]
		varies strict continuously with $\delta \in K$.
		
		Since the ordered set $E:=\{e\in \cB^0 \mid  e_0\leq e, \| e\| <1 \}$
		gives an approximate unit of $\cB$ and $\Phi_{b,d,K} (1_{\cM(\cB)}) = 0$, 
		we see that $ 0 $ and $\Phi(E)$ is not separated by any bounded linear functional on $\cB\oplus\cD \oplus C(K;\cD)$. 
		Now, by the convexity of $E$, we get  
		$0\in \overline{\Phi_{b,d,K}(E)}{}^{\rm weak} =\overline{\Phi_{b,d,K}(E)}{}^{\rm norm}$ 
		by the Hahn--Banach separation theorem. 
	\end{proof}

	\begin{proof}[Proof of \cref{thm:technicalappendix}]
		For $i \in I$, let 
		\[ \cD_i:=C^*(\{ \sigma_i(x), d\sigma_i(x) \mid d \in \Delta_i, x \in \cA_1 \}) \subset \cM(\cJ_i).\]
		Then, by the assumption $[\sigma_i(\cA_1), \Delta_i] \subset  \overline{\sigma_i(\cA_{1}) \cM(\cJ_i) \sigma_i(\cA_{1})}$, each $\sigma_i \colon \cA_1 \to \cD_i$ is essential. 
		Take $\bfK \subset \cM(\bigoplus_i \cJ_i)$ as a norm-compact subset of the unit ball with $ \overline{\bigoplus_i \Delta_i} = \cspan \bfK$. 
		Fix strictly positive elements $x \in \cA_1$, and $y=(y_i)_{i} \in \bigoplus_i \cA_{2,i}$ with the norms less than $1$. 
		
		Take $(I_n)_{n=1}^{\infty}$ as a family of monotonically non-decreasing finite subsets of $I$ such that $\bigcup_{n=1}^{\infty}I_n=I$.
		Starting from $u_0:=0\in \cA_1$, take $u_n\in \cA_1$ for $n\geq 1$ inductively in the way that 
		\begin{itemize}
			\item
			$u_{n-1}\leq u_n$, $\|u_n\|<1$, \ $u_n$ is even, and  
			\item
			$\| x- u_nx\|, \| [d_i, \sigma_i (u_n) ] \| < 2^{-n} $ for any $d=(d_i)_i \in \bfK$ and $i \in I_n$,
		\end{itemize} 
		by applying \cref{lem:sublemma.technical} for 
		\[ \cB=\cA_1, \; \cD:=\bigoplus_{i \in I_n} \cD_i, \; \tau:= \prod_{i \in I_n} \sigma_i, \; K=\ad (\bfK) \circ \tau, \;  b=x, \; d=0, \;\varepsilon=2^{-n}, \;  e_0=u_{n-1}.\] 
		Similarly, starting from $v_0=0$, we also take $v_n \in \cJ$ for $n \geq 1$ inductively in the way that 
		\begin{itemize}
			\item
			$0\leq v_n$,\ $\| v_n\|<1$,\ $v_n$ is even, and 
			\item
			$\| (1 - \sigma_i(v_n))\sigma_i(u_n-u_{n-1})y_i \|, \| [d_i, \sigma_i(v_n)]\|, \| [\sqrt{u_n-u_{n-1}},v_n] \| <2^{-n}$ for any $d = (d_i)_i\in \bfK$ and $i\in I_n$,
		\end{itemize} 
		by applying \cref{lem:sublemma.technical} for 
		\begin{gather*}
			\cB=\cJ, \; \cD= \cJ \oplus \bigoplus\limits_{i\in I_n} \cJ_i, \; \tau:=\id_\cJ \oplus  \prod\limits_{i\in I_n} \sigma_i, \; K= \ad (\sqrt{u_n-u_{n-1}},\bfK) \circ \tau, \\
			b=0, \; d=\tau (u_n-u_{n-1}) \cdot (0,y), \; \varepsilon = 2^{-n}, \; e_0=v_{n-1}. \;      
		\end{gather*}

		Set 
		\[ M_n:=\sum\limits_{k=1}^{n}\sqrt{u_k-u_{k-1}}v_k\sqrt{u_k-u_{k-1}} \in \cJ, \quad \text{and } M:=\lim _{n \to \infty} M_n\in \cM(\cJ),\]
		where the limit is taken in the strict topology of $\cM(\cA_1)$. 
		The strict limit converges since $\| M_n \| \leq 1$ and $\| \sqrt{u_n-u_{n-1}}v_n\sqrt{u_n-u_{n-1}}xa \| \leq 5\| a\|2^{-n}$ for any $a \in \cA_1$.
		Let $M_n':= \sum _{k=1}^n v_k (u_k-u_{k-1})$. Then $M_n - M_n' = \sum [\sqrt{u_k-u_{k-1}} , v_k] \sqrt{u_k - u_{k-1}} \in \cJ$ converges in the norm topology.
		For $a \in \cA_1$, $z_i \in \cA_{2,i}$ and $d_i \in \Delta_i$, the sequences
		\begin{itemize}
			\item $M_n a = M_n'a + (M_n-M_n')a \in \cJ$,
			\item $(\sigma_i(u_n)-\sigma_i(M_n))z_i = (\sigma_i(u_n)-\sigma_i(M_n'))z_i + \sigma_i(M_n' - M_n)z_i \in \cJ_i$, and 
			\item $[d_i, \sigma_i(M_n) ]  = [d_i, \sigma_i(M_n')] + [d_i, \sigma_i(M_n-M_n')] \in \cJ_i$
		\end{itemize}
		converge in the norm topology to $Ma$, $(1-\sigma_i(M))z_i$, and $[d_i, \sigma_i(M)]$ respectively. This concludes the proof. 
	\end{proof}


\end{document}